%% file: paper_elsart.tex
\documentclass[times,sort&compress,3p]{elsarticle}
\usepackage{graphicx}
\usepackage[utf8]{inputenc}
\usepackage{color}
\usepackage{tikz}
\usepackage{pgfplots}
\pgfplotsset{width=8.5cm,compat=1.9}
\usetikzlibrary{intersections,shapes.arrows}

\usepackage{microtype}
\usepackage{amsmath}
\usepackage{dsfont}
\usepackage{amsfonts}
\usepackage{xr-hyper}
\usepackage{hyperref}
\usepackage{amssymb}
\usepackage{amsthm}
\usepackage{epstopdf}
\usepackage{enumerate}
\usepackage{paralist}
\usepackage{booktabs}
\usepackage{color}
 \usepackage{xcolor,colortbl}
 \usepackage{comment}
 \usepackage{wasysym}
\usepackage[linesnumbered,ruled,vlined]{algorithm2e}
 \SetKwInput{KwInput}{Input}                %
\SetKwInput{KwOutput}{Output}
\SetKwRepeat{Do}{do}{while}

\graphicspath{{./art/}}

\definecolor{royalblue}{RGB}{0,78,156}
\definecolor{darkraspberry}{rgb}{0.53, 0.15, 0.34}
\definecolor{aurometalsaurus}{rgb}{0.43, 0.5, 0.5}
\definecolor{britishracinggreen}{rgb}{0.0, 0.26, 0.15}

\input{preamble}

\makeatletter
\newcommand*{\addFileDependency}[1]{%
  \typeout{(#1)}
  \@addtofilelist{#1}
  \IfFileExists{#1}{}{\typeout{No file #1.}}
}
\makeatother

\newcommand*{\myexternaldocument}[1]{%
    \externaldocument{#1}%
    \addFileDependency{#1.tex}%
    \addFileDependency{#1.aux}%
}

\myexternaldocument{paper_elsart_supp_mat}

\begin{document}
	
\begin{frontmatter}
	
\title{Measuring dependence between random vectors via optimal transport}
	
\author[1]{Gilles Mordant\corref{cor1}}
\ead{mordantgilles@gmail.com}
	
\author{Johan Segers\corref{cor2}}
\ead{johan.segers@uclouvain.be}
	
\cortext[cor1]{Corresponding author}
\cortext[cor2]{J. Segers gratefully acknowledges funding by FNRS-F.R.S. grant CDR J.0146.19}	
\address[1]{ Universit\"at G\"ottingen, IMS, Goldschmidtstra\ss e 7, 37077 G\"ottingen, Germany}
\address[2]{LIDAM/ISBA, UCLouvain, Voie du Roman Pays 20, 1348 Louvain-la-Neuve, Belgium}

\begin{abstract}
To quantify the dependence between two random vectors of possibly different dimensions, we propose to rely on the properties of the 2-Wasserstein distance. We first propose two coefficients that are based on the Wasserstein distance between the actual distribution and a reference distribution with independent components. The coefficients are normalized to take values between 0 and 1, where 1 represents the maximal amount of dependence possible given the two multivariate margins.
We then make a quasi-Gaussian assumption that yields two additional coefficients rooted in the same ideas as the first two. These different coefficients are more amenable for distributional results and admit attractive formulas in terms of the joint covariance or correlation matrix. Furthermore, maximal dependence is proved to occur at the covariance matrix with minimal von Neumann entropy given the covariance matrices of the two multivariate margins. This result also helps us revisit the RV coefficient by proposing a sharper normalisation. 
The two coefficients based on the quasi-Gaussian approach can be estimated easily via the empirical covariance matrix. The estimators are asymptotically normal and their asymptotic variances are explicit functions of the covariance matrix, which can thus be estimated consistently too. The results extend to the Gaussian copula case, in which case the estimators are rank-based. The results are illustrated through theoretical examples.  Monte Carlo simulations and a case study involving electroencephalography data are proposed in the supplementary material.
\end{abstract}	

\begin{keyword}
 	Bures-Wasserstein distance \sep
	Copula \sep
	Delta method \sep
	Normal scores rank correlation \sep
	RV coefficient \sep
\end{keyword}

\end{frontmatter}

\section{Introduction}

\input{paper-intro}

\section{Wasserstein dependence coefficients}
\label{sec:dep}
\input{paper-dep}

\section{A quasi-Gaussian approach}
\label{sec:Gauss}
\input{paper-Gauss2}

\section{Estimation of quasi-Gaussian Wasserstein dependence coefficients}
\label{sec:estim}
\input{paper-estim2}

\subsection{Additional lemmas}%
\label{sec:estim:lemmas}

\input{paper-app-lemmas}

\section{Discussion}
\label{sec:disc}

\input{paper-disc}

\appendix

\section{Simulation experiments}
\label{sec:simu}
\input{paper-simu}

\section{Case study: EEG data}
\label{sec:EEG}
\input{paper-EEG}

\section{Formulas for dependence coefficients in parametric models}
\label{sec: GausForms}
We now present some closed-form formulas for some of the coefficients presented in the examples in Section~\ref{sec:Gauss:ex}. 
In Example~\ref{ex:trivequi}, as the eigenvalues of $\Sigma$ are $1+2\rho$, $1-\rho$ and $1-\rho$, we get, after some simplifications,
\begin{align*}
	\dep_1(\Sigma) 
	&=
	\frac%
	{1 + \sqrt{1+\rho} - \sqrt{1+2\rho} - \sqrt{1-\rho}}%
	{1 + \sqrt{1 + |\rho|} - \sqrt{2+|\rho|}}.
\end{align*}
For the second coefficient, a more involved calculation yields
	\[
		\dep_2(\Sigma) = \frac{2 + \rho - \sqrt{\lambda_+(\rho)} - \sqrt{\lambda_-(\rho)}}{2 + \abs{\rho} - \sqrt{\rho^2+2\abs{\rho}+2}},
	\]
with $\lambda_{\pm}(\rho) = \frac{1}{2}[\rho^2 + 2\rho + 2 \pm \rho \sqrt{\rho^2 + 12\rho + 12}]$. 
The RV coefficient and its adjusted version in~\eqref{eq:RVadj} are
\begin{align*}
	\RV(\Sigma) &= \frac{2 \rho^2}{\sqrt{2(1+\rho^2)}}, &
	\overline{\RV}(\Sigma) &= \frac{2 \rho^2}{1 + |\rho|}.
\end{align*}

In Example~\ref{ex:comp}, for the trivariate autoregressive matrix,
one has 
\[
\RV (\Sigma, 1) = \frac{\rho^4+\rho^2}{\sqrt{2(1+\rho^2)}}  \quad \text{and} \quad \overline{\RV} (\Sigma, 1) =\frac{\rho^4+\rho^2}{1+\lvert \rho \rvert},
\]
while 
\[
\dep_1 (\Sigma, 1) =\frac{1 + \sqrt{1+\rho} +\sqrt{1-\rho}- \sqrt{1-\rho^2} - \sqrt{\lambda_{1,+}(\rho) }- \sqrt{\lambda_{1,-}(\rho) } }{
	1 + \sqrt{1 + \lvert \rho \rvert}- \sqrt{2+ \lvert \rho \rvert}}
\]
and 
\[
\dep_2(\Sigma, 1)= \frac{ 3- \sqrt{1 - \rho^2} -
 \sqrt{\lambda_{2,+}(\rho)}-
\sqrt( \lambda_{2,-}(\rho)}
{2 + \lvert \rho \rvert- \sqrt{2+ 2\lvert \rho \rvert+\rho^2}}
\]
with $\lambda_{1,\pm}(\rho)= \rho^2/2 \pm \rho \sqrt{\rho^2+8}/2 +1$ and  $\lambda_{1,\pm}(\rho)= 3 \rho^2/2 \pm (\sqrt{5} \rho \sqrt{\rho^2 + 4})/2 + 1$.
For the trivariate moving average matrix, it holds that
\[
\RV (\Sigma, 1) = \frac{\rho^2}{\sqrt{2(1+\rho^2)}}  \quad \text{and} \quad \overline{\RV} (\Sigma, 1) =\frac{\rho^2}{1+\lvert \rho \rvert},
\]
while 
\[
\dep_1 (\Sigma, 1) = \frac{ \sqrt{1+\rho} +\sqrt{1-\rho} -\sqrt{1+\rho\sqrt{2}}-\sqrt{1-\rho\sqrt{2}} }%
	{1 + \sqrt{1 + \lvert \rho \rvert}- \sqrt{2+ \lvert \rho \rvert}}.
\]
In this case, the formula for $\dep_2$ is not particularly convenient and the eigendecomposition was obtained numerically.

\bibliographystyle{elsarticle-harv}
\bibliography{Dependency_upd.bib}

\end{document}

%% file: preamble.tex
\usepackage{amsthm}
\newtheorem{prop}{Proposition}[section]
\newtheorem{lem}[prop]{Lemma}
\newtheorem{thm}[prop]{Theorem}

\newtheorem{cor}[prop]{Corollary}
\newtheorem{dfn}[prop]{Definition}

\theoremstyle{definition}
\newtheorem{ex}[prop]{Example}

\theoremstyle{remark}
\newtheorem{rmk}[prop]{Remark}

\providecommand{\abs}[1]{\left\lvert{#1}\right\rvert}
\providecommand{\trace}{\operatorname{tr}}
\providecommand{\diag}{\operatorname{diag}}
\providecommand{\normal}{\mathcal{N}}
\providecommand{\reals}{\mathbb{R}}
\newcommand{\Rd}{\reals^d}
\providecommand{\oh}{\mathrm{o}}
\providecommand{\Oh}{\mathrm{O}}
\providecommand{\qnorm}{\Phi^{-1}}
\providecommand{\1}{\mathds{1}}

\providecommand{\dto}{\rightsquigarrow}
\providecommand{\dep}{\mathfrak{D}}
\providecommand{\expec}{\mathbb{E}}

\newcommand{\norm}[1]{\lVert{#1}\rVert}

\newcommand{\Prob}{\mathcal{P}}

\newcommand{\diff}{\mathrm{d}}

\providecommand{\RV}{\mathrm{RV}} %
\newcommand{\Ctr}{\mathrm{ctr}}
\newcommand{\Alc}{\mathrm{alc}}

\newcommand{\Sd}{\mathbb{S}^d} %
\newcommand{\Sdp}{\Sd_{\ge}} %
\newcommand{\Sdpp}{\Sd_{>}} %

%% file: paper-intro.tex
Measuring dependence is a fundamental problem in statistics that has applications in nearly all other domains of science. Because of this importance, it is not surprising that early in their careers, most students learn about the Pearson correlation coefficient, quantifying linear association between two univariate random variables. %
In modern days, the abundance of data makes it possible to consider groups of variables and the question of measuring dependence between two random vectors appears naturally.

\citet{hotelling:1936} proposed to address the matter by finding the linear combinations of both groups of variables that maximise the correlation coefficient. Canonical correlation analysis was born. Not much attention was devoted to the problem for decades and the next development we are aware of is the RV coefficient proposed by 
\citet{escoufier1973}. For a partitioned $d \times d$ covariance matrix
\begin{equation}
\label{eq:Sigma12}
\Sigma = 
\begin{bmatrix}
\Sigma_1 & \Psi \\ \Psi^\top &\Sigma_2
\end{bmatrix},
\end{equation}
with $d = p + q$ and with diagonal blocks $\Sigma_1$ and $\Sigma_2$ of dimensions $p \times p$ and $q \times q$, respectively,
the RV coefficient \citep{escoufier1973, robert1976unifying} is
\begin{equation}
\label{eq:RV}
	\RV(\Sigma) = \frac{\trace(\Psi \Psi^\top)}{\left(\trace(\Sigma_1^2) \trace(\Sigma_2^2)\right)^{1/2}},
\end{equation}
where $\trace(\,\cdot\,)$ is the trace operator and $(\,\cdot\,)^\top$ denotes matrix transposition.
The coefficient is based on the scalar product between certain linear operators associated to the random vectors and is the first extension of the correlation coefficient that is multivariate in nature.
Still, for given diagonal blocks $\Sigma_1$ and $\Sigma_2$ the maximal value attainable is in general smaller than one. In the course of our developments, we will propose another scaling that repairs this minor deficiency (Remark~\ref{rem:RVadj}). 

The following milestone is the work by \citet*{szekely2007measuring}, where a weighted $L_2$ distance between characteristic functions is used to construct a dependence measure. %
Since then, a renewed interest for the question of quantifying dependence between random vectors has grown.
The measure proposed by \citet*{zhu:2017} is of the same nature, involving a weighted integral of the squared covariances between indicators associated to linear combinations with varying coefficient vectors.

To test for independence between several random vectors, \citet{quessy:2010} studies a Cram\'er--von Mises statistic comparing the joint empirical copula with the product of the empirical copulas of the vectors separately. In \citet{Medovikov}, the population version of this quantity lies at the basis of a copula-based dependence measure between several random vectors.

Another line of research considered measuring dependence relying on an aggregation of vectors into variables, an approach which can be seen as extending canonical correlation analysis. 
The multivariate generalisations of Spearman's $\rho$ and Kendall's $\tau$ in \citet{grothe2014measuring} fall into this framework. 
In the same vein, \citet*{Hofert} proposed to compute the correlation between collapsing functions of groups of variables. 

Recently, \citet{puccetti2019measuring} proposed a dependence coefficient based on optimal transportation theory. Alike the RV-coefficient, it is based on traces of covariance matrices but the scaling accommodates for those that are attainable given the ones of both vectors of interest.
The coefficient cannot be used for vectors with different dimensions and is not invariant with respect to permutations of variables within a group.

Still, as we shall see, the (2-)Wasserstein distance is a particularly convenient metric on the space of probability distributions with finite (second) moments and it can be leveraged to construct new dependence coefficients. The interest of this distance for statistical inference is not new but blossomed recently. We refer to \citet{panaretos2019invitation, panaretos2019statistical} for background and surveys.

Recent developments regarding dependence coefficients include \citet{chatterjee2020coefficient} and \citet{azadkia2019simple} as well. The latter are however not directly relevant for our work. After posting the first version of the manuscript, we became aware of the works by \citet{mori2020earth,nies2021transport} and \citet{wiesel2021measuring} also measuring association based on the Wasserstein distance. The coefficient defined in the latter reference is elegant at the population level but the proposed estimator appears impractical for statistical inference. 

In this paper, we propose new dependence coefficients based on the 2-Wasserstein distance. As the asymptotic theory of the empirical Wasserstein distance is currently not yet sufficiently developed to derive the results needed for statistical inference for these coefficients, we also propose quasi-Gaussian counterparts in terms of a partitioned covariance or correlation matrix. Our approach thus shares common points with both Escoufier's RV and Puccetti's coefficients. The proper normalisation of the coefficients involves the interesting side-problem of characterising, among all partitioned covariance matrices $\Sigma$ of the form~\eqref{eq:Sigma12} with fixed diagonal blocks $\Sigma_1$ and $\Sigma_2$, the $p \times q$ cross-covariance matrix $\Psi$ that yields the strongest dependence.

We then propose plug-in estimators and prove their asymptotic normality by means of the delta method.
The asymptotic variances admit analytic formulas and can therefore be estimated by a plug-in approach too, avoiding the need for resampling procedures.
The Fréchet differentiability of the maps that send a covariance or correlation matrix to the coefficients means that the asymptotic distributions of plug-in estimators can be studied in a wide variety of settings, including time series, graphical models, and rank-based estimators.
The approach is akin to the one of estimating the Wasserstein distance between Gaussian distributions in \citet{rippl2016limit}.
In passing, our calculations shed new light on the Fréchet differentiability of the Wasserstein distance derived in that article.

Rescaling the univariate margins to the standard Gaussian distribution prior to computing the correlation matrix has two advantages: first, no moment conditions are required and second, the coefficients become invariant under component-wise increasing transformations. The proposed standardisation is particularly natural in the Gaussian copula case, a model assumption which has been gaining popularity since \citet{liu2009nonparanormal}, for instance for graphical models. We illustrate the coefficients on electroencephalogram (EEG) data modelled in this way in \citet{solea2020copula} in the supplementary material. The estimates relies on the matrix of normal scores rank correlation coefficients, asymptotic expansions of which were established in \citet{klaassen1997efficient}.

The outline of this paper is the following. In Section~\ref{sec:dep}, we propose new dependence coefficients between random vectors exploiting the properties of the Wasserstein distance. In Section~\ref{sec:Gauss}, we introduce a quasi-Gaussian version of the coefficients based on the Bures--Wasserstein distance \citep{BHATIA2019165} between certain covariance matrices. Plug-in estimators and their limiting distributions are treated in Section~\ref{sec:estim}. Section~\ref{sec:disc} concludes and paves the way for further developments.
In the supplementary material, we study the performance of the proposed estimator via Monte Carlo simulations in \ref{sec:simu} and propose an application to the already mentioned EEG data in \ref{sec:EEG}.

%% file: paper-dep.tex
Let $\Prob(\Rd)$ be the set of Borel probability measures on $\Rd$ and let $\Prob_2(\Rd) \subset \Prob(\Rd)$ be the set of such measures with finite second moments.  
For $(\pi, \pi') \in \Prob_2(\Rd)^2$, let $\Gamma(\pi, \pi')$ be the set of couplings $\gamma \in \Prob_2(\reals^{2d})$ of $\pi$ and $\pi'$, that is, probability measures $\gamma$ such that $\gamma(B \times \Rd) = \pi(B)$ and $\gamma(\Rd \times B) = \pi'(B)$ for Borel sets $B \subseteq \Rd$.
Let $W_2$ denote the $2$-Wasserstein distance on $\Prob_2(\Rd)$: its square is
\[
	W_2^2(\pi, \pi') = 
	\inf_{\gamma \in \Gamma(\pi, \pi')} 
	\int_{\reals^{2d}} \norm{v-v'}^2 \, \diff\gamma(v,v'),
	\qquad \pi, \pi' \in \Prob_2(\Rd).
\]
This defines a metric on $\Prob_2(\Rd)$, the origins of which go back to  Kantorovich; see \citet{panaretos2019statistical} for a survey and historical notes.
The infimum is attained and the corresponding $\gamma$ is called an optimal coupling between $\pi$ and $\pi'$.

For a random vector $ (X, Y)$ of dimension $d = p+q$ and with joint law $\pi \in \Prob_2(\Rd)$, we seek to quantify the dependence between the subvectors $X$ and $Y$.
Let $\mu \in \Prob_2(\reals^p)$ and $\nu \in \Prob_2(\reals^q)$ denote the distributions of $X$ and $Y$, respectively.
Note that $\pi$ belongs to $\Gamma(\mu, \nu)$, the set of couplings of $\mu$ and $\nu$.
The assumption that $\pi$ has finite second moments is not a real restriction since we can first transform its univariate margins to a suitable distribution, see Remark~\ref{rem:cop}.

To quantify the dependence between $X$ and $Y$, we compare $\pi$ to $\mu \otimes \nu$, where $\otimes$ denotes product measure---the distribution of an independent coupling.
Let $\Prob_{2,0}(\reals^r)$ be the subset of $\Prob_2(\reals^r)$ of all non-degenerate distributions.
Choose reference laws $\upsilon_1 \in \Prob_{2,0}(\reals^p)$ and $\upsilon_2 \in \Prob_{2,0}(\reals^q)$ and put
\begin{equation}
\label{eq:Td1d2}
	T_{p,q}(\pi; \upsilon_1, \upsilon_2)=
    W_2^2(\pi, \upsilon_1 \otimes \upsilon_2) 
    - W_2^2(\mu \otimes \nu, \upsilon_1 \otimes \upsilon_2) %
    =
    W_2^2(\pi, \upsilon_1 \otimes \upsilon_2) - W_2^2(\mu, \upsilon_1) - W_2^2(\nu, \upsilon_2).
\end{equation}
For the second identity, see for instance the beginning of Section~2 in \citet{panaretos2019statistical}.

\begin{lem}
\label{lem:Td1d2}
For $\pi,\mu,\nu,\upsilon_1,\upsilon_2$ as above, $T_{p,q}$ in \eqref{eq:Td1d2} satisfies the following properties:
\begin{enumerate}[(i)]
\item $T_{p,q}(\pi; \upsilon_1, \upsilon_2) \ge 0$.
\item $T_{p,q}(\mu\otimes \nu; \upsilon_1, \upsilon_2) = 0$.
\item If either $\upsilon_1 = \mu$ and $\upsilon_2 = \nu$ or if both $\upsilon_1$ and $\upsilon_2$ are absolutely continuous, then $T_{p,q}(\pi; \upsilon_1, \upsilon_2) = 0$ implies $\pi = \mu \otimes \nu$.
\end{enumerate}
\end{lem}

\begin{proof}[Proof of Lemma~\ref{lem:Td1d2}]
	(i) Let $V = (V_1, V_2)$ be a random vector with law $\upsilon_1 \otimes \upsilon_2$ and let $((X, Y),V)$ be a coupling of $(X,Y)$ and $V$.
	Then $(X, V_1)$ and $(Y,V_2)$ are couplings of $\mu$ and $\upsilon_1$ and of $\nu$ and $\upsilon_2$, respectively, and thus
	\begin{equation}
	\label{eq:EXY2:minorant}
	\expec[\norm{(X,Y) - V}^2] 
	= \expec[\norm{X-V_1}^2] + \expec[\norm{Y-V_2}^2]
	\ge W_2^2(\mu, \upsilon_1) + W_2^2(\nu, \upsilon_2).
	\end{equation}
	Take the infimum over all couplings $((X, Y),V)$.
	
	(ii) Trivial.
	
	(iii) If $\mu = \upsilon_1$ and $\nu = \upsilon_2$, then $T_{p,q}(\pi; \upsilon_1, \upsilon_2) = W_2^2(\pi, \mu \otimes \nu)$ and the statement is trivial. 
	Suppose that $\upsilon_1$ and $\upsilon_2$ are absolutely continuous. 
	Equality to zero means that there exists an optimal coupling $((X, Y),V)$ of $\pi$ and $\upsilon_1 \otimes \upsilon_2$ such that the inequality in Eq.~\eqref{eq:EXY2:minorant} is an equality and thus that $(X, V_1)$ and $(Y, V_2)$ are optimal couplings of $\mu \otimes \upsilon_1$ and $\nu \otimes \upsilon_2$ respectively. 
	As $\upsilon_1$ and $\upsilon_2$ are absolutely continuous, then, by Brenier's theorem \citep[Theorem~2.12]{villani2008optimal}, there exist two convex functions $\varphi_1 : \reals^{p} \to \reals \cup \{\infty\}$ and $\varphi_2 : \reals^{q} \to \reals \cup \{\infty\}$ such that $X = \nabla \varphi_1(V_1)$ and $Y = \nabla \varphi_2(V_2)$ almost surely.
	Hence, $X$ and $Y$ are independent and their distribution is $\pi = \mu \otimes \nu$.
\end{proof}

For $\upsilon_1$ and $\upsilon_2$ as in Lemma~\ref{lem:Td1d2}(iii), we have $T_{p,q}(\pi; \upsilon_1, \upsilon_2) \ge 0$ with equality if and only if $\pi = \mu \otimes \nu$.
This fact motivates the use of $T_{p,q}$ to quantify dependence between the subvectors $X$ and $Y$ of a random vector $X = (X, Y)$ with law $\pi$.
To obtain a coefficient between $0$ and $1$, we propose to rescale $T_{p,q}(\pi; \upsilon_1, \upsilon_2)$ by the largest possible value over all couplings $\tilde{\pi}$ of $\mu$ and $\nu$, provided these are both non-degenerate:
\begin{equation}
\label{eq:dep}
	\tilde{\dep}(\pi; \upsilon_1, \upsilon_2) =
	\frac{T_{p,q}(\pi; \upsilon_1, \upsilon_2)}{\sup_{\tilde{\pi} \in \Gamma(\mu,\nu)} T_{p,q}(\tilde{\pi}; \upsilon_1, \upsilon_2)}.
\end{equation}
The coefficient is indicated with a tilde to indicate the link and difference with the covariance-matrix-based coefficients defined in Section~\ref{sec:Gauss}.
Under the conditions of Lemma~\ref{lem:Td1d2}(iii) and as $\mu$ and $\nu$ are non-degenerate, the supremum in the denominator in~\eqref{eq:dep} is positive.
In that case,  $\tilde{\dep}(\pi; \upsilon_1, \upsilon_2) \in [0, 1]$, while $\tilde{\dep}(\pi; \upsilon_1, \upsilon_2) = 0$ if and only if $\pi = \mu \otimes \nu$.
The supremum in the denominator is attained since $T_{p,q}(\,\cdot\,;\upsilon_1,\upsilon_2)$ is $W_2$-continuous on $\Prob_2(\Rd)$ and $\Gamma(\mu,\nu)$ is $W_2$-compact in $\Prob_2(\Rd)$, as $W_2$-convergence implies convergence in distribution and the margins are fixed.

From Eq.~\eqref{eq:dep}, we can define two dependence measures that are theoretically particularly appealing.
For integer $m \ge 1$, let $\gamma_m = \normal_m(0, I_m)$ denote the $m$-variate centred and isotropic Gaussian distribution, with $I_m$ the $m \times m$ identity matrix. 

\begin{dfn}[Wasserstein dependence coefficients]
	\label{def:dep}
	For positive integer $d = p+q$ and for $\pi \in \Gamma(\mu,\nu)$ with $\mu \in \Prob_{2,0}(\reals^p)$ and $\nu \in \Prob_{2,0}(\reals^q)$, define
	\begin{align*}
		\tilde{\dep}_1(\pi; p, q)
		&= \tilde{\dep}(\pi; \gamma_{p}, \gamma_{q}) 
		= \frac%
	{W_2^2(\pi, \gamma_d) - W_2^2(\mu, \gamma_{p}) - W_2^2(\nu, \gamma_{q})}%
	{\sup_{\tilde{\pi} \in \Gamma(\mu,\nu)} W_2^2(\tilde{\pi}, \gamma_d) - W_2^2(\mu, \gamma_{p}) - W_2^2(\nu, \gamma_{q})} 		 \\
\intertext{and}
		\tilde{\dep}_2(\pi; p, q)
		&= \tilde{\dep}(\pi; \mu, \nu) 
		= \frac%
	{W_2^2(\pi, \mu \otimes \nu)}%
	{\sup_{\tilde{\pi} \in \Gamma(\mu, \nu)} 
		W_2^2(\tilde{\pi}, \mu \otimes \nu)}. 			
	\end{align*}
	If the dimensions $p$ and $q$ are clear from the context, we just write $\tilde{\dep}_r(\pi)$ for $r \in \{1,2\}$.
\end{dfn}

These measures enjoy the following properties. %
Recall that an orthogonal transformation of Euclidean space is a linear transformation induced by an orthogonal matrix.

\begin{prop}
	\label{prop:deptilde}
	Let $d = p + q$, let $\mu \in \Prob_{2,0}(\reals^p)$ and $\nu \in \Prob_{2,0}(\reals^q)$ and let $\pi \in \Gamma(\mu, \nu)$. The dependence coefficients $\tilde{\dep}_r = \tilde{\dep}_{r}(\,\cdot\,;p,q)$ for $r \in \{1,2\}$ satisfy the following properties:
	\begin{enumerate}[(i)]
	\item $\tilde{\dep}_r(\pi) \in [0, 1]$, while $\tilde{\dep}_r(\pi) = 0$ if and only if $\pi = \mu \otimes \nu$.
	\item There exists $\pi^{(r)} \in \Gamma(\mu, \nu)$ such that $\tilde{\dep}_r(\pi^{(r)}) = 1$.
	\item $\tilde{\dep}_r$ is invariant w.r.t.\ orthogonal linear transformations within the first $p$ and the last $q$ coordinates.
	\end{enumerate}
\end{prop}

\begin{proof}[Proof of Proposition~\ref{prop:deptilde}]
	Assertions~(i) and~(ii) follow in a straightforward way from Lemma~\ref{lem:Td1d2}.

	Assertion~(iii) follows from the invariance of the $2$-Wasserstein distance and the multivariate standard Gaussian distribution with respect to orthogonal transformations. For instance, for any orthogonal transformation $O$ of $\reals^p$ we have $W_2^2(\mu \circ O^{-1}, \gamma_p) = W_2^2(\mu \circ O^{-1}, \gamma_p \circ O^{-1}) = W_2^2(\mu, \gamma_p)$.
\end{proof} 

\begin{rmk}
\label{rem:cop}
If the univariate margins of $\pi$ are continuous, then one can apply the dependence coefficients not to $\pi$ but rather to a measure sharing the same copula and with margins admitting a finite second moment.
The resulting coefficient would then be invariant with respect to permutations within the first $p$ and last $q$ coordinates and also to monotone increasing and decreasing transformations of the $d$ univariate margins.
\end{rmk}

The two dependence measures are illustrated in Figure~\ref{fig: Representation}. Up to scaling, $\tilde{\dep}_2$ is the (squared) distance between $\pi$ and $\mu \otimes \nu$, whereas $\tilde{\dep}_1$ is the excess squared distance from $\pi$ to $\gamma_d$ compared to the one between $\mu \otimes \nu$ and $\gamma_d$.

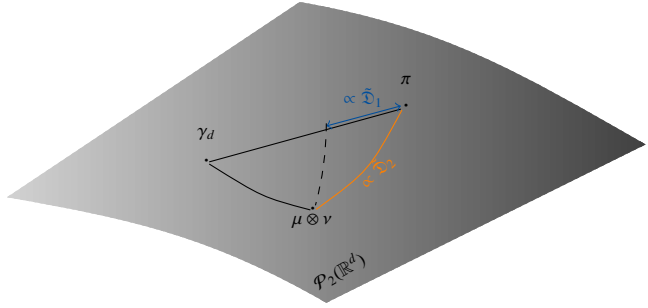
\begin{figure}
\center
\begin{tikzpicture}[scale=0.7]
\path[name path=border1] (0,0) to[out=-10,in=150] (6,-2);
\path[name path=border2] (12,1) to[out=150,in=-10] (5.5,3.2);
\path[name path=redline] (0,-0.4) -- (12,1.5);
\path[name intersections={of=border1 and redline,by={a}}];
\path[name intersections={of=border2 and redline,by={b}}];

\shade[left color=gray!35,right color=gray!150] 
  (0,0) to[out=-10,in=150] (6,-2) -- (12,1) to[out=150,in=-10] (5.5,3.7) -- cycle;

\node[label=above: {\scriptsize$\gamma_d$}] at (3.75,0.65) {$\cdot$} ;
\node[label=above:{\scriptsize$\pi$}] at (7.5,1.7) {$\cdot$} ;
\node[label=right : {}] at (5.75,-.25) {$\cdot$} ;
\node[label=left:] at (5.73,-.45) {\scriptsize$\mu \otimes \nu$ } ;

\draw[-, color=black] (3.8,0.67) --node[below] {} ++ (3.6,1.0) ;
\draw[ color=black] (3.8,0.63) .. controls (4.8,0)  .. (5.7,-0.24);
\draw[<->, color=royalblue] (6,1.35) --node[above] {\scriptsize$\propto \tilde\dep_1$} ++ (1.4,.38) ;
\draw[dashed] (6,1.4) arc (-0:-15:6cm);
\draw[color= orange] (5.81,-0.23) .. controls (6.8,0.5)  .. (7.42,1.65);
\node[rotate=45, color=orange] at (6.95,0.5) {\scriptsize$\propto \tilde\dep_2$};

\node[rotate=30] at (6.25,-1.4) {\scriptsize$\Prob_2(\mathbb{R}^d)$};
\end{tikzpicture}
\caption{\label{fig: Representation}Representation of the proposed dependence coefficients}
\end{figure}

%% file: paper-Gauss2.tex
Although theoretically appealing, the actual computation of the two Wasserstein dependence coefficients in Definition~\ref{def:dep} is involved, not in the least because of the suprema in the denominators.
Moreover, statistical inference on the coefficients is hampered by a lack of a comprehensive large-sample theory for the Wasserstein distance involving empirical measures. We refer to \citet{panaretos2019statistical} for a recent review of the known results. Further contributions by \citet{tameling2019empirical}, \citet{lei2020convergence}, \citet{manole2021sharp} or \citet{delbarrio2021central} improve the understanding of the empirical Wasserstein distance. The latter constitutes a concrete step towards statistical inference for the coefficients of Definition~\ref{def:dep}. Additional theory is still needed, however.

Despite these drawbacks, the story does not end here.
We instead propose a quasi-Gaussian approach based on covariance matrices.
We start in Section~\ref{se: QuasiGaus} by defining the modified coefficients.
The calculation of the two coefficients relies on an interesting optimisation problem yielding an elegant solution in terms of the minimum-entropy covariance matrix with given diagonal blocks in Section~\ref{sec:Gauss:minent}.
The same matrix also realises the maximum value of the RV coefficient for fixed diagonal blocks, motivating the definition of an adjusted RV coefficient with range $[0, 1]$.
The coefficients are illustrated for various families of structured covariance matrices in Section~\ref{sec:Gauss:ex}.
We conclude in Section~\ref{sec: G-copulas} with some thoughts on the application of the coefficients to distributions with standard Gaussian margins, which we call G-copulas.

\subsection{Definition and basic properties}
\label{se: QuasiGaus}

The Wasserstein distance between centred Gaussian distributions is given by the so-called Bures--Wasserstein distance between their covariance matrices. 
We refer to \citet{BHATIA2019165} for an introduction to this distance between positive semi-definite matrices and to  \citet{dowson1982frechet, olkin1982distance} for a proof that this distance coincides with the Wasserstein distance for two (centred) measures belonging to the same elliptical family. 
Let $\Sd = \{ A \in \reals^{d \times d} : A^\top = A \}$ be the set of real symmetric $d \times d$ matrices, $\Sdp \subset \Sd$ the set of positive semi-definite ones and $\Sdpp \subset \Sdp$ the set of positive definite ones.

\begin{dfn}
\label{def:dW}
The squared \emph{Bures--Wasserstein distance} between $\Sigma, \Xi \in \Sdp$ is
\begin{equation}
\label{eq:W2cor}
		d_W^2(\Sigma, \Xi) := 
		W_2^2\bigl(\normal_d(0, \Sigma), \, \normal_d(0, \Xi)\bigr) 
		= \trace(\Sigma) + \trace(\Xi) - 2 \trace\bigl( ( \Sigma^{1/2}\Xi\Sigma^{1/2})^{1/2}\bigr).
\end{equation}
\end{dfn}

The right-hand side of \eqref{eq:W2cor} is symmetric in $\Sigma$ and $\Xi$, a fact which follows from the identity with the Wasserstein distance, but which can also be proven algebraically from \eqref{eq:abracadabra} below together with the cyclic permutation property of the trace operator.
To introduce the quasi-Gaussian version of the Wasserstein dependence coefficients in Definition~\ref{def:dep}, let $d = p+q$ be integer, let $\Sigma_1 \in \mathbb{S}^p_\ge$ and $\Sigma_2 \in \mathbb{S}^q_\ge$, and introduce the set
\begin{equation}
\label{eq:Sigma}
	\Gamma(\Sigma_1, \Sigma_2) = \left\{
		\Sigma \in \Sdp : \
		\Sigma = \begin{bmatrix}
			\Sigma_1 & \Psi \\
			\Psi^\top & \Sigma_2
		\end{bmatrix} \text{ for some } \Psi \in \reals^{p \times q}
	\right\}.
\end{equation}
If $(X, Y)$ is a random vector of dimension $d$ such that $X$ and $Y$ have covariance matrices $\Sigma_1$ and $\Sigma_2$, respectively, then its joint covariance matrix $\Sigma$ belongs to $\Gamma(\Sigma_1, \Sigma_2)$.
Put
\begin{equation}
\label{eq:Sigma0}
\Sigma_0 := \begin{bmatrix} \Sigma_1 & 0\\ 0 & \Sigma_2 \end{bmatrix},
\end{equation}
the covariance matrix of an independent coupling of $X$ and $Y$.
To avoid division by zero in the next definition, we need to exclude the zero matrix: let $\mathbb{S}^d_{\ge,0} = \mathbb{S}^d_\ge \setminus \{ 0 \}$.
Recall $d_W$ in Definition~\ref{def:dW}.

\begin{dfn}[Quasi-Gaussian Wasserstein dependence coefficients]
	\label{def:depFin}
	For $\Sigma \in \Gamma(\Sigma_1, \Sigma_2)$ with $\Sigma_1 \in \mathbb{S}^p_{\ge,0}$ and $\Sigma_2 \in \mathbb{S}^q_{\ge,0}$, define
	\begin{align*}
		&\dep_1(\Sigma; p, q) = \frac%
	{d_W^2(\Sigma,I_d) - d_W^2(\Sigma_1,I_p) - d_W^2(\Sigma_2,I_q)}%
	{\sup_{\tilde{\Sigma} \in \Gamma(\Sigma_1,\Sigma_2)} d_W^2(\tilde{\Sigma},I_d) - d_W^2(\Sigma_1,I_p) - d_W^2(\Sigma_2,I_q)}, 		 \\
\intertext{and }
	&	\dep_2(\Sigma; p, q) = \frac%
	{d_W^2(\Sigma,\Sigma_0)}%
	{\sup_{\tilde{\Sigma} \in \Gamma(\Sigma_1,\Sigma_2)} d_W^2(\tilde{\Sigma},\Sigma_0)}.
	\end{align*}
	If the random vector $(X, Y)$ in dimension $d = p+q$ has law $\pi \in \Prob_2(\reals^d)$ and covariance matrix $\Sigma$, then we also put $\dep_r(X, Y) = \dep_r(\pi; p, q) = \dep_r(\Sigma; p, q)$ for $r \in \{1, 2\}$.
\end{dfn}

These coefficients are to be compared with those in Definition~\ref{def:dep}.
The Wasserstein distances in the latter have now been replaced by those between the centred Gaussian distributions with the same covariance matrices.
Furthermore, in the denominator, the supremum is now with respect to all Gaussian couplings rather than between all couplings, Gaussian or not.
Even when $X$ and $Y$ are themselves Gaussian, it is, to the best of our knowledge, an open question whether the supremum over all Gaussian couplings is equal to the supremum over all couplings.

Definition~\ref{def:depFin} leaves open the question of the calculation of the suprema in the denominators of $\dep_1$ and $\dep_2$.
According to Proposition~\ref{prop:dep}, the suprema are attained, but the matrices where this occurs and the values of the suprema remain unspecified.
The problem turns out to have an elegant and explicit solution described in Section~\ref{sec:Gauss:minent}. Proposition~\ref{prop:dep:Gauss} leverages this fact to provide a computationally-friendly version of the proposed dependence coefficients.

\begin{prop}
	\label{prop:dep}
	Let $d = p + q$ and let $\Sigma \in \Gamma(\Sigma_1, \Sigma_2)$ with $\Sigma_1 \in \mathbb{S}^p_{\ge,0}$ and $\Sigma_2 \in \mathbb{S}^q_{\ge,0}$. The dependence coefficients $\dep_r = \dep_{r}(\,\cdot\,;p,q)$ for $r \in \{1,2\}$ satisfy the following properties:
	\begin{enumerate}[(i)]
		\item $\dep_r(\Sigma) \in [0, 1]$, while $\dep_r(\Sigma) = 0$ if and only if $\Sigma = \Sigma_0$ in \eqref{eq:Sigma0}.
		\item There exists $\Sigma^{(r)} \in \Gamma(\Sigma_1, \Sigma_2)$ such that $\dep_r(\Sigma^{(r)}) = 1$.
		\item $\dep_r$ is invariant w.r.t.\ orthogonal transformations within the first $p$ and the last $q$ coordinates: for orthogonal matrices $O_1$ and $O_2$ of dimensions $p \times p$ and $q \times q$, respectively, we have
		\[
			\dep_r(O \Sigma O^\top) 
			= \dep_r(\Sigma)
			\quad \text{ with } \quad 
			O = \begin{bmatrix} O_1 & 0 \\ 0 & O_2 \end{bmatrix}.
		\]
	\end{enumerate}
\end{prop}

\begin{proof}[Proof of Proposition~\ref{prop:dep}]
	Assertion~(i) follows from Assertion~(i) in Proposition~\ref{prop:deptilde} upon identifying $d_W^2$ with the squared Wasserstein distance between centered Gaussian distributions as in~\eqref{eq:W2cor}.
	Assertion~(ii) is a consequence of continuity of $d_W$ and the fact that the set $\Gamma(\Sigma_1, \Sigma_2)$ is compact.
	Assertion~(iii), finally, follows from the invariance of $d_W$ with respect to orthogonal transformations.
\end{proof}

As the coefficients $\dep_{r}$ in Definition~\ref{def:depFin} are defined in terms of covariance matrices---including correlation matrices---they can be applied whenever such matrices show up and inference on them is feasible.
A case we have in mind is when the copula of $(X, Y)$ is Gaussian and $\Sigma$ is the correlation matrix of the random vector obtained from $(X, Y)$ by transforming the univariate margins to the standard normal distribution (Section~\ref{sec: G-copulas}).
Plugging in an estimate of the covariance or correlation matrix produces estimates of the coefficients the asymptotic distributions of which can be obtained by the delta method (Section~\ref{sec:estim}).
This approach is akin to the one in \citet{rippl2016limit}, who propose inference on the Wasserstein distance between Gaussian distributions based on estimated means and covariance matrices.

As one may expect, the simplification to covariance matrices comes at a price: in Proposition~\ref{prop:dep}, a vanishing coefficient is no longer a guarantee for independence as it was in Proposition~\ref{prop:deptilde} but only implies that all cross-covariances are zero.
This fact property is shared with the RV coefficient and the one in \citet{puccetti2019measuring}.

Assume all diagonal elements of $\Sigma$ are positive and let $R = D_\Sigma^{-1/2} \Sigma D_\Sigma^{-1/2}$ be the correlation matrix associated to $\Sigma$, where $D_\Sigma$ is the diagonal matrix having the same diagonal as $\Sigma$.
Then $\dep_r(\Sigma)$ and $\dep_r(R)$ are different in general.
Hence, as in principal component analysis, it may be a good idea to scale variables to have unit variance prior to the use of the coefficients.

\subsection{Majorisation of vectors of eigenvalues}
\label{sec:Gauss:minent}

To explain the intuition, let $R$ be a $d \times d$ correlation matrix with eigenvalues $\lambda_1 \ge \ldots \ge \lambda_d \ge 0$.
Since it holds that $\lambda_1+\cdots+\lambda_d = \trace(R)=d$, the proportion of the total variance explained by the first $k$ principal components is $(\lambda_1+\cdots+\lambda_k)/ d$. 
The larger this proportion, the better the quality of representation of the $d$ standardised variables on the linear subspace spanned by the first $k$ principal components.
Intuitively, the dimension reduction is more successful as the eigenvalues are more spread out. 
The worst case in this respect occurs when $\trace(R)$ is the identity matrix and all eigenvalues are equal to $1$. The idea also applies in general for covariance matrices and underlies many inequalities in mathematics. It goes back to \citet*{HLP} and even earlier to the works of I. Schur. This theory will be key to derive the maxima in $\dep_1$ and $\dep_2$.

We rely on the monograph by \citet{marshall2011b}, from which the next definition and proposition are taken: see Definition~1.A.1 on page~8 and Proposition~3.C.1 on page~92, as well as the historical remarks on pages~93--95.

\begin{dfn}[Majorization]
	For two vectors $x,y \in \reals^d$, we say that $y$ majorizes $x$, notation $x \prec y$, if
	\[
	\left\{
	\begin{array}{rcl}
	\sum_{i=1}^k x_{[i]} &\leq& \sum_{i=1}^k y_{[i]}, \qquad k=1,\ldots, d-1, \\[1ex]
	\sum_{i=1}^d x_{[i]} &=& \sum_{i=1}^d y_{[i]},
	\end{array}
	\right.
	\]
	where $x_{[1]} \ge \ldots \ge x_{[d]}$ denote the elements of $x$ in decreasing order, and similarly for $y$.
\end{dfn}

When applied to the vectors of eigenvalues $\lambda$ and $\mu$ of two $d \times d$ covariance matrices $\Sigma$ and $\Xi$, respectively, the relation $\lambda \prec \mu$ states that, for any $k = 1,\ldots,d-1$, the reduction to the first $k$ principal components is more successful for $\Xi$ than for $\Sigma$ in terms of proportion of variance explained.
The link between majorisation and the computation of the suprema in the denominators of $\dep_1$ and $\dep_2$ stems from the following property \citep[Proposition~3.C.1]{marshall2011b}.

\begin{prop}[Majorisation and convexity]
	\label{prop:Schur}
	If $I \subseteq \reals$ is an interval and if $g : I \to \reals$ is convex, then for all $x, y \in I^d$, we have
	\[
	x \prec y \implies \sum_{i=1}^d g(x_i) \le \sum_{i=1}^d g(y_j).
	\]
\end{prop}

For fixed diagonal blocks $\Sigma_1 \in \mathbb{S}^p_\ge$ and $\Sigma_2 \in \mathbb{S}^q_\ge$, does there exist $\Sigma_m \in \Gamma(\Sigma_1, \Sigma_2)$ in~\eqref{eq:Sigma} whose vector of ordered eigenvalues majorises those of all other covariance matrices of that form? 
The answer is positive and this matrix turns out to attain the suprema in the definitions of $\dep_1$ and $\dep_2$ in Definition~\ref{def:depFin}.
The eigendecompositions of $\Sigma_1$ and $\Sigma_2$ are
\begin{equation}
\label{eq:Sigmaj:eigen}
	\Sigma_j = U_j \Lambda_j U_j^\top, \qquad j \in \{1,2\},
\end{equation}
where $\Lambda_1 = \diag(\lambda_{1,1}, \ldots, \lambda_{p,1})$ is the $p \times p$ diagonal matrix containing the $p$ ordered eigenvalues $\lambda_{1,1} \ge \ldots \ge \lambda_{p,1} \ge 0$ of $\Sigma_1$, counting multiplicities, and where the columns of the $p \times p$ orthogonal matrix $U_1$ contain the corresponding eigenvectors.
We set similar notation for the elements arising from the eigenvalue decomposition of $\Sigma_2$.

\begin{thm}[Eigenvalue majorisation given diagonal blocks]
	\label{thm:Sigmam}
	Let $\Sigma_1 \in \mathbb{S}^p_\ge$ and $\Sigma_2 \in \mathbb{S}^q_\ge$ have eigendecompositions~\eqref{eq:Sigmaj:eigen}.
	Let $d = p+q$ and define the $d \times d$ matrix
	\begin{equation}
	\label{eq:Sigma_m}
	\Sigma_m = 	\begin{bmatrix}
	\Sigma_1 & \Psi_m \\ \Psi_m^\top & \Sigma_2
	\end{bmatrix}
	\end{equation}
	with $p \times q$ off-diagonal block
	\begin{equation}
	\label{eq:Psim}
		\Psi_m = U_1 \Lambda_1^{1/2} \Pi \Lambda_2^{1/2} U_2^\top,
	\end{equation}
	where $\Pi \in \reals^{p \times q}$ is the $p \times q$ upper left block of $I_d$.
	The eigenvalues of $\Sigma_m$ are
	\begin{equation}
	\label{eq:lambda_max}
	\lambda(\Sigma_m) = (\lambda_{j,1} + \lambda_{j,2})_{j=1}^d
	\end{equation}
	where $\lambda_{j,1} = 0$ if $j \ge p + 1$ and $\lambda_{j,2} = 0$ if $j \ge q+1$.
	For any $\Sigma \in \Gamma(\Sigma_1,\Sigma_2)$ with eigenvalues $\lambda(\Sigma) = (\lambda_j)_{j=1}^d$, we have
	\begin{equation*}
	\lambda(\Sigma) \prec \lambda(\Sigma_m).
	\end{equation*}
\end{thm}

The matrix $\Sigma_m$ in \eqref{eq:Sigma_m} can be interpreted as the joint covariance matrix of two random vectors having common principal components, yielding cross-covariance matrix $\Psi_m$ in \eqref{eq:Psim}; see Remark~\ref{rmk:PCA}. The matrix $\Sigma_m$ also possesses various extremal properties (Proposition~\ref{prop:max} and Remark~\ref{rmk: vonNeumann}).
	Interchanging $\Sigma_1$ and $\Sigma_2$ leads to a matrix $\Sigma_m$ of the same form, with obvious changes, and with the same eigenvalues in~\eqref{eq:lambda_max}.

\begin{proof}[Proof of Theorem~\ref{thm:Sigmam}]
	We need to show two things: first, the eigenvalues of $\Sigma_m$ are as in Eq.~\eqref{eq:lambda_max} (which implies that $\Sigma_m$ is positive semi-definite) and second, the eigenvalues of any other $\Sigma$ of the form \eqref{eq:Sigma} are majorized by those of $\Sigma_m$.
	For ease of writing, we assume that $p \le q$; otherwise, switch the roles of the two parts in the partition. The matrix $\Pi$ then becomes
	\[
	\Pi = \begin{bmatrix} I_{p} & 0_{p \times (q-p)} \end{bmatrix} \in \reals^{p \times q}.
	\]
	
	First, since $\begin{bmatrix} U_1 & 0 \\ 0 & U_2 \end{bmatrix}$ is orthogonal, the eigenvalues of $\Sigma_m$ are the same as those of
	$
	\Lambda_m = \begin{bmatrix} 
	\Lambda_1 & \Lambda_1^{1/2} \Pi \Lambda_2^{1/2} \\
	\Lambda_2^{1/2} \Pi^\top \Lambda_1^{1/2} & \Lambda_2 
	\end{bmatrix}.
	$
	The eigenvalues and eigenvectors of $\Lambda_m$ can be found explicitly.
	For integer $1 \le r \le s$, let $e_{r,s}$ be the $r$-th canonical unit vector in $\reals^s$. 
	Then:
	\begin{itemize}
		\item 
		For $j = 1,\ldots,p$, the vector 
		$(\lambda_{j,1}^{1/2} e_{j,p}^\top, 
		\lambda_{j,2}^{1/2} e_{j,q}^\top)^\top$ 
		is an eigenvector of $\Lambda_m$ with eigenvalue $\lambda_{j,1} + \lambda_{j,2}$.
		\item
		For $j = 1,\ldots,p$, the vector 
		$(\lambda_{j,2}^{1/2} e_{j,p}^\top, 
		- \lambda_{j,1}^{1/2} e_{j,q}^\top)^\top$ 
		is an eigenvector of $\Lambda_m$ with eigenvalue $0$.
		\item
		For $j = p+1,\ldots,q$, the vector $(0^\top, e_{j,q}^\top)^\top$ is an eigenvector of $\Lambda_m$ with eigenvalue $\lambda_{j,2}$.
	\end{itemize}
	
	Second, let $\lambda_1 \ge \ldots \ge \lambda_d \ge 0$ be the eigenvalues of $\Sigma$. We need to show that
	\begin{align*}
	\sum_{j=1}^k \lambda_j 
	&\le \sum_{j=1}^k (\lambda_{j,1}+\lambda_{j,2}),
	&& k = 1, \ldots, p, \\
	\sum_{j=1}^k \lambda_j 
	&\le p + \sum_{j=1}^k \lambda_{j,2},
	&& k = p+1, \ldots, q.
	\end{align*}
	By Theorem~1 in \citet{thompson1972inequalities}, we have, for any choice of integers
	\[
	1 \le i_1 < \ldots < i_\mu \le p,
	\qquad
	1 \le j_1 < \ldots < j_\nu \le q
	\]
	that 
	\[
	\sum_{s=1}^{\mu+\nu} \lambda_{i_s+j_s-s}
	\le
	\sum_{s=1}^\mu \lambda_{i_s,1} + \sum_{s=1}^\nu \lambda_{j_s,2},
	\]
	where $i_s = p - \mu + s$ for $s > \mu$ and $j_s = q - \nu + s$ for $s > \nu$. Now:
	\begin{itemize}
		\item 
		For $k = 1, \ldots, p$, set $\mu = \nu = k$ and $i_s = j_s = s$ to find the first inequality to be proved.
		\item 
		For $k = p+1, \ldots, q$, set $\mu = p$ with  $i_s = s$ for $s = 1,\ldots,p$ and set $\nu = k$ with $j_s = s$ for $s = 1,\ldots,q$ to find the second inequality to be proved. \qedhere
	\end{itemize}
\end{proof}

\begin{ex}[$d = 2$]
	If $p = q = 1$ and $\Sigma_j = \sigma_j^2$ for $j \in \{1,2\}$, the matrix in~\eqref{eq:Sigma_m} is
	$
	\Sigma_m = \begin{bmatrix} \sigma_1^2 & \sigma_1 \sigma_2 \\ \sigma_1 \sigma_2 & \sigma_2^2 \end{bmatrix}
	$
	with eigenvalues $\sigma_1^2 + \sigma_2^2$ and $0$.
\end{ex}

\begin{ex}[$d = 3$]
	\label{ex: trivariateEquicor}
	If $p = 1$ with $\Sigma_1 = 1$ and $q = 2$ with $\Sigma_2 = \begin{bmatrix} 1 & \rho \\ \rho & 1 \end{bmatrix}$ and $\rho \in [-1, 1]$, then
	\[
	\Sigma_m
	= \begin{bmatrix}
	1 & \sqrt{(1+\abs{\rho})/2} & \sqrt{(1+\abs{\rho})/2} \\
	\sqrt{(1+\abs{\rho})/2} & 1 & \rho \\
	\sqrt{(1+\abs{\rho})/2} & \rho & 1 
	\end{bmatrix},
	\]
	the correlation matrix of $(Z_1, X_2, X_3)$, with $Z_1 = (X_2+\operatorname{sign}(\rho)X_3)/\sqrt{2}$ the first principal component of the couple $(X_2, X_3) \sim \normal_2(0, \Sigma_2)$.
	The ordered eigenvalues of $\Sigma_2$ are $1+|\rho|$ and $1-|\rho|$ and those of $\Sigma_m$ are $2+|\rho|$, $1-|\rho|$ and $0$.
\end{ex}

Among all members of $\Gamma(\Sigma_1, \Sigma_2)$, the matrix~$\Sigma_m$ occupies a special place.
According to the following proposition, it maximises the RV coefficient as well as the $2$-Wasserstein distance with respect to both $\normal_d(0, I_d)$ and $\normal_d(0, \Sigma_0)$ for $\Sigma_0$ in~\eqref{eq:Sigma0}.
Given the constraints on the margins, we think of $\normal_d(0, \Sigma_m)$ as the Gaussian distribution that is ``least random'', ``most structured'', or ``farthest away from independence''.
These claims can be made precise if, as in Remark~\ref{rmk: vonNeumann}, the amount of structure is quantified by the von Neumann entropy.

\begin{prop}[Extremal properties of $\Sigma_m$]	
	\label{prop:max}
	Let $d = p+q$ be integer and let $\Sigma_1 \in \mathbb{S}^p_\ge$ and $\Sigma_2 \in \mathbb{S}^q_\ge$. Among all $\Sigma \in \Gamma(\Sigma_1, \Sigma_2)$ in~\eqref{eq:Sigma}, the matrix $\Sigma_m$ in \eqref{eq:Sigma_m}:
	\begin{enumerate}[(i)]
		\item
		maximizes $d_W(\Sigma, I_d)$;
		\item
		maximizes $d_W(\Sigma, \Sigma_0)$ with $\Sigma_0$ as in~\ref{eq:Sigma0};
		\item
		maximizes $\trace(\Psi \Psi^\top)$ and therefore maximizes the RV coefficient.
	\end{enumerate}
\end{prop}

As a consequence, the dependence coefficients $\dep_1(\Sigma)$ and $\dep_2(\Sigma)$ are maximal, i.e., equal to $1$, if $\Sigma$ is equal to $\Sigma_m$. See Remark~\ref{rmk:PCA} for a statistical interpretation of this form of dependence in terms of principal components.

\begin{proof}[Proof of Proposition~\ref{prop:max}]
	(i) Recall that $\lambda_1 \ge \ldots \ge \lambda_d \ge 0$ are the eigenvalues of $\Sigma$.
	By Eq.~\eqref{eq:W2cor}, we have
	\[
	W_2^2 \bigl( \normal_d(0, \Sigma), \, \normal_d(0, I_d) \bigr)
	= d + \trace(\Sigma) - 2\trace(\Sigma^{1/2})
	= d + \sum_{j=1}^d \lambda_j - 2\sum_{j=1}^d \lambda_j^{1/2}.
	\]
	Since the function $\lambda \mapsto \lambda -2 \lambda^{1/2}$ is convex on $\lambda \in \reals_{\ge}$, the claim of maximality follows from Proposition~\ref{prop:Schur} and Theorem~\ref{thm:Sigmam}.
	
	(ii) We have
	\[
	\Sigma_0^{1/2} \Sigma \Sigma_0^{1/2}
	=
	\begin{bmatrix} \Sigma_1^{1/2} & 0 \\ 0 & \Sigma_2^{1/2} \end{bmatrix}
	\begin{bmatrix} \Sigma_1 & \Psi \\ \Psi^\top & \Sigma_2 \end{bmatrix}
	\begin{bmatrix} \Sigma_1^{1/2} & 0 \\ 0 & \Sigma_2^{1/2} \end{bmatrix}
	=
	\begin{bmatrix}
	\Sigma_1^2 & \Sigma_1^{1/2} \Psi \Sigma_2^{1/2} \\
	\Sigma_2^{1/2} \Psi^\top \Sigma_1^{1/2} & \Sigma_2^2
	\end{bmatrix}.
	\]

	Recall the eigendecomposition~\eqref{eq:Sigmaj:eigen} of $\Sigma_j$.
	For $r \in \{1, 2\}$ and for $\alpha > 0$, the eigendecomposition of $\Sigma_r^a$ is $U_r \Lambda_r^\alpha U_r$, i.e., the eigenvectors are the same as those of $\Sigma_r$ while the eigenvalues are raised to the exponent $\alpha$.
	For $\Psi_m$ as in Eq.~\eqref{eq:Psim}, we get
	\begin{align*}
	\Sigma_1^{1/2} \Psi_m \Sigma_2^{1/2}
	&=
	\left(U_1 \Lambda_1^{1/2} U_1^\top\right) \, 
	\left(U_1 \Lambda_1^{1/2} \Pi \Lambda_2^{1/2} U_2^\top\right) \,
	\left(U_2 \Lambda_2^{1/2} U_2^\top\right) %
	=
	U_1 \Lambda_1 \Pi \Lambda_2 U_2.
	\end{align*}
	The latter matrix is of the same form as $\Psi_m$ in Eq.~\eqref{eq:Psim} but with $\Lambda_r$ replaced by $\Lambda_r^2$.
	By Theorem~\ref{thm:Sigmam} with $\Sigma_r$ replaced by $\Sigma_r^2$ for $r \in \{1, 2\}$, it follows that of all positive semidefinite $d \times d$ matrices with diagonal blocks $\Sigma_1^2$ and $\Sigma_2^2$, the eigenvalues are majorised by those of the matrix $\Sigma_0^{1/2} \Sigma_m \Sigma_0^{1/2}$.
	In view of Eq.~\eqref{eq:W2cor}, we have
	\[
	W_2^2 \bigl( \normal_d(0, \Sigma), \, \normal_d(0, \Sigma_0) \bigr)
	=
	2\trace{\Sigma} - 2\trace\bigl\{ (\Sigma_0^{1/2} \Sigma \Sigma_0^{1/2})^{1/2} \bigr\}
	=
	2\trace{\Sigma} - 2 \sum_{j=1}^d \kappa_j^{1/2} 
	\]
	with $\kappa_1, \dots, \kappa_d$ the eigenvalues of $\Sigma_0^{1/2} \Sigma \Sigma_0^{1/2}$, counting multiplicities.
	The function $\kappa \mapsto - \kappa^{1/2}$ being convex on $\kappa \in \reals_{\ge}$, the maximality follows from Proposition~\ref{prop:Schur} and Theorem~\ref{thm:Sigmam}.  
	
	(iii) For any rectangular matrix $A$, we have $\trace(A A^\top) = \sum_i \sum_j A_{ij}^2$.
	It follows that
	$
	\trace(\Sigma^2) = \trace(\Sigma_1^2) + \trace(\Sigma_2^2) + 2 \trace(\Psi \Psi^\top).
	$
	Given the diagonal blocks $\Sigma_1$ and $\Sigma_2$, maximising $\trace(\Psi \Psi^\top)$ is thus equivalent to maximising $\trace(\Sigma^2)$.
	As the function $\lambda \mapsto \lambda^2$ is convex, Proposition~\ref{prop:Schur} and Theorem~\ref{thm:Sigmam} imply that $\trace(\Sigma^2) = \sum_{j=1}^d \lambda_j^2$ is maximal for $\Sigma$ equal to $\Sigma_m$.
\end{proof}

In view of Proposition~\ref{prop:max}, we can now work out the dependence coefficients $\dep_1$ and $\dep_2$ in Definition~\ref{def:depFin}.
Let $\Sigma_1 \in \mathbb{S}^p_\ge$ and $\Sigma_2 \in \mathbb{S}^q_\ge$ and let $\Sigma \in \Gamma(\Sigma_1, \Sigma_2)$.
Let $\lambda_1 \ge \ldots \ge \lambda_d \ge 0$ denote the eigenvalues of $\Sigma$, let $\lambda_{1,1} \ge \ldots \ge \lambda_{p,1} \ge 0$ denote those of $\Sigma_1$  and $\lambda_{1,2} \ge \ldots \ge \lambda_{q,2} \ge 0$ those of $\Sigma_2$.

\begin{prop}[Quasi-Gaussian Wasserstein dependence coefficients: computation]
	\label{prop:dep:Gauss}
	Let $\Sigma_1, \Sigma_2, \Sigma$ be as above, with $\Sigma_1$ and $\Sigma_2$ non-zero.
	For $\Sigma_0$ and $\Sigma_m$ as in~\eqref{eq:Sigma0} and~\eqref{eq:Sigma_m}, respectively, we have
	\begin{align*}
		\dep_1(\Sigma)
		&= 
		\frac%
		{\trace(\Sigma_1^{1/2}) + \trace(\Sigma_2^{1/2}) - \trace(\Sigma^{1/2})}%
		{\trace(\Sigma_1^{1/2}) + \trace(\Sigma_2^{1/2}) - \trace(\Sigma_m^{1/2})} %
		=
		\frac%
		{\sum_{j=1}^{p} \lambda_{j,1}^{1/2} + \sum_{j=1}^{q} \lambda_{j,2}^{1/2} - \sum_{j=1}^d \lambda_j^{1/2}}%
		{\sum_{j=1}^{p} \lambda_{j,1}^{1/2} + \sum_{j=1}^{q} \lambda_{j,2}^{1/2} - \sum_{j=1}^{p \vee q} (\lambda_{j,1} + \lambda_{j,2})^{1/2}}, \\
	\intertext{and}
		\dep_2(\Sigma)
		&=
		\frac%
		{\trace(\Sigma) - \trace\{ (\Sigma_0^{1/2} \Sigma \Sigma_0^{1/2})^{1/2} \}}%
		{\trace(\Sigma) - \trace\{ (\Sigma_0^{1/2} \Sigma_m \Sigma_0^{1/2})^{1/2} \}} %
		=
		\frac%
		{\sum_{j=1}^d \lambda_j - \sum_{j=1}^d \kappa_j^{1/2}}%
		{\sum_{j=1}^d \lambda_j - \sum_{j=1}^{p \vee q} (\lambda_{j,1}^2 + \lambda_{j,2}^2)^{1/2}}		
	\end{align*}
	where $\kappa_1 \ge \ldots \ge \kappa_d \ge 0$ denote the eigenvalues of $\Sigma_0^{1/2} \Sigma\Sigma_0^{1/2}$.
\end{prop}
\begin{proof}[Proof of Proposition~\ref{prop:dep:Gauss}]
	First we calculate $\dep_1(\Sigma)$. By Eq.~\eqref{eq:W2cor}, we have
	$%
	W_2^2\bigl( \normal_{d}(0, \Sigma), \normal_{d}(0, I_d) \bigr)
	=
	d + \trace(\Sigma) - 2\trace(\Sigma^{1/2})
	$. %
	 Apply this result to the three terms in the numerator of~$\dep_1(\Sigma)$ and use the content of 
	Theorem~\ref{thm:Sigmam} for the denominator.
	The claim about $\dep_1(\Sigma)$ follows from direct simplifications, using $d = p + q$ and $\trace(\Sigma) = \trace(\Sigma_m) = \trace(\Sigma_0) = \trace(\Sigma_1) + \trace(\Sigma_2)$.

	The value of $\dep_2(\Sigma)$ is obtained in a similar way.	
\end{proof}

The coefficient $\dep_1(\Sigma)$ depends on $\Sigma$ only through the eigenvalues of $\Sigma_1$, $\Sigma_2$ and $\Sigma$ itself.
The coefficient $\dep_2(\Sigma)$, instead, requires the eigenvalues of $\Sigma_1$, $\Sigma_2$ and $\Sigma_0^{1/2} \Sigma \Sigma_0^{1/2}$.
We will see in the examples and the case study that the values of $\dep_1(\Sigma)$ and $\dep_2(\Sigma)$ are often rather close.
The interpretation of $\dep_2(\Sigma)$ may be more straightforward, comparing $\Sigma$ directly with $\Sigma_0$, but in terms of computations, coefficient~$\dep_1(\Sigma)$ is the simpler one.

\begin{rmk}[Perfectly correlated principal components]
	\label{rmk:PCA}
	The matrix $\Sigma_m$ in Eq.~\eqref{eq:Sigma_m} is the covariance matrix of the random vector 
	\[ 
	\begin{bmatrix} 
	U_1 \Lambda_1^{1/2} Z_1 	\vspace{1mm}
	\\
	 U_2 \Lambda_2^{1/2} Z_2 
	\end{bmatrix}
	\]
	where $Z_1 = (Z_{1,1},\ldots,Z_{1,p})^\top \sim \normal_{p}(0, I_{p})$ and $Z_2 = (Z_{2,1},\ldots,Z_{2,q})^\top \sim \normal_{q}(0, I_{q})$ and where $Z_{k,1} = Z_{k,2}$ for $k$ belonging to the set $ \{1,\ldots,p \wedge q\}$, i.e., $Z_1$ and $Z_2$ have the first $p \wedge q$ components in common.
	If the random vector $(X,Y)$ of dimension $d = p+q$ has covariance matrix $\Sigma_m$, then for $k \in \{1,\ldots,p \wedge q\}$, the $k$-th principal components of $X$ and $Y$ are perfectly correlated.
	Moreover, if $q \le p$ and if the first $q$ eigenvalues of $\Lambda_1$ are positive, we then have $Y = H X$ with $H = U_2 \Lambda_2^{1/2} \Pi' \Lambda_1^{-1/2} U_1'$, with $\Pi$ as in Theorem~\ref{thm:Sigmam} and where $\Lambda_1$ and $U_1$ can be limited to their first $q$ columns.
		Note that in the singular value decomposition of $H$, the first $q$ right-singular vectors are equal to the first $q$ eigenvectors of $\Sigma_1$.
		For general $q \times p$ matrices $A$, however, the equality $Y = AX$ does not imply that our dependence coefficients are equal to one.
	Given the two diagonal blocks, the joint covariance matrix of two such random vectors does not necessarily maximize the Bures--Wasserstein distance to the joint covariance matrix with zero cross-correlations.
\end{rmk}

\begin{rmk}[von Neumann entropy]
	\label{rmk: vonNeumann}
	Among all matrices $\Sigma$ of the form \eqref{eq:Sigma}, the matrix $\Sigma_m$ in Eq.~\eqref{eq:Sigma_m} also minimises the von Neumann entropy \citep[see Eq.~(11)]{petz2001entropy}
	\[
	- \trace(\Sigma \ln\Sigma)
	= - \sum_{j=1}^d \lambda_j \ln \lambda_j
	\]
	with $\lambda \ln \lambda$ to be interpreted as $0$ for $\lambda = 0$, and where the sum is over all $d$ eigenvalues of $\Sigma$, counting multiplicities.
	The property follows from Proposition~\ref{prop:Schur} and Theorem~\ref{thm:Sigmam} since the function $\lambda \mapsto - \lambda \ln \lambda$ is convex.
	The von Neumann entropy is a generalisation of the concept of entropy that turned useful in quantum physics in which the operators of interest are density matrices. The definition strongly resembles the one of the Shannon entropy in information theory where the eigenvalues in the above display are replaced by the probabilities associated to a finite number of events.
\end{rmk}

\begin{rmk}[Adjusted RV coefficient]
	\label{rem:RVadj}
	For $\Psi_m$ as in Eq.~\eqref{eq:Psim}, we have 
	$ 
	\trace(\Psi_m \Psi_m^\top)
	= \trace(\Lambda_1 \Pi \Lambda_2)
	= \sum_{j=1}^{p} \lambda_{j,1} \lambda_{j,2}.
	$
	Given the diagonal blocks $\Sigma_1$ and $\Sigma_2$, this is the maximal value of the numerator in the RV coefficient in Eq.~\eqref{eq:RV}. We therefore propose to adjust the RV coefficient by
	\begin{equation}
	\label{eq:RVadj}
	\overline{\RV}(\Sigma)
	=
	\frac{\RV(\Sigma)}{\RV(\Sigma_m)}
	=
	\frac{\trace(\Psi \Psi^\top)}{\trace(\Psi_m \Psi_m^\top)}.
	\end{equation}
	We have $0 \le \RV \le \overline{\RV} \le 1$, and in contrast to $\RV$, given $\Sigma_1$ and $\Sigma_2$, the adjusted version $\overline{\RV}$ can take on all values between $0$ and $1$.
\end{rmk}

\subsection{Examples}
\label{sec:Gauss:ex}

We compute the dependence coefficients $\dep_1(\Sigma)$ and $\dep_2(\Sigma)$ for $\Sigma$ in some parametric families of correlation matrices. For comparison, we also show the RV coefficient and its adjusted version $\overline{\RV}$ in~\eqref{eq:RVadj}. In these low-dimensional examples, the difference between the RV and the adjusted coefficient remains small. The difference however clearly materializes in higher-dimensional examples as in Figure~\ref{fig:EEG} (Top row) of the supplementary material, for instance.

\begin{ex}[Bivariate correlation matrix]
\label{ex:bivcorr}
Let $p = q = 1$ and for $\rho \in [-1, 1]$ put
\[
	\Sigma = \begin{bmatrix} 1 & \rho \\ \rho & 1 \end{bmatrix}.
\]
From Proposition~\ref{prop:max}, we find
\begin{equation*}
	\dep_1(\Sigma) = \dep_2(\Sigma) 
	= \frac{2 - \sqrt{1+\rho}-\sqrt{1-\rho} }{2-\sqrt{2}}.
\end{equation*}
The RV coefficient and the adjusted version $\overline{\RV}$ in~\eqref{eq:RVadj} are both equal to $\rho^2$ while the coefficient in \citet{puccetti2019measuring} is equal to $\rho$ itself.
In this case, the square of the distance correlation by \citet{szekely2007measuring} is given in their Theorem~7 and reads
$
\{\rho \arcsin(\rho) + (1-\rho^2)^{1/2}- \rho \arcsin(\rho/2) - (4-\rho^2)^{1/2}+1\}/\{1+ \pi/3 - 3^{1/2}\}
$.
These different coefficients are shown in Figure~\ref{fig: BivCompDep} on the left.
\end{ex}

\begin{figure}
\centering
\begin{tikzpicture}[scale=0.5]
\begin{axis}[
    axis lines = left,
    xlabel = {$\rho\quad$ a)},
    ylabel = {},
    legend style={at={(0.5,1)},anchor=north east}
]
\addplot [
    domain=-1:1, 
    samples=200, 
    color=green,
]
{  x^2 };
\addlegendentry{$\RV = \overline{\RV}$}
\addplot [
domain=-1:1, 
samples=200, 
color=magenta,
]
{  (x* asin(x)/180*pi + sqrt(1-x^2)- x* asin(x/2)/180*pi- sqrt(4-x^2)+1 )/(1+pi/3 -sqrt(3) };
\addlegendentry{$\text{distCorr}^2$}
\addplot [
domain=-1:1, 
samples=200, 
color=black,
]
{ (2 - sqrt(1+x)- sqrt(1-x))/(2-sqrt(2)) };
\addlegendentry{$\dep_1=\dep_2$}
\end{axis}
\end{tikzpicture} %
\begin{tikzpicture}[scale=0.5]%
\begin{axis}[
legend columns=2, 
    axis lines = left,
    xlabel = {$\rho\quad$ b)},
    ylabel = {},
    legend style={at={(0.5,1)},anchor=north east}
]
\addplot [
    domain=-.5:1, 
    samples=200, 
    color=green,
]
{  2*x^2/sqrt(2*(1+x^2))};
\addlegendentry{$\RV$ }
\addplot [
domain=-.5:1, 
samples=200, 
color=blue,
]
{ 2*x^2/(1+abs(x)) };
\addlegendentry{$\overline{\RV}$}
\addplot [
domain=-.5:1, 
samples=200, 
color=black,
]
{ 
(1 + sqrt(1+x) - sqrt(1+2 *x) - sqrt(1-x))%
	/(1 + sqrt(1 + abs(x))- sqrt(2+ abs(x))
 };
\addlegendentry{$\dep_1$}
\addplot [
domain=-.5:1, 
samples=200, 
color=red,
]
{ (2 + x - sqrt(  x^2/2 + x + 1 + x*sqrt(x^2 +12*x+12  )/2 ) - sqrt( x^2/2 + x + 1 - x*sqrt(x^2 +12*x+12  )/2) ) 
/(2 + abs(x)- sqrt(2+ 2*abs(x)+x^2)
 };
\addlegendentry{$\dep_2$}
\end{axis}
\end{tikzpicture}
\begin{tikzpicture}[scale=0.5]
\begin{axis}[
legend columns=2, 
    axis lines = left,
    xlabel = {$\rho\quad$ c)},
    ylabel = {},
    legend style={at={(0.5,1)},anchor=north east}
]
\addplot [
    domain=-1:1, 
    samples=200, 
    color=green,
]
{  (x^4+x^2)/sqrt(2*(1+x^2)) };
\addlegendentry{$\RV$ }
\addplot [
domain=-1:1, 
samples=200, 
color=blue,
]
{ (x^4+x^2)/(1+abs(x)) };
\addlegendentry{$\overline{\RV}$}
\addplot [
domain=-1:1, 
samples=200, 
color=black,
]
{ 
(1 + sqrt(1+x) +sqrt(1-x)- sqrt(1-x^2) - sqrt(x^2/2 -x*sqrt(x^2+8)/2 +1 )- sqrt(x^2/2 +x*sqrt(x^2+8)/2 +1 ) )%
	/(1 + sqrt(1 + abs(x))- sqrt(2+ abs(x)))
 };
\addlegendentry{$\dep_1$}
\addplot [
domain=-1:1, 
samples=200, 
color=red,
]
{ ( 3- sqrt(1 - x^2) -
 sqrt((3*x^2)/2 - (5^(1/2)*x*(x^2 + 4)^(1/2))/2 + 1)-
sqrt( (3*x^2)/2 + (5^(1/2)*x*(x^2 + 4)^(1/2))/2 + 1 )) 
/(2 + abs(x)- sqrt(2+ 2*abs(x)+x^2))
 };
\addlegendentry{$\dep_2$}
\end{axis}
\end{tikzpicture}%
\begin{tikzpicture}[scale=0.5] %
\begin{axis}[
legend columns=2, 
    axis lines = left,
    xlabel = {$\rho\quad$  d)},
    ylabel = {},
    ymax=1,
    legend style={at={(0.5,1)},anchor=north east}
]
\addplot [
    domain=-.7:.7, 
    samples=200, 
    color=green,
]
{  x^2/sqrt(2*(1+x^2) };
\addlegendentry{$\RV$ }
\addplot [
domain=-.7:.7, 
samples=200, 
color=blue,
]
{ x^2/(1+abs(x) };
\addlegendentry{$\overline{\RV}$}
\addplot [
domain=-.7:.7, 
samples=200, 
color=black,
]
{ 
( sqrt(1+x) + sqrt(1-x) -sqrt(1+x*sqrt(2))-sqrt(1-x*sqrt(2)) )%
	/(1 + sqrt(1 + abs(x))- sqrt(2+ abs(x)))
 };
\addlegendentry{$\dep_1$}
\addplot [
domain=-.7:.7, 
samples=200, 
color=red,
]
 coordinates {
 (-0.707106781186548,0.404948912692253)     
 (-0.686610932456503,0.285122128037974)     
 (-0.666115083726458,0.244757307658637)     
  (-0.645619234996413,0.216313219422394)     
 (-0.625123386266368,0.193748265964508)     
  (-0.604627537536323,0.174814363085492)     
 (-0.584131688806278,0.158399941939054)     
 (-0.563635840076234,0.143866457764784)     
  (-0.543139991346189,0.130810442094357)     
  (-0.522644142616144,0.118959406874391)     
  (-0.502148293886099,0.108120007103927)     
 (-0.481652445156054,0.0981497310001994)    
  (-0.461156596426009,0.0889403252279246)    
 (-0.440660747695964,0.0804075467923236)    
 (-0.42016489896592,0.0724845417162966)     
 (-0.399669050235875,0.0651174091171502)    
  (-0.37917320150583,0.0582621379405009)     
  (-0.358677352775785,0.0518824367724176)    
  (-0.33818150404574,0.0459481626632798)     
 (-0.317685655315695,0.0404341626227442)    
  (-0.29718980658565,0.0353194063064322)     
  (-0.276693957855606,0.0305863287218994)    
 (-0.256198109125561,0.0262203275402444)    
 (-0.235702260395516,0.0222093764797446)    
  (-0.215206411665471,0.0185437275457643)    
  (-0.194710562935426,0.0152156826632267)    
  (-0.174214714205381,0.0122194206574787)    
  (-0.153718865475336,0.00955086940603093)   
(-0.133223016745292,0.00720761580212766)   
(-0.112727168015247,0.00518884827413066)   
  (-0.0922313192852018,0.00349532821763509)  
 (-0.0717354705551569,0.00212938797535654)  
  (-0.0512396218251121,0.00109495405393153)  
 (-0.0307437730950673,0.000397595178340822) 
 (-0.0102479243650224,4.45956166641614e-05) 
 (0.0102479243650224,4.45956166641614e-05)  
  (0.0307437730950674,0.000397595178341569)  
 (0.0512396218251122,0.00109495405393005)   
  (0.0717354705551571,0.00212938797535507)   
  (0.0922313192852019,0.00349532821763581)   
 (0.112727168015247,0.00518884827413066)    
 (0.133223016745292,0.00720761580212766)    
 (0.153718865475337,0.00955086940603022)    
 (0.174214714205381,0.0122194206574787)     
  (0.194710562935426,0.0152156826632274)     
  (0.215206411665471,0.0185437275457643)     
 (0.235702260395516,0.0222093764797446)     
 (0.256198109125561,0.0262203275402444)     
  (0.276693957855606,0.0305863287218994)     
  (0.297189806585651,0.0353194063064315)     
 (0.317685655315695,0.0404341626227435)     
 (0.33818150404574,0.0459481626632804)      
   (0.358677352775785,0.0518824367724183)     
 (0.37917320150583,0.0582621379405002)      
 (0.399669050235875,0.0651174091171502)     
 (0.42016489896592,0.0724845417162966)      
  (0.440660747695964,0.0804075467923236)     
  (0.461156596426009,0.0889403252279246)     
 (0.481652445156054,0.0981497310001988)     
  (0.502148293886099,0.108120007103928)      
 (0.522644142616144,0.118959406874392)      
 (0.543139991346189,0.130810442094358)      
 (0.563635840076234,0.143866457764784)      
 (0.584131688806279,0.158399941939055)      
  (0.604627537536323,0.174814363085493)      
  (0.625123386266368,0.193748265964507)      
  (0.645619234996413,0.216313219422393)      
  (0.666115083726458,0.244757307658637)      
 (0.686610932456503,0.285122128037974)      
 (0.707106781186548,0.404948912692253)      
 };
\addlegendentry{$\dep_2$}
\end{axis}
\end{tikzpicture}
\caption{\small\slshape Dependence coefficients in various families of correlation matrices. (From left to right) Bivariate correlation matrix, trivariate equicorrelated matrix, trivariate autoregressive model and trivariate moving average model.} 
\label{fig: BivCompDep}
\end{figure}
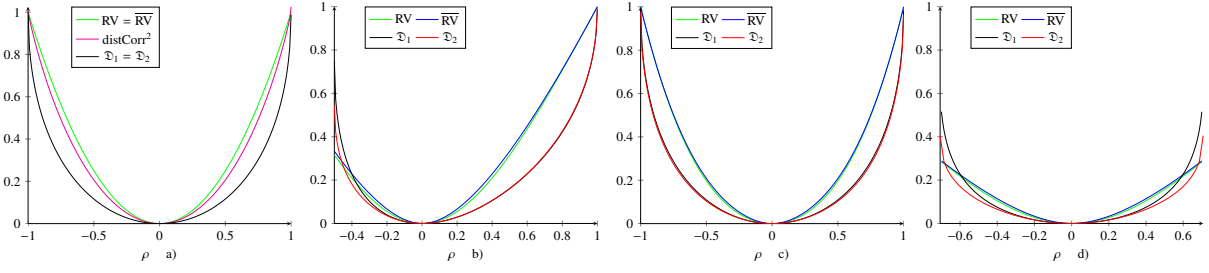

\begin{ex}[Trivariate equicorrelated matrix]
	\label{ex:trivequi}
Let $p = 1$ and $q = 2$ and for $\rho \in [-1/2, 1]$ put
\[
	\Sigma = \begin{bmatrix} 
		1 & \rho & \rho \\ 
		\rho & 1 & \rho \\ 
		\rho & \rho & 1 
	\end{bmatrix}.
\]
The matrix $\Sigma_m$ was calculated in Example~\ref{ex: trivariateEquicor}.
Even though $\dep_1(\Sigma) \ne \dep_2(\Sigma)$ in general, both functions are extremely close in this case for $\rho$ positive, with $\sup_{0 \le \rho \le 1} |\dep_1(\Sigma) - \dep_2(\Sigma)| < 0.005$. The various coefficients are shown in Figure~\ref{fig: BivCompDep} b). Some closed-form formulas used to produce the graphs exist and are deferred to \ref{sec: GausForms}. for space considerations.
\end{ex}

\begin{ex}[Model comparison]
\label{ex:comp}
In this example, we measure the dependence between a univariate random variable and a bivariate vector when the joint structure is either moving average or auto-regressive. The result for the various dependence coefficients is shown in Figure~\ref{fig: BivCompDep}. The graph c) pertains to the auto-regressive structure while the graph d) corresponds to moving averages structure, that is, to the matrices
\begin{equation}
\label{eq:AR-MA}
\begin{bmatrix}
1 &\rho &\rho^2\\
\rho & 1 & \rho \\
\rho^2 & \rho & 1
\end{bmatrix} \quad \text{for $-1 \le \rho \le 1$} \quad \text{and} \quad
\begin{bmatrix}
1 &\rho &0\\
\rho & 1 & \rho \\
0 & \rho & 1
\end{bmatrix} \quad \text{for $-\frac{1}{\sqrt{2}} \le \rho \le \frac{1}{\sqrt{2}}$},
\end{equation}
respectively. The corresponding formulas are again deferred to the supplementary material, \ref{sec: GausForms}.
\end{ex}

\subsection{G-copulas}
\label{sec: G-copulas}

For a random vector $(X, Y)$ in dimension $d = p+q$, the dependence coefficients $\dep_1$ and $\dep_2$ were defined in terms of its joint covariance matrix $\Sigma$.
As already mentioned, one may first want to rescale the variables and define the coefficients in terms of the joint correlation matrix instead.
A more radical standardisation is to transform the univariate margins to a common distribution with finite second moment.
This can be achieved by a combination of the probability and quantile transforms, provided the margins are continuous, i.e., do not have atoms.
The advantage of such an approach is that the dependence coefficients become invariant with respect to component-wise monotone increasing or decreasing transformations.
Also, on the original scale, the distribution is no longer subject to any moment conditions.

In view of the coefficients' origin in the Wasserstein distance between Gaussian distributions, a natural choice for the standardisation target is the standard normal distribution.
We call the resulting multivariate distribution a G-copula, as an alternative to classical copulas, whose margins are uniform on the unit interval.
The idea is not new: in the context of copula density estimation, \citet{geenens2017probit} also prefer the standard normal distribution as pivot.

Let $\Phi$ denote the standard normal cumulative distribution function (cdf) and let $\Phi^{-1} : [0, 1] \to [-\infty, \infty]$ denote its inverse.
A G-copula is simply a multivariate cdf with standard normal margins.
By a trivial extension of Sklar's theorem, every multivariate cdf $F$ with univariate margins $F^{(1)}, \ldots, F^{(d)}$, admits a G-copula $G$ such that
\[
	F(z)
	=
	G\bigl( \Phi^{-1} \circ F^{(1)}(z^{(1)}), \ldots, \Phi^{-1} \circ F^{(d)}(z^{(d)}) \bigr),
	\qquad z = (z^{(1)}, \ldots, z^{(d)}) \in \Rd.
\]
If the margins $F^{(1)}, \ldots, F^{(d)}$ are continuous, the G-copula $G$ in the above identity is unique and is equal to the cdf of
\[
	Z_G =
	\bigl(
	\Phi^{-1} \circ F^{(1)}(Z^{(1)}), \ldots, 
	\Phi^{-1} \circ F^{(d)}(Z^{(d)}) 
	\bigr),
\]
where the random vector $Z$ has cdf $F$.
The entries of the correlation matrix $\Sigma_G$ of $Z_G$ are called normal correlation coefficients in \citet{klaassen1997efficient}.
They are the population versions of the normal scores rank correlation coefficients.
The (ordinary) copula of $Z$ is equal to the one of a Gaussian distribution with correlation matrix $\Sigma_G$ if and only if the $G$-copula of $Z$ is equal to $\normal_d(0, \Sigma_G)$.

Given a random vector $Z = (X, Y)$ of dimension $d = p+q$ with continuous margins, we can now apply the dependence coefficients $\dep_1$ and $\dep_2$ to the random vector $Z_G = (X_G, Y_G)$ with standard normal margins obtained by the above operation. We obtain
\[
	\dep_{G,r}(X, Y) = \dep_r(X_G, Y_G) = \dep_r(\Sigma_G; p, q),
\]
where $\Sigma_G$ is the correlation matrix of the random vector $(X_G, Y_G)$.
Estimating $\Sigma_G$ by the matrix of normal scores rank correlation coefficients yields a non-parametric rank-based estimator of $\dep_{G,r}(X, Y)$.
In Section~\ref{sec:estim}, we study the asymptotic distribution of this estimator in case the copula of $(X, Y)$ is Gaussian.

%% file: paper-estim2.tex
In this section, we propose plug-in estimators for the Wasserstein-based dependence coefficients (Section~\ref{sec:estim:plugin}) and establish their limiting distributions, which is paramount for inferential purposes (Section~\ref{sec:estim:asy}). 
Before obtaining the latter results, we establish in Section~\ref{sec: FrechDiff} the Fréchet differentiability of the maps $\Sigma \mapsto \dep_r(\Sigma;p,q)$ for $r \in \{1,2\}$ in Definition~\ref{def:depFin}, where $\Sigma$ must satisfy some conditions.
The latter result opens the door to the application of our coefficients in many contexts.

\subsection{Estimators}
\label{sec:estim:plugin}

The dependence coefficients $\dep_r(\Sigma;p, q)$ for $r \in \{1, 2\}$ can be studied in any setting where a covariance or correlation matrix $\Sigma$ shows up.
The coefficient is zero if and only if $\Sigma = \Sigma_0$ in~\eqref{eq:Sigma0}. This identity implies independence provided $\Sigma$ is the covariance or correlation matrix of a Gaussian distribution.
The latter may be the distribution of the observations themselves or, as in Section~\ref{sec: G-copulas}, it may be their G-copula.
Still, the coefficients can be used in non-Gaussian settings too, in the same way as a principal component analysis can be applied to any covariance or correlation matrix. 

Recalling that $\dep_r$ for $r \in \{1,2\}$ is a function from a subset of the $d \times d$ symmetric positive semi-definite matrices to $[0,1]$, a natural way to estimate the coefficients is to consider a \emph{plug-in} estimator. 
If $\hat \Sigma_n$ is an estimator of the covariance or correlation matrix $\Sigma$ of interest, we set 
\begin{equation}
\label{eq:plug-in}
	\hat{\dep}_{n,r} = \dep_r(\hat \Sigma_n). 
\end{equation}
 
An important point to highlight at this stage is the generality of the approach. 
The matrix $\hat{\Sigma}_n$ could be the empirical covariance or correlation matrix or, in case of a G-copula, the one of normal scores rank correlation coefficients. 
Constrained covariance matrices could be used for factor models, graphical models etc.
In higher dimensions and depending on the context, one could employ a variety of regularization techniques, such as enforcing sparsity of the precision matrix or shrinking the eigenvalues.
The impact of the latter will be investigated numerically in \ref{sec:simu}.

Before stating the results, let us give an overview of how estimation of and inference on the dependence coefficients can be carried out in practice.
\begin{compactenum}
\item Estimate a covariance matrix and calculate the plug-in point estimate in Equation~\eqref{eq:plug-in}.
\item Compute the quantities appearing in Theorems~\ref{thm: DiffDep1} and \ref{thm: DiffDep2} for coefficients $\dep_1$ and $\dep_2$ respectively, based on the estimated covariance matrix.
\item Insert the latter quantities in Equation~\eqref{eq:zetar2} to estimate the asymptotic variances in the Gaussian (copula) case.
\item Construct confidence intervals and perform hypotheses tests based on the the normal approximation (Theorem~\ref{thm: LimitDep}) using the estimated variances.
\end{compactenum}
\subsection{Fréchet differentiability} 
\label{sec: FrechDiff}

First, we will prove the Fréchet differentiability of the maps $\Sigma \mapsto \dep_r(\Sigma; p, q)$ with $r \in \{1, 2\}$ for $\Sigma$ a positive definite symmetric matrix.
To this end, we will need an assumption on the diagonal blocks $\Sigma_1$ and $\Sigma_2$:
we require that $\Sigma_1$ has $p$ distinct non-zero eigenvalues and that $\Sigma_2$ has $q$ distinct non-zero eigenvalues. 
Otherwise, the functionals are still compactly (Hadamard) differentiable, but the derivatives are no longer linear and the asymptotic distribution of the plug-in estimator~\eqref{eq:plug-in} is no longer Gaussian.
The phenomenon is caused by the denominator in the definition of the coefficients, which relies on the ordering of the eigenvalues. The issue is visible in Example~\ref{ex:trivequi} at $\rho = 0$.

As $\mathbb{S}^d$, the space of symmetric real $d \times d$ matrices, is isomorph to a linear subspace of $\reals^{d^2}$, any linear map $\mathbb{S}^d \to \reals$ can be written as a trace inner product of the form
\begin{equation}
\label{eq:trace}
	H \mapsto \trace(MH) = \sum_{i=1}^d \sum_{j=1}^d M_{ij} H_{ij}
\end{equation}
for some $M \in \mathbb{S}^d$.
Fréchet derivatives being linear maps, we will write them in the above form.
The main challenge will thus be to identify the matrices $M_r$ in the limits
\begin{equation}
\label{eq:trace:form}
	\lim_{t \downarrow 0} t^{-1} \bigl( \dep_r(\Sigma + t H_t) - \dep_r(\Sigma) \bigr)
	= \trace(M_r H)
\end{equation}
for $r \in \{1, 2\}$, where $H_t, H \in \mathbb{S}^d$ and $H_t \to H$ element-wise as $t \downarrow 0$.
We will assume that $\Sigma$ is positive definite, and then $\Sigma + t H_t$ will be so too for $t$ sufficiently close to zero.

We introduce some notation.
Recall that $\mathbb{S}^m_>$ denotes the set of symmetric positive definite real $m \times m$ matrices.
Fix positive integer $d = p+q$. 
Let $\Sigma_1 \in \mathbb{S}^p_>$ and $\Sigma_2 \in \mathbb{S}^q_>$ and let $\Sigma \in \Gamma(\Sigma_1, \Sigma_2)$ as in~\eqref{eq:Sigma} and $\Sigma_0$ as in \eqref{eq:Sigma0}.
The eigendecompositions $\Sigma_r = U_r \Lambda_r U_r^\top$ for $r \in \{1,2\}$  in~\eqref{eq:Sigmaj:eigen} allow us to define the matrix $\Sigma_m$ in~\eqref{eq:Sigma_m}.
Let $\Pi_{1}$ be the projection matrix onto the first $p$ coordinates and $\Pi_{2}$ the one onto the last $q$ coordinates, that is,
\begin{align}
\label{eq:Pi12}
	\Pi_1 &= \begin{bmatrix} I_{p} & 0 \end{bmatrix} \in \reals^{p \times d}, &
	\Pi_2 &= \begin{bmatrix} 0 & I_{q} \end{bmatrix} \in \reals^{q \times d}.
\end{align}
Note that $\Sigma_j = \Pi_j \Sigma \Pi_j^\top$ for $j \in \{1, 2\}$.
Assume $q \ge p$ (otherwise, switch the roles of $p$ and $q$) and partition the second eigenvalue matrix $\Lambda_2 \in \mathbb{S}^q_>$ as
\begin{equation}
\label{eq:Lambda212}
	\Lambda_2 = 
	\begin{bmatrix} \Lambda_{2,1} & 0 \\ 0 & \Lambda_{2,2} \end{bmatrix}
\end{equation}
with $\Lambda_{2,1} \in \mathbb{S}^p_>$ containing the first $p$ eigenvalues and $\Lambda_{2,2} \in \mathbb{S}^{q-p}_>$ the remaining $q - p$ ones, the second block being empty if $q = p$. Finally, define
\begin{align}
\label{eq:Deltar}
	\Delta_1 &= 
	(\Lambda_1 + \Lambda_{2,1})^{-1/2}, &
	\Delta_2 &= 
	\begin{bmatrix} 
		\Delta_1 & 0 \\
		0 & \Lambda_{2,2}^{-1/2}
	\end{bmatrix}.
\end{align}
We can now state the differentiability of $\dep_1$ and 
$\dep_2$ with derivatives in the form~\eqref{eq:trace:form}.
The meaning of the constants and matrices in the formulas is explained in Remark~\ref{rmk:matcon}.

\begin{thm}[Differentiability of $\dep_1$]
\label{thm: DiffDep1}
Consider the set-up in the previous paragraph.
Assume that $\Sigma \in \mathbb{S}^d_>$, that $\Sigma_1$ has $p$ distinct eigenvalues and $\Sigma_2$ has $q$ distinct eigenvalues.
Let $H_t \in \mathbb{S}^d$ for $t > 0$ and $H \in \mathbb{S}^d$ be such that $H_t \to H$ element-wise as $t \downarrow 0$.
Then
\begin{align*}
	\lim_{t\to 0} 
	t^{-1} \bigl( \dep_1(\Sigma+tH_t) - \dep_1(\Sigma) \bigr)
	= \trace(M_1H)
\end{align*}
with
\begin{align}
\nonumber
	M_1 &= 
	\frac{1}{2c_1} \left( 
		- \Sigma^{-1/2} 
		+ \bigl(1 - \dep_1(\Sigma)\bigr) \Sigma_0^{-1/2} 
		+ \dep_1(\Sigma) \Upsilon_1 
	\right), \\
\nonumber
	c_1 &= 
	\trace(\Sigma_1^{1/2}) + \trace(\Sigma_2^{1/2}) - \trace(\Sigma_m^{1/2}), \\
\label{eq:Upsilon1}
	\Upsilon_1 &= 
	\begin{bmatrix} 
		U_1 \Delta_1 U_1^\top & 0 \\ 
		0 & U_2 \Delta_2 U_2^\top 
	\end{bmatrix}.
\end{align}
\end{thm}
\begin{proof}[Proof of Theorem \ref{thm: DiffDep1}]
	Note that for $t$ close enough to zero, $\Sigma + t H_t$ is positive definite since $\Sigma$ is so and since $\Sigma + t H_t \to \Sigma$ element-wise as $t \downarrow 0$.
	Consider the function
	\[
		f(\bar{x}, \bar{y}, \bar{z}, \bar{w}) 
		= \frac{\bar{y} + \bar{z} - \bar{x}}{\bar{y} + \bar{z} - \bar{w}}.
	\]
	We have $\dep_1(\Sigma) = f(x, y, z, w)$ and $\dep_1(\Sigma + t H_t) = f(x_t, y_t, z_t, w_t)$ where
	\begin{align*}
		x &= \trace(\Sigma^{1/2}), &
		y &= \trace(\Sigma_1^{1/2}), &
		z &= \trace(\Sigma_2^{1/2}), &
		w &= \trace(\Sigma_m^{1/2}),
	\end{align*}
	and similarly
	\begin{align*}
		x_t &= \trace\bigl( (\Sigma + t H_t)^{1/2} \bigr), &
		y_t &= \trace\bigl( (\Sigma + t H_t)_1^{1/2} \bigr), &
		z_t &= \trace\bigl( (\Sigma + t H_t)_2^{1/2} \bigr), &
		w_t &= \trace\bigl( (\Sigma + t H_t)_m^{1/2} \bigr).
	\end{align*}
	Here, $(\Sigma + t H_t)_1$ and $(\Sigma + t H_t)_2$ are the upper $p \times p$ and lower $q \times q$ diagonal blocks of $\Sigma + t H_t$, respectively, while $(\Sigma + t H_t)_m$ is the matrix in \eqref{eq:Sigma_m} with $\Sigma$ replaced by $\Sigma + t H_t$.
	
	Provided the quantities $(x_t - x)/t$ and so on converge, we have
	\[
		\frac{\dep_1(\Sigma + t H_t) - \dep_1(\Sigma)}{t}
		=
		\dot{f}_x \frac{x_t-x}{t} + 
		\dot{f}_y \frac{y_t-y}{t} +
		\dot{f}_z \frac{z_t-z}{t} + 
		\dot{f}_w \frac{w_t-w}{t} + \oh(1),
		\qquad t \downarrow 0,	
	\]
	where $\dot{f}_x$ and so on are the partial derivatives of $f$ evaluated at $(x, y, z, w)$. Using the notation $c_1 = y+z-w$, straightforward computation gives
	\begin{align*}
	\dot{f}_x &= - \frac{1}{c_1}, &
	\dot{f}_y = \dot{f}_z &= \frac{1 - \dep_1(\Sigma)}{c_1}, &
	\dot{f}_w &= \frac{\dep_1(\Sigma)}{c_1}.
	\end{align*}
	It follows that, as $t \downarrow 0$ and provided $(x_t-x)/t$ and so on converge,
	\begin{equation}
	\label{eq: Interm}
	\frac{\dep_1(\Sigma + tH_t) - \dep_1(\Sigma)}{t} \\
	=
	\frac{1}{c_1}
	\left(
		- \frac{x_t-x}{t} 
		+ (1 - \dep_1(\Sigma)) \frac{y_t - y + z_t - z}{t}
		+ \dep_1(\Sigma) \frac{w_t-w}{t}
	\right)
	+ \oh(1).
	\end{equation}
	Let $H_{11}$ and $H_{22}$ be the upper $p \times p$ and lower $q \times q$ diagonal blocks of $H$.
	By~\eqref{eq:Tdiff:simple},
	\begin{equation}
	\label{eq:Qt}
		\lim_{t \downarrow 0} \frac{x_t-x}{t} 
		= \frac{1}{2} \trace(\Sigma^{-1/2} H),
	\end{equation}
	as well as
	\begin{align}
         \label{eq:Qt12}
	\lim_{t \downarrow 0} \frac{y_t - y + z_t - z}{t}
	 &= \frac{1}{2} \trace(\Sigma_1^{-1/2} H_{11}) + 
	 \frac{1}{2} \trace(\Sigma_2^{-1/2} H_{22}) %
	 = \frac{1}{2} \trace( \Sigma_0^{-1/2} H ).
	\end{align}
	Lemma~\ref{lem: D1Denom} further yields
	\begin{equation}
	\label{eq:Qtm}
		\lim_{t \downarrow 0} \frac{z_t-z}{t} = \frac{1}{2} \trace(\Upsilon_1 H).
	\end{equation}
	Combine equations~\eqref{eq: Interm}, \eqref{eq:Qt}, \eqref{eq:Qt12} and~\eqref{eq:Qtm} to see that
	\begin{equation*}
		\lim_{t \downarrow 0} 
		\frac{\dep_1(\Sigma + tH_t) - \dep_1(\Sigma)}{t} \\
		=
		\frac{1}{2c_1} \left(
		- \trace(\Sigma^{-1/2} H)
		+ (1 - \dep_1(\Sigma)) \trace(\Sigma_0^{-1/2} H) 
		+ \dep_1(\Sigma) \trace( \Upsilon_1 H )
		\right).
	\end{equation*}
	The claim follows by the linearity of the trace operator followed by isolating $H$.
\end{proof}

To state the Fréchet differentiability of $\dep_2$, we need some additional notation.
Recall the eigendecompositions~\eqref{eq:Sigmaj:eigen} of $\Sigma_1$ and $\Sigma_2$ and recall the partitioning of $\Lambda_2$ in~\eqref{eq:Lambda212}.
Similar to~\eqref{eq:Deltar}, define
\begin{align*}
\Delta_1' &= 
(\Lambda_1^2 + \Lambda_{2,1}^2)^{-1/2}\Lambda_1, &
\Delta_2' &= 
\begin{bmatrix} 
(\Lambda_1^2 + \Lambda_{2,1}^2)^{-1/2}  \Lambda_{2,1}& 0 \\
0 &I_{q-p} 
\end{bmatrix},
\end{align*}
the second diagonal block of $\Delta_2'$ being empty if $q = p$.
Consider the $d \times d$ matrices
\begin{align}
\label{eq:J:Sigma}
	J &= \Sigma_0^{-1/2} \bigl( \Sigma_0^{1/2} \Sigma \Sigma_0^{1/2} \bigr)^{1/2} \Sigma_0^{-1/2} 
	= \begin{bmatrix} J_{11} & J_{12} \\ J_{21} & J_{22} \end{bmatrix}, \\
\label{eq:J0:Sigma}
	J_0 &= \begin{bmatrix} J_{11} & 0 \\ 0 & J_{22} \end{bmatrix},
\end{align}
the dimensions of the two diagonal blocks $J_{11}$ and $J_{22}$ being $p \times p$ and $q \times q$, respectively.

\begin{thm}[Differentiability of $\dep_2$]
\label{thm: DiffDep2}
	Under the same assumptions as in Theorem~\ref{thm: DiffDep1}, we have
	\begin{align*}
		\lim_{t\to 0} 
		 t^{-1} \bigl( \dep_2(\Sigma+tH_t) - \dep_2(\Sigma) \bigr)
		= \trace(M_2 H)
	\end{align*}
	where
	\begin{align}
	\nonumber
		M_2 &= \frac{1}{c_2} \left(
		- \frac{1}{2} (J_0 + J^{-1})
		+ (1 - \dep_2(\Sigma)) I_d
		+ \dep_2(\Sigma) \Upsilon_2
		\right), \\
	\nonumber
		c_2 &= \trace(\Sigma)- \trace \left(
		\bigl(\Sigma_0^{1/2}\Sigma_m\Sigma_0^{1/2}\bigr)^{1/2} 
		\right), \\
	\label{eq:Upsilon2}
		\Upsilon_2 &= \begin{bmatrix} 
		U_1 \Delta_1' U_1^\top & 0 \\ 
		0 & U_2 \Delta_2' U_2^\top 
		\end{bmatrix}.
	\end{align}
\end{thm}
 \begin{proof}[Proof of Theorem \ref{thm: DiffDep2}]
	The proof is similar to the one of Theorem~\ref{thm: DiffDep1}.
	Writing $f(\bar{x}, \bar{y}, \bar{z}) = (\bar{z}-\bar{x}) / (\bar{z}-\bar{y})$, we have 
	\begin{align*}
		\dep_2(\Sigma) &= f(x, y, z), &
		\dep_2(\Sigma + t H_t) &= f(x_t, y_t, z_t)
	\end{align*}
	where
	\begin{align*}
		x &= \trace\left(
		\bigl(\Sigma_0^{1/2}\Sigma\Sigma_0^{1/2}\bigr)^{1/2} 
		\right), &
		y &= \trace\left(
		\bigl(\Sigma_0^{1/2}\Sigma_m\Sigma_0^{1/2}\bigr)^{1/2} 
		\right), &
		z &= \trace ( \Sigma ),
	\end{align*}
	and similarly for $x_t, y_t, z_t$, with $\Sigma$ replaced by $\Sigma + t H_t$.
	If we can show that the three expressions $(x_t-x)/t$, $(y_t-y)/t$ and $(z_t-z)/t$ converge as $t \downarrow 0$, the chain rule yields
	\begin{equation*}
	\frac{\dep(\Sigma + t H_t) - \dep_2(\Sigma)}{t}
	= 
	\dot{f}_x \frac{x_t - x}{t}
	+ \dot{f}_y \frac{y_t - y}{t}
	+ \dot{f}_z \frac{z_t - z}{t}
	+ \oh(1),
	\qquad t \downarrow 0,
	\end{equation*}
	with partial derivatives
	\begin{align*}
	\dot{f}_x &= - \frac{1}{z-y}, &
	\dot{f}_y &= \frac{z-x}{(z-y)^2} = \frac{\dep_2(\Sigma)}{z-y}, &
	\dot{f}_z &= \frac{1-\dep_2(\Sigma)}{z-y}.
	\end{align*}
	By Corollary~\ref{cor: D2Num} and Lemma~\ref{lem: D2Denom}, we have, respectively
	\begin{align*}
		\lim_{t \downarrow 0} \frac{x_t - x}{t}
		&= \frac{1}{2} \trace \bigl( (J_0 + J^{-1}) H \bigr), &
		\lim_{t \downarrow 0} \frac{y_t - y}{t}
		&= \trace(\Upsilon_2 H).
	\end{align*}
	Further, $(z_t - z)/t = \trace(H_t) \to \trace(H)$ as $t \downarrow 0$.
	It follows that
	\begin{align*}
	\frac{\dep(\Sigma + t H_t) - \dep_2(\Sigma)}{t}
	&= \frac{1}{z-y} \left(
		- \frac{x_t - x}{t}
		+ \dep_2(\Sigma) \frac{y_t - y}{t}
		+ (1-\dep_2(\Sigma)) \frac{z_t - z}{t}
	\right)
	+ \oh(1) \\
	&\to \frac{1}{z-y} \left( 
		- \frac{1}{2} \trace\bigl((J_0 + J^{-1}) H\bigr) 
		+ \dep_2(\Sigma) \trace( \Upsilon_2 H)
		+ (1 - \dep_2(\Sigma)) \trace( H )
	\right)
	\end{align*}
	as $t \downarrow 0$.
	Isolating $H$ yields the stated limit.
\end{proof}

\begin{rmk}[Matrices and constants in Theorems~\ref{thm: DiffDep1} and~\ref{thm: DiffDep2}.]
	\label{rmk:matcon}
	The constants $c_1$ and $c_2$ are just the denominators of $\dep_1(\Sigma)$ and $\dep_2(\Sigma)$, respectively.
	The matrices $\Upsilon_1$ and $\Upsilon_2$ determine the Fréchet derivatives at $\Sigma$ of $\trace(\Sigma_m^{1/2})$ and $\trace((\Sigma_0^{1/2} \Sigma_m \Sigma_0^{1/2})^{1/2})$, appearing in the denominators of $\dep_1(\Sigma)$ and $\dep_2(\Sigma)$, see Lemmas~\ref{lem: D1Denom} and~\ref{lem: D2Denom}, respectively.
	The matrix $J$ is the unique solution in~$\mathbb{S}^d_>$ to the equation $J \Sigma_0 J = \Sigma$ and the associated linear operator constitutes the optimal transport with respect to the squared Euclidean distance from $\normal_d(0, \Sigma_0)$ to $\normal_d(0, \Sigma)$ \citep{olkin1982distance}.
\end{rmk}

\begin{rmk}[Fréchet derivative of Bures--Wasserstein distance]
\label{rmk:phi:deriv}
The proof of Theorem~\ref{thm: DiffDep2} requires the Fréchet derivative of the squared Bures--Wasserstein distance $d_W^2$ in~\eqref{eq:W2cor}.
The latter is stated in Lemma~2.4 in \citet{rippl2016limit}, but the formula is incorrect in case of repeated eigenvalues: the final double sum in their Eq.~(21) should extend over all pairs $(i, m) \in \{1, \ldots, d\}^2$ such that $i \ne m$, even those with $\lambda_i = \lambda_m$.
Their expression is derived from Corollary~2.3 in \citet{gilliam2009}, but the projection matrix~$P_j$ in there is the one on the eigenspace of the eigenvalue~$\lambda_j$, which, in case of repeated eigenvalues, has dimension larger than one.
A formula for the Fréchet derivative of $d_W^2$ in the trace form~\eqref{eq:trace} and not requiring eigendecompositions is given in Lemma~\ref{lem:phi:deriv}.
\end{rmk}

The matrix estimate used as input of the plug-in estimator in~\eqref{eq:plug-in} could be a correlation matrix obtained from an estimated covariance matrix by rescaling the $d$ variables by their estimated standard deviations.
To find the asymptotic distribution of the resulting plug-in estimator, it is useful to know the Fréchet derivative of the composite map
\begin{equation}
\label{eq:compmap}
	\Sigma \mapsto \varphi(\Sigma) \mapsto \dep_r\bigl(\varphi(\Sigma)\bigr)
\end{equation}
for $r \in \{1, 2\}$, where, for $\Sigma \in \mathbb{S}^d_\ge$ with positive diagonal elements, we put
\begin{equation}
\label{eq:varphi}
	\varphi(\Sigma) = D_{\Sigma}^{-1/2} \Sigma D_{\Sigma}^{-1/2}
\end{equation}
with $D_A$ the diagonal matrix having the same dimension and diagonal as the square matrix $A$.
The map $\varphi$ is scale invariant in the sense that $\varphi(\Delta \Sigma \Delta) = \varphi(\Sigma)$ for any diagonal matrix $\Delta \in \mathbb{S}^d_>$.
It will therefore be sufficient to calculate the Fréchet derivative of the map~\eqref{eq:compmap} at a $d \times d$ correlation matrix $R$.
Note that $D_R = I_d$ and thus $\varphi(R) = R$ for such a matrix.

\begin{cor}[Differentiability of dependence coefficients after rescaling]
	\label{cor:depcorrFrechet}
	Let $R \in \mathbb{S}^d_>$ be a correlation matrix ($D_R = I_d$).
	Under the assumptions and notation of Theorems~\ref{thm: DiffDep1} and~\ref{thm: DiffDep2} with $\Sigma$ replaced by $R$, we have, for $r \in \{1, 2\}$,
	\[
		\lim_{t \downarrow 0} t^{-1} \bigl( 
			\dep_r(\varphi(R + t H_t)) - \dep_r(R) 
		\bigr)
		=
		\trace\bigl( (M_{R,r} - D_{M_{R,r} R}) H \bigr),
	\]
	where $M_{R,r}$ is the matrix $M_r$ with $\Sigma$ replaced by $R$.
\end{cor}
\begin{proof}[Proof of Corollary~\ref{cor:depcorrFrechet}]
	Write $H_t = (h_{t,jk})_{j,k=1}^d$ and $H = (h_{jk})_{j,k=1}^d$.
	For any $j \in \{1, \ldots, d\}$, we have 
	\[ 
		[R + t H_t]_{jj}^{-1/2} 
		= (1 + t h_{t,jj})^{-1/2} 
		= 1 - \tfrac{1}{2} t h_{jj} + \oh(t),
		\qquad t \downarrow 0.
	\]
	Write $R = (\rho_{jk})_{j,k=1}^d$.
	It follows that, for $j, k \in \{1, \ldots, d\}$,
	\begin{align*}
		[\varphi(R + t H_t)]_{jk}
		&= \left(1 - \tfrac{1}{2} t h_{jj} + \oh(t)\right)^{-1/2} 
		\bigl(\rho_{jk} + t h_{jk} + \oh(t)\bigr)
		\left(1 - \tfrac{1}{2} t h_{kk} + \oh(t)\right)^{-1/2} \\
		&= \rho_{jk} + t \left( h_{jk} - \tfrac{1}{2} (h_{jj} \rho_{jk} + \rho_{jk} h_{kk}) \right) + \oh(t),
		\qquad t \downarrow 0.
	\end{align*}
	In matrix form, we find
	\begin{equation}
	\label{eq:diffphi}
	\lim_{t \downarrow 0}
	t^{-1} \bigl( \varphi(R + t H_t) - R \bigr)
	= H - \tfrac{1}{2} (D_H R + R D_H)
	=: \dot{\varphi}_R(H).
	\end{equation}
	Note that the operator $\dot{\varphi}_R : \mathbb{S}^d \to \mathbb{S}^d$ is indeed linear. 
	By the chain rule, we have
	\[
	\lim_{t \downarrow 0}
	t^{-1} \left( \dep_r\bigl(\varphi(R + t H_t)) - \dep_r(R) \right)
	= \trace \bigl(M_{R,r} \dot{\varphi}_R(H) \bigr).
	\]
	By the cyclic permutation property of the trace operator, the identity $\trace(A \diag(B)) = \trace(\diag(A) B)$ for square matrices $A$ and $B$, and the fact that $R$ and $M_{R,r}$ are symmetric and thus $R M_{R,r}$ and $M_{R,r} R$ share the same diagonal, we get
	\begin{align}
	\label{eq:diffcomp}
	\trace\bigl(M_{R,r} \dot{\varphi}_R(H)\bigr)
	&= \trace\left(
		M_{R,r} \bigl( H - \tfrac{1}{2} (D_H R + R D_H) \bigr)
	\right) %
	= \trace(M_{R,r} H) - \trace(D_{M_{R,r} R} H)  %
	= \trace \bigl( (M_{R,r} - D_{M_{R,r} R}) H \bigr). \qedhere
	\end{align}
\end{proof}

\subsection{Asymptotic distributions}
\label{sec:estim:asy}

Suppose that $\hat{\Sigma}_n$ is an estimator sequence of a covariance matrix $\Sigma$ such that, for some deterministic sequence $0 < a_n \to \infty$, we have
\begin{equation}
\label{eq:H}
	a_n \left( \hat{\Sigma}_n - \Sigma \right) \dto H, \qquad n \to \infty,
\end{equation}
where $H$ is a random symmetric matrix and the arrow $\dto$ denotes convergence in distribution.
The delta method in combination with Theorems~\ref{thm: DiffDep1} and~\ref{thm: DiffDep2} then yields
\begin{equation}
\label{eq:Drlim}
	a_n \left( \dep_r(\hat{\Sigma}_n) - \dep(\Sigma) \right)
	\dto
	\trace(M_r H), \qquad n \to \infty,
	\qquad r \in \{1, 2\}.
\end{equation}
Next, suppose $\Sigma$ has correlation matrix $\varphi(\Sigma) = R$ as in~\eqref{eq:varphi} and we wish to estimate the dependence coefficient based on the estimated correlation matrix $\varphi(\hat{\Sigma}_n)$.
The continuous mapping theorem and~\eqref{eq:H} imply
\[
	a_n \left( D_\Sigma^{-1/2} \hat{\Sigma}_n D_\Sigma^{-1/2} - R \right)
	\dto D_\Sigma^{-1/2} H D_\Sigma^{-1/2}, \qquad n \to \infty.
\]
By scale invariance of $\varphi$, Corollary~\ref{cor:depcorrFrechet} and the delta method, it follows that, for $r \in \{1, 2\}$,
\begin{align}
\label{eq:Drlim2}
	a_n \left( \dep_r(\varphi(\hat{\Sigma}_n)) - \dep_r(R) \right)
	&\dto
	\trace \bigl( 
		(M_{R,r} - D_{M_{R,r} R}) D_\Sigma^{-1/2} H D_\Sigma^{-1/2} 
	\bigr),
	\qquad n \to \infty.
\end{align}

Often, the joint distribution of the elements of the random matrix $H$ in~\eqref{eq:H} is Gaussian.
By linearity, the weak limits in \eqref{eq:Drlim} and~\eqref{eq:Drlim2} are then Gaussian too.
This includes for instance the sample covariance matrix of an independent random sample from a distribution with finite fourth moments \citep[Thm~3.1.4]{kollo2006advanced} or the matrix of pairwise Spearman's rank correlation coefficients of an independent random sample from a continuous distribution \citep[Thm~2.2]{el2003spearman}.

Here, we work out the limit distributions of the plug-in estimators in two settings:
\begin{itemize}
\item[(GD)]
	the sample correlation matrix from an independent random sample from a Gaussian distribution;
\item[(GC)]
	the matrix of normal scores rank correlation coefficients of an independent random sample from a continuous distribution with a Gaussian copula (see Section~\ref{sec: G-copulas}).
\end{itemize}
The common limit distribution in the two cases is centered normal.
The asymptotic variance is an explicit and continuous function of the underlying correlation matrix.
The latter can therefore be estimated consistently by a plug-in estimator too,
permitting the construction of asymptotic confidence intervals.

For setting~(GD), let $\xi_1,\ldots,\xi_n$ be an independent random sample from the $d$-variate normal distribution $\normal_d(\mu, \Sigma)$ with mean vector $\mu \in \reals^d$ and covariance matrix $\Sigma \in \mathbb{S}^d$.
We want to estimate the dependence coefficients $\dep_r(R)$ for $r \in \{1, 2\}$ associated to the correlation matrix $R = \varphi(\Sigma)$.
The plug-in estimator is $\hat{\dep}_{n,r} = \dep_r(\hat{R}_{n})$ where 
\begin{equation}
\label{eq:Rnhat}
	\hat{R}_n = \varphi(\hat{\Sigma}_n)
	\qquad \text{with} \qquad
	\hat{\Sigma}_n = \frac{1}{n-1} \sum_{i=1}^n (\xi_i - \bar{\xi}_n) (\xi_i - \bar{\xi}_n)^\top 
\end{equation}
is the empirical correlation matrix, based on the empirical covariance matrix $\hat{\Sigma}_n$ and with $\bar{\xi}_n = n^{-1} \sum_{i=1}^n \xi_i$ the sample mean vector.

For setting~(GC), let $\xi_1,\ldots,\xi_n$ be an independent random sample from a $d$-variate cdf $F$ with continuous univariate margins $F_1,\ldots,F_d$ and G-copula equal to the cdf of $\normal_d(0, R)$ with correlation matrix $R$.
The plug-in estimator is now $\check{\dep}_{n,r} = \dep_r(\check{R}_n)$ where 
\begin{equation}
\label{eq:Rncheck}
	\check{R}_{n} = (\check{\rho}_{n,jk})_{j,k=1}^d
	\qquad \text{with} \qquad
	\check{\rho}_{n,jk} 
	= \frac{1}{n} \sum_{i=1}^n \hat{Z}_{ij} \hat{Z}_{ik}
	\Bigg/ \frac{1}{n} \sum_{i=1}^n \bigl(\qnorm(\tfrac{i}{n+1})\bigr)^{2},	
\end{equation}
is the matrix of normal scores rank correlation coefficients
\citep[p.~113]{hajek+s:1967}, defined in terms of the normal scores
\[
	\hat{Z}_{ij} = \qnorm\bigl( \tfrac{n}{n+1} \hat{F}_{nj}(\xi_{ij}) \bigr)
\]
and the marginal empirical cdf $x_j \mapsto \hat{F}_{nj}(x_j) = n^{-1} \sum_{i=1}^n \1\{\xi_{ij} \le x_j\}$.

Surprisingly, the estimators $\hat{R}_n$ and~$\check{R}_n$ in settings~(GD) and~(GC), respectively, share the same asymptotic expansions: see Lemma~\ref{lem:Rn}, which repackages Theorem~3.1 in \citet{klaassen1997efficient}.
This explains why the limit distributions of the plug-in estimators in both settings coincide.
The form of the limit variance is a consequence of a particular property of the limit distribution of the empirical covariance matrix of a sample from the multivariate standard Gaussian distribution (Lemma~\ref{lem:Wn:asym}).

\begin{thm}[Asymptotic normality of plug-in estimators: Gaussian (copula) case]
\label{thm: LimitDep}
Let $R \in \mathbb{S}^d_>$ be a correlation matrix ($D_R = I_d$) such that the conditions of Theorem~\ref{thm: DiffDep1} are satisfied with $\Sigma$ replaced by $R$.
In settings~(GD) and~(GC) above, we have, for $\dep_{n,r} \in \{ \hat{\dep}_{n,r}, \check{\dep}_{n,r} \}$ and $r \in \{1, 2\}$,
\[
	\sqrt{n} \bigl( \dep_{n,r} - \dep_r(R) \bigr)
	\dto
	\normal(0, \zeta_r^2), \qquad n \to \infty,
\]
with asymptotic variance
\begin{equation}
\label{eq:zetar2}
	\zeta_r^2 = 2 \trace \left(
		\bigl(
			R (M_{R,r} - D_{M_{R,r} R})
		\bigr)^2
	\right)
\end{equation}
and $M_{R,r}$ the matrix $M_r$ in Theorems~\ref{thm: DiffDep1} and~\ref{thm: DiffDep2} with $\Sigma$ replaced by $R$.
\end{thm}
\begin{proof}[Proof of Theorem \ref{thm: LimitDep}]
	We have $\dep_{n,r} = \dep_r(R_n)$ with $R_n$ equal to either $\hat{R}_n$ in~\eqref{eq:Rnhat} in the Gaussian distribution setting~(GD) or $\check{R}_n$ in~\eqref{eq:Rncheck} in the Gaussian copula setting~(GC).
	In both cases, we have the expansion~\eqref{eq:Rn:mat} and thus
	\[
		\sqrt{n} (R_n - R) 
		= \sqrt{n} \left( 
			\varphi\left( {\textstyle\frac{1}{n} \sum_{i=1}^n Z_i Z_i^\top} \right) - R 
		\right) + \oh_p(1), 
		\qquad n \to \infty.
	\]
	Let the eigendecomposition of $R$ be $R = U \Lambda U^\top$, where the diagonal matrix $\Lambda$ contains the eigenvalues of $R$ on the diagonal and the columns of the orthogonal matrix $U$ contain the associated eigenvectors.
	Then $Z_i = U \Lambda^{1/2} \epsilon_i$ for $i \in \{1,\ldots,n\}$ where $\epsilon_1,\ldots,\epsilon_n$ is an independent random sample from $\normal_d(0, I_d)$.
	For $W_n$ as in~\eqref{eq:Wn}, we find
	\[
		\frac{1}{\sqrt{n}} \left( 
			\frac{1}{n} \sum_{i=1}^n Z_i Z_i^\top - R
		\right)
		= U \Lambda^{1/2} W_n \Lambda^{1/2} U^\top.
	\]
	Combining the previous expansions with the delta method and Corollary~\ref{cor:depcorrFrechet}, we get
	\begin{align*}
		\sqrt{n} \bigl( \dep_{n,r} - \dep_r(R) \bigr)
		&= \trace \left( (M_{R,r} - D_{M_{R,r} R}) U \Lambda^{1/2} W_n \Lambda^{1/2} U^\top \right) + \oh_p(1) \\
		&= \trace \left( 
			\Lambda^{1/2} U^\top (M_{R,r} - D_{M_{R,r} R}) U \Lambda^{1/2} W_n
		\right) + \oh_p(1) \\
		&\dto \trace \left( 
			\Lambda^{1/2} U^\top (M_{R,r} - D_{M_{R,r} R}) U \Lambda^{1/2} W
		\right), \qquad n \to \infty,
	\end{align*}
	with $W$ the random matrix in Lemma~\ref{lem:Wn:asym}.
	By the covariance formula~\eqref{eq:covAB} in the same lemma, the limit is centered Gaussian with asymptotic variance
	\[
		2 \trace \left( \bigl( 
			\Lambda^{1/2} U^\top (M_{R,r} - D_{M_{R,r} R}) U \Lambda^{1/2}
		\bigr)^2 \right)
		=
		2 \trace \left( \bigl( R (M_{R,r} - D_{M_{R,r} R}) \bigr)^2 \right)
	\]
	for $r \in \{1, 2\}$, using the cyclical property of the trace.
\end{proof}

For $r \in \{1, 2\}$, let $\zeta_{n,r}^2$ be the plug-in estimator of $\zeta_r^2$ given by replacing $R$ in~\eqref{eq:zetar2} by $\hat{R}_{n}$ and $\check{R}_{n}$ in settings~(GD) and~(GC), respectively.

\begin{cor}[Asymptotic normality of studentized plug-in estimators]
\label{cor: LimitDep}
	In the set-up of Theorem~\ref{thm: LimitDep}, we have $\zeta_{n,r}^2 \dto \zeta_r^2$ as $n \to \infty$ for $r \in \{1, 2\}$. If $\zeta_r^2 > 0$, then also
	\[
		\sqrt{n} \bigl(\dep_{n,r} - \dep_r(R)\bigr) / \zeta_{n,r} 
		\dto \normal(0, 1), \qquad n \to \infty.
	\]
\end{cor}
\begin{proof}[Proof of Corollary~\ref{cor: LimitDep}]
	Since $\hat{R}_n$ in setting (GD) and $\check{R}_n$ in setting (GC) are consistent estimators of $R$, it suffices to check that $M_{R,r}$ is a continuous function of $R$.
	To do so, we need to inspect the formulas for $M_1$ and $M_2$ in Theorems~\ref{thm: DiffDep1} and~\ref{thm: DiffDep2}.
	The crucial point is that the eigenvalues and eigenvectors of the upper and lower diagonal blocks $R_1$ (dimension $p \times p$) and $R_2$ (dimension $q \times q$) depend continuously on $R$, since by assumption these two blocks have $p$ and $q$ distinct eigenvalues, respectively.
\end{proof}
Corollary~\ref{cor: LimitDep} permits a standard construction of asymptotic confidence intervals for $\dep_r(R)$.
An alternative would be to employ the bootstrap as in \citet{rippl2016limit}. We do not develop this here in view of the satisfactory finite-sample performance (\ref{sec:simu:CI}) of the confidence intervals based on the normal approximation.

\begin{rmk}[Zero coefficient and testing independence]
If $\dep_r(R) = 0$, then necessarily $\zeta_r^2 = 0$ in Theorem~\ref{thm: LimitDep}: $\sqrt{n}(\dep_{n,r} - \dep_r(R))$ is non-negative and its limit distribution is centered normal, so the asymptotic variance must be zero.
This means that Theorem~\ref{thm: LimitDep} and Corollary~\ref{cor: LimitDep} cannot be used to construct tests for independence.
Instead, a higher-order result would be needed, stating weak convergence of $n \dep_{n,r}$ to a non-degenerate limit distribution, as in \citet[Theorem~2.3]{rippl2016limit}.
Since $\dep_r(R) = 0$ does not imply independence anyway, we do not pursue this idea further.
\end{rmk}

\begin{rmk}[$d = 2$]
For bivariate correlation matrices, the dependence coefficient $\dep_1(R) = \dep_2(R)$ is a smooth function of the pairwise correlation $\rho$ (Example~\ref{ex:bivcorr}).
The estimator $\dep_{n,r}$ is then equal to the corresponding value of the coefficient at the estimated correlation.
The limit distribution in Theorem~\ref{thm: LimitDep} is equal to the one given by the delta method in combination with the asymptotic normality of the empirical correlation for the bivariate normal distribution in setting~(GD) and the normal scores rank correlation for the bivariate Gaussian copula in setting~(GC).
\end{rmk}

%% file: paper-app-lemmas.tex
The following lemmas played a role in the proofs of the results in this section.
Recall that $\mathbb{S}^d$ denotes the set of real symmetric $d \times d$ matrices and $\mathbb{S}^d_> \subset \mathbb{S}^d$ the subset of positive definite such matrices.

\begin{lem}
	\label{lem:sqrt:deriv}
	Let $B \in \mathbb{S}^d_>$ and let $H_t, H \in \mathbb{S}^d$ for $t > 0$ be such that $H_t \to H$ element-wise as $t \downarrow 0$. Then 
	\begin{equation}
	\label{eq:sylvester}
		\lim_{t \downarrow 0}
		t^{-1} \bigl( (B+tH_t)^{1/2} - B^{1/2} \bigr)
		=
		X,	
	\end{equation}
	where $X \in \mathbb{S}^d$ is the solution to the Sylvester equation %
	$
		B^{1/2} X + X B^{1/2} = H.
	$
	Moreover,
	\begin{equation}
	\label{eq:Tdiff:simple}
		\lim_{t \downarrow 0}
		t^{-1} \left( 
			\trace\bigl((B + tH_t)^{1/2}\bigr) - \trace\bigl(B^{1/2}\bigr) 
		\right)
		= \trace(X) = \tfrac{1}{2} \trace(B^{-1/2} H).%
	\end{equation}
\end{lem}

In the sequel, we will also use the notation $\psi : \mathbb{S}^d_> \to \mathbb{S}^d_> : B \mapsto B^{1/2}$ and denote the Fréchet derivative of the latter map at $B$ evaluated in $G$ by $D \psi_B (G)$.

\begin{proof}
	The existence of the limit \eqref{eq:sylvester} follows from the fact that function $z \mapsto z^{1/2}$ is analytic on the positive part of the complex plane and the fact that $B$ has positive eigenvalues.
	Squaring both sides of the expansion 
	\[
	(B + t H_t)^{1/2} = B^{1/2} + t X + \oh(t)
	\] as $t \downarrow 0$ yields
	$
		B + t H_t 
		= \bigl(B^{1/2} + t X + \oh(t)\bigr)^2 
		= B + t (B^{1/2} X + X B^{1/2}) + \oh(t)$
		as $t \downarrow 0$.
	Examining the terms linear in $t$ yields the stated Sylvester equation~\eqref{eq:sylvester}.
	In that equation, premultiply both sides with $B^{-1/2}$ and take the trace to see that
	$
	\trace(X) + \trace(B^{-1/2} X B^{1/2}) = \trace(B^{-1/2} H).
	$
	But 
	$
	\trace(B^{-1/2} X B^{1/2}) = \trace(X B^{1/2} B^{-1/2}) = \trace(X)
	$ and thus
	$
	\trace(X) = \tfrac{1}{2} \trace(B^{-1/2} H)$.
\end{proof}

For $A \in \mathbb{S}^d$, let $L(A) \in \mathbb{S}^d$ be the diagonal matrix whose diagonal is equal to the $d$ eigenvalues (counting multiplicities) of $A$ in decreasing order.

\begin{lem}
	\label{lem: DifEig}
	Let $A \in \mathbb{S}^d$ have $d$ distinct (real) eigenvalues and let the orthogonal matrix $U \in \reals^{d \times d}$ contain the associated eigenvectors as columns.
	Let $H_t, H \in \mathbb{S}^d$ for $t > 0$ be such that $H_t \to H$ element-wise as $t \downarrow 0$. Then 
	\[
		\lim_{t \downarrow 0} t^{-1} \bigl( L(A + t H_t) - L(A) \bigr)
		= D_{U^\top H U}
		=: \dot{L}_A(H).
	\]
\end{lem}

\begin{proof}
	This is a special case of Theorem~3.3 in \citet{hiriarturrutylewis1999}.
\end{proof}

\begin{lem}
	\label{lem: D1Denom}
	Under the conditions of Theorem~\ref{thm: DiffDep1}, it holds that
	\begin{equation*}
	\lim_{t \downarrow 0}
	t^{-1} \left( 
		\trace \bigl( (\Sigma + t H_{t})_m^{1/2} \bigr) 
		- 
		\trace( \Sigma_m^{1/2} ) 
	\right)
	= \frac{1}{2} \trace( \Upsilon_1 H ),
	\end{equation*}
	with $(\Sigma + t H_t)_m$ the matrix in~\eqref{eq:Sigma_m} for $\Sigma$ replaced by $\Sigma + t H_t$ and with $\Upsilon_1$ defined in~\eqref{eq:Upsilon1}.
\end{lem}

\begin{proof}
	The diagonal elements of the diagonal matrix $L(\Sigma_r) = \Lambda_r$ are $\lambda_{1,1} \ge \ldots \ge \lambda_{p,1}$ for $r = 1$ and $\lambda_{1,2} \ge \ldots \ge \lambda_{q,2}$ for $r = 2$. We need to deal with the term
	\begin{equation}
	\label{eq:trB} 
	\trace(\Sigma_m^{1/2})
	= \sum_{j=1}^{q} (\lambda_{j,1} + \lambda_{j,2})^{1/2}
	=: g(\Lambda_1, \Lambda_2),
	\end{equation}
	where $\lambda_{j,1} = 0$ if $j \in \{p+1,\ldots,q\}$ (recall $q \ge p$). 
	Similarly,
	\[ 
		\trace \bigl( (\Sigma + tH_t)_m^{1/2} \bigr) 
		= g\bigl( L(\Sigma_1 + t H_{t,11}), L(\Sigma_2 + t H_{t,22}) \bigr),
	\]
	where $H_{t,11}$ and $H_{t,22}$ are the upper $p \times p$ and lower $q \times q$ diagonal blocks of $H_t$.
	In view of Lemma~\ref{lem: DifEig} and the differentiability of~$g$ in~\eqref{eq:trB}, the chain rule gives
	\begin{multline*}
	\lim_{t \downarrow 0} t^{-1} \left( 
	g\bigl(L(\Sigma_1 + t H_{t,11}), L(\Sigma_2 + t H_{t,22})\bigr) 
	- 
	g(\Lambda_1, \Lambda_2) 
	\right)
	\\
	= \sum_{j=1}^{p}
	\frac{1}{2 (\lambda_{j,1} + \lambda_{j,2})^{1/2}} [U_1^\top H_{11} U_1]_{jj}
	+ \sum_{j=1}^{q}
	\frac{1}{2 (\lambda_{j,1} + \lambda_{j,2})^{1/2}} [U_2^\top H_{22} U_2]_{jj},
	\end{multline*}
	where $H_{11}$ and $H_{22}$ are the upper $p \times p$ and lower $q \times q$ diagonal blocks of $H$.
	The right-hand side can be simplified as follows:
	with $\Pi_1$ and $\Pi_2$ as in~\eqref{eq:Pi12},
	\begin{align*}
	\ldots
	&\stackrel{\text{(a)}}{=} 
	\frac{1}{2} \left(
	\trace(\Delta_1 U_1^\top H_{11} U_1) + \trace(\Delta_2 U_2^\top H_{22} U_2) 
	\right) %
	\stackrel{\text{(b)}}{=} 
	\frac{1}{2} \left(
	\trace(\Pi_1^\top U_1\Delta_1 U_1^\top \Pi_1 H  ) 
	+ \trace(\Pi_2^\top U_2\Delta_2 U_2^\top \Pi_2 H )
	\right) \\
	&\stackrel{\text{(c)}}{=} 
	\frac{1}{2} \trace \left( 
	\begin{bmatrix} 
	U_1 \Delta_1 U_1^\top & 0 \\ 
	0 & U_2 \Delta_2 U_2^\top 
	\end{bmatrix} 
	H 
	\right)
	= \frac{1}{2} \trace( \Upsilon_1 H ),
	\end{align*}
	using the following arguments:
	\begin{enumerate}[(a)]
		\item
		by the identity $\trace(A \diag(B)) = \trace(\diag(A) B)$ for square matrices $A$ and $B$;
		\item
		by the cyclic permutation property of the trace operator together with $H_{rr} = \Pi_r H \Pi_r^\top$ for $r \in \{1, 2\}$;
		\item 
		by the identity $\Pi_1^\top A_1 \Pi_1 + \Pi_2^\top A_2 \Pi_2 = \begin{bmatrix} A_1 & 0 \\ 0 & A_2 \end{bmatrix}$ for matrices $A_1$ and $A_2$ of dimensions $p \times p$ and $q \times q$, respectively.
		\qedhere
	\end{enumerate}
\end{proof}

The following lemma provides the Fréchet derivative of the squared $2$-Wasserstein distance~\eqref{eq:W2cor} between Gaussian distributions.
As explained in Remark~\ref{rmk:phi:deriv}, it rectifies the formula in Lemma~2.4 in \citet{rippl2016limit}.

\begin{lem}[Differentiability of the Bures--Wasserstein distance]
	\label{lem:phi:deriv}
	The Fréchet derivative of the map 
	\[
	\phi : (\mathbb{S}^d_>)^2 \to \reals : (A, B) \mapsto 2 \trace\bigl((A^{1/2} B A^{1/2})^{1/2}\bigr)
	\] at $(A, B) \in (\mathbb{S}^d_>)^2$ evaluated at $(G, H) \in (\mathbb{S}^d)^2$ is
	\begin{equation}
	\label{eq:phi:deriv}
		\lim_{t \downarrow 0} t^{-1} 
		\bigl( \phi(A + t G_t, B + t H_t) - \phi(A, B) \bigr)
		= \trace(J G) + \trace(J^{-1} H) =: D\phi_{(A, B)}(G, H)
	\end{equation}
	where $G_t, H_t \in \mathbb{S}^d$ for $t > 0$ are such that $G_t \to G$ and $H_t \to H$ element-wise as $t \downarrow 0$ and where
	\begin{equation}
	\label{eq:J}
		\begin{array}{rcccc}
		J 
		&=& A^{-1/2} (A^{1/2} B A^{1/2})^{1/2} A^{-1/2} 
		&=& B^{1/2} (B^{1/2} A B^{1/2})^{-1/2} B^{1/2}, \\[1ex]
		J^{-1} 
		&=& A^{1/2} (A^{1/2} B A^{1/2})^{-1/2} A^{1/2} 
		&=& B^{-1/2} (B^{1/2} A B^{1/2})^{1/2} B^{-1/2}. 
		\end{array}
	\end{equation}
	As a consequence, the Fréchet derivative of the squared Bures--Wasserstein distance is
	\begin{equation}
	\label{eq:dW2:deriv}
		\lim_{t \downarrow 0} t^{-1} \bigl( d_W^2(A + t G_t, B + t H_t) - d_W^2(A, B) \bigr)
		= \trace\bigl( (I_d - J)G \bigr) + \trace\bigl( (I_d - J^{-1}) H \bigr).
	\end{equation}

\end{lem}

The matrices $J$ and $J^{-1}$ in~\eqref{eq:J} are the unique solutions in $\mathbb{S}^d_>$ to the matrix equations $J A J = B$ and $J^{-1} B J^{-1} = A$.
They operationalize the optimal couplings between $\normal_d(0, A)$ and $\normal_d(0, B)$ with the squared Euclidean distance as cost function \citep{olkin1982distance}.

\begin{proof}
	Equation~\eqref{eq:dW2:deriv} is an immediate consequence of \eqref{eq:phi:deriv} and the linearity of the trace operator.
	So it suffices to show~\eqref{eq:phi:deriv}.
	
	We start by showing the two identities following the definitions of $J$ and $J^{-1}$. A direct calculation gives
	\[
	\left( 
	A^{1/2} B^{1/2} (B^{1/2} A B^{1/2})^{-1/2} B^{1/2} A^{1/2} 
	\right)^2
	= A^{1/2} B A^{1/2}.
	\]
	Since the left-hand side is the square of a symmetric matrix, we find
	\begin{equation}
	\label{eq:abracadabra}
	A^{1/2} B^{1/2} (B^{1/2} A B^{1/2})^{-1/2} B^{1/2} A^{1/2}
	= (A^{1/2} B A^{1/2})^{1/2}.
	\end{equation}
	Pre- and post-multiply with $A^{1/2}$ to find
	\[
	B^{1/2} (B^{1/2} A B^{1/2})^{-1/2} B^{1/2}
	= A^{-1/2} (A^{1/2} B A^{1/2})^{1/2} A^{-1/2}, 
	\]
	which is the identity following the definition of $J$.
	The identity following the definition of $J^{-1}$ follows in the same way, by changing the roles of $A$ and $B$.
	Note that, by~\eqref{eq:abracadabra} and the cyclic permatution property of the trace operator,
	\begin{align*}
	\phi(A, B)
	&= 2 \trace \bigl( (A^{1/2} B A^{1/2})^{1/2} \bigr) %
	= 2 \trace \bigl( A^{1/2} B^{1/2} (B^{1/2} A B^{1/2})^{-1/2} B^{1/2} A^{1/2} \bigr) \\
	&= 2 \trace \bigl( (B^{1/2} A B^{1/2})^{1/2} \bigr) 
	= \phi(B, A),
	\end{align*}
	confirming the symmetry of $\phi$.

	By Lemma~\ref{lem:sqrt:deriv}, we have, as $t \downarrow 0$,
	\begin{align*}
		(A + tG_t)^{1/2} (B + tH_t) (A + tG_t)^{1/2}
	&=
	\bigl(A^{1/2} + t D\psi_A(G) + \oh(t)\bigr)
	\bigl( B + tH + \oh(t) \bigr)
	\bigl(A^{1/2} + t D\psi_A(G) + \oh(t)\bigr) \\
	&=
	A^{1/2} B A^{1/2} + t \bigl( D\psi_A(G) B A^{1/2} + A^{1/2} H A^{1/2} + A^{1/2} B D\psi_A(G) \bigr) + \oh(t).
	\end{align*}
	In Eq.~\eqref{eq:Tdiff:simple}, we have calculated the Fréchet derivative of the map $\mathbb{S}^d_> \to \reals : C \mapsto 2 \trace(C^{1/2})$ to be the linear operator $\mathbb{S}^d \to \reals : K \mapsto \trace(C^{-1/2} K)$. 
	Therefore,
	\begin{align*}
	D\phi_{(A, B)}(G, H)
	=
	\trace \left(
	(A^{1/2} B A^{1/2})^{-1/2} \bigl(
	D\psi_A(G) B A^{1/2} + A^{1/2} H A^{1/2} + A^{1/2} B D\psi_A(G)
	\bigr)
	\right).
	\end{align*}
	Isolating the term involving $H$, we find $\trace(J^{-1} H)$, as required.
	It remains to deal with the terms involving $G$. By symmetry of $\phi$, we have $D\phi_{(A, B)}(G, H) = D\phi_{(B, A)}(H, G)$. 
	The terms involving $G$ must therefore simplify to become the term involving $H$ but with the roles of $A$ and $B$ reversed: this transformation leads from $J^{-1}$ to $J$.
\end{proof}

For a $d \times d$ matrix $A$ partitioned into blocks
\[
	A = \begin{bmatrix} A_{11} & A_{12} \\ A_{21} & A_{22} \end{bmatrix}
\]
of dimensions $p \times p$, $p \times q$, $q \times p$ and $q \times q$, respectively, we put
\begin{equation}
\label{eq:A0}
	A_0 = \begin{bmatrix} A_{11} & 0 \\ 0 & A_{22} \end{bmatrix},
\end{equation}
with zero off-diagonal blocks.
This notation is coherent with the one used for $\Sigma_0$ in~\eqref{eq:Sigma0} and for $J_0$ in~\eqref{eq:J0:Sigma}.

\begin{cor} 
	\label{cor: D2Num}
	The Fréchet derivative of the map 
	\[ 
		\eta : \mathbb{S}^d_> \to \reals : \Sigma \mapsto \trace\left(\bigl(\Sigma_0^{1/2} \Sigma \Sigma_0^{1/2}\bigr)^{1/2}\right)
	\]
	is given by
	\[
		\lim_{t \downarrow 0} t^{-1} \bigl( \eta(\Sigma + t H_t) - \eta(\Sigma) \bigr)
		= \tfrac{1}{2} \trace\bigl( (J_0 + J^{-1}) H \bigr)
	\]
	for $H_t, H \in \mathbb{S}^d$ such that $H_t \to H$ element-wise as $t \downarrow 0$, with $J$ and $J_0$ as in~\eqref{eq:J:Sigma} and~\eqref{eq:J0:Sigma}, respectively.
\end{cor}

\begin{proof}
	We apply Lemma~\ref{lem:phi:deriv} with $A = \Sigma_0$, $B = \Sigma$, and, following the convention in~\eqref{eq:A0}, $G_t = (H_{t})_0$ as well as $G = H_0$ obtained from $H_t$ and $H$, respectively.
	The limit is equal to $\frac{1}{2} (\trace(J H_0) + \trace(J^{-1} H))$ with $J$ as in~\eqref{eq:J:Sigma}.
	Now %
	$\trace(J H_0) = \trace(J_0 H)$ in view of~\eqref{eq:trace}.
\end{proof}

It remains to treat the last term in the denominator in the expression for $\dep_2(\Sigma)$ in Proposition~\ref{prop:dep:Gauss}.
This is not particularly involved in the light of the earlier developments.

\begin{lem}
	\label{lem: D2Denom}
	Under the conditions of Theorem~\ref{thm: DiffDep2}, it holds that
	\begin{equation*}
	\lim_{t \downarrow 0}
	t^{-1} \trace \left(
		\bigl(
			(\Sigma + t H_{t})_0^{1/2}
			(\Sigma + t H_{t})_m 
			(\Sigma + t H_{t})_0^{1/2}
		\bigr)^{1/2}
		-
		\bigl(
			\Sigma_0^{1/2} \Sigma_m \Sigma_0^{1/2}
		\bigr)^{1/2} 
	\right)
	=  \trace (\Upsilon_2 H),
	\end{equation*}
	with $(\Sigma + t H_t)_0$ as in \eqref{eq:A0}, with $(\Sigma + t H_t)_m$ the matrix in~\eqref{eq:Sigma_m} for $\Sigma$ replaced by $\Sigma + t H_t$, and with $\Upsilon_2$ defined in~\eqref{eq:Upsilon2}.
\end{lem}

\begin{proof}[Proof of Lemma \ref{lem: D2Denom}]
	The proof is similar to the one of Lemma~\ref{lem: D1Denom}, exploiting the eigenvalue map $L$ in Lemma~\ref{lem: DifEig}.
	Recall from Proposition~\ref{prop:dep:Gauss} that the trace of interest can be written as
	\[
		\trace\left( \bigl( \Sigma_0^{1/2} \Sigma_m \Sigma_0^{1/2} \bigr)^{1/2} \right)
		= \sum_{j=1}^{p \vee q} (\lambda_{j,1}^2 +  \lambda_{j,2}^2)^{1/2}
		=: h(\Lambda_1, \Lambda_2).
	\]
	This expression is similar to the one for $\trace(\Sigma_m^{1/2})$ in~\eqref{eq:trB}, so that one can see, using the same arguments and the same notation, that 
	\begin{align*}
	\lefteqn{
		\lim_{t \downarrow 0}
		t^{-1} \Bigl( 
			h \bigl( L(\Sigma_1 + t H_{11}), L(\Sigma_2 + t H_{22}) \bigr) 
			- 
			h(\Lambda_1, \Lambda_2) 
		\Bigr) 
	} \\
	&= \sum_{j=1}^{p} 
	\frac{\lambda_{j,1}}{ (\lambda_{j,1}^2 + \lambda_{j,2}^2)^{1/2}}
	 [U_1^\top H_{11} U_1]_{jj} 
	 +
	\sum_{j=1}^{q} 
	\frac{\lambda_{j,2}}{ (\lambda_{j,1}^2 + \lambda_{j,2}^2)^{1/2}}
	 [U_2^\top H_{22} U_2]_{jj} \\
	&=  \trace(\Delta_1' U_1^\top H_{11} U_1) + \trace(\Delta_2' U_2^\top H_{22} U_2) \\
	&= \trace(\Upsilon_2 H). \qedhere
	\end{align*}
\end{proof}

\begin{lem}[Empirical covariance matrix, standard Gaussian case]
	\label{lem:Wn:asym}
	Let $\epsilon_1, \ldots, \epsilon_n$ be independent $\normal_d(0, I_d)$ random vectors and let
	\begin{equation}
	\label{eq:Wn}
		W_n = \frac{1}{\sqrt{n}} \sum_{i=1}^n (\epsilon_i \epsilon_i^\top - I_d).
	\end{equation}
	Then $W_n \dto W$ as $n \to \infty$, with $W$ a random symmetric matrix such that 
	\begin{equation*}
	W_{jk} \sim 
	\begin{cases} 
	\normal(0, 2), & \text{for $j = k \in \{1, \ldots, d\}$,} \\
	\normal(0, 1), & \text{for $1 \le j < k \le d$,}
	\end{cases}
	\end{equation*}
	all entries being independent (except for the symmetry of $W$). 
	For $A, B \in \mathbb{S}^d$, we have
	\begin{equation}
	\label{eq:covAB}
	\expec[ \trace(A W) \trace(B W) ] = 2 \trace(A B).
	\end{equation}
\end{lem}

\begin{proof}[Proof of Lemma \ref{lem:Wn:asym}]
	The weak convergence $W_n \dto W$ with $W$ as stated is a direct consequence of the multivariate central limit theorem.
	For $A \in \Sd$, we have, by symmetry of $W$,
	\begin{equation*}
	\trace(A W) 
	= \sum_{j=1}^d \sum_{k=1}^d A_{jk} W_{jk} 
	= \sum_{j=1}^d A_{jj} W_{jj} + 2 \sum_{1 \le j < k \le d} A_{jk} W_{jk}.
	\end{equation*}
	Since the random variables appearing on the last line are independent and have zero mean, it follows that, for $A, B \in \Sd$,
	\begin{align*}
	\expec[\trace(A W) \trace(B W)]
	&= \sum_{j=1}^d A_{jj} B_{jj} \expec[W_{jj}^2]
	+ 4 \sum_{1 \le j < k \le d} A_{jk} B_{jk} \expec[W_{jk}^2] \\
	&= 2 \sum_{j=1}^d A_{jj} B_{jj}
	+ 4 \sum_{1 \le j < k \le d} A_{jk} B_{jk} \\
	&= 2 \sum_{j=1}^d \sum_{k=1}^d A_{jk} B_{jk} 
	= 2 \trace(A B). \qedhere
	\end{align*}
\end{proof}

\begin{lem}[Asymptotic expansion of correlation matrix estimates]
	\label{lem:Rn}
	Let $R = (\rho_{jk})_{j,k=1}^d$ be a $d \times d$ correlation matrix and let $R_n = (\rho_{n,jk})_{j,k=1}^d$ be either the empirical correlation matrix $\hat{R}_n$ in \eqref{eq:Rnhat} in the Gaussian distribution setting (GD) or the matrix $\check{R}_{n}$ in \eqref{eq:Rncheck} of normal scores rank correlation coefficients in the Gaussian copula setting (GC).
	In both cases, for $j, k \in \{1, \ldots, d\}$,
	\begin{equation}
	\label{eq:Rn:jk}
		\sqrt{n} (\rho_{n,jk} - \rho_{jk})
		= \frac{1}{\sqrt{n}} \sum_{i=1}^n \left( Z_{ij} Z_{ik} - \frac{1}{2} \rho_{jk} (Z_{ij}^2 + Z_{ik}^2) \right) + \oh_p(1), \qquad n \to \infty,
	\end{equation}
	or, in matrix form,
	\begin{equation}
	\label{eq:Rn:mat}
		\sqrt{n} (R_n - R) =
		\frac{1}{\sqrt{n}} \sum_{i=1}^n \dot{\varphi}_R(Z_i Z_i^\top) + \oh_p(1),
		\qquad n \to \infty
	\end{equation}
	with $\dot{\varphi}_R$ as in \eqref{eq:diffphi} and with $Z_1,\ldots,Z_n$ an independent random sample from $\normal_d(0, R)$.
\end{lem}

\begin{proof}
	The matrix formula~\eqref{eq:Rn:mat} is just a repackaging of the element-wise one~\eqref{eq:Rn:jk} exploiting~\eqref{eq:diffphi}.
	
	In the Gaussian distribution setting (GD), put $Z_i = D_\Sigma^{-1/2}(\xi_i - \mu)$ for $i \in \{1,\ldots,n\}$.
	The common distribution of $Z_i$ is $\normal_d(0, R)$. Let $\hat{\Sigma}_{n,Z}$ be their empirical covariance matrix, replacing $\xi_i$ by $Z_i$ in~\eqref{eq:Rnhat}.
	We have $\xi_i = \mu + \xi_i D_\Sigma^{1/2}$ and thus
	\[
		\hat{\Sigma}_n = D_\Sigma^{1/2} \hat{\Sigma}_{n,Z} D_\Sigma^{1/2}.
	\]
	As $\varphi$ reduces variables to unit scale anyway, we have
	$
		\hat{R}_n 
		= \varphi(\hat{\Sigma}_n) 
		= \varphi(\hat{\Sigma}_{n,Z}).
	$
	By the multivariate central limit theorem and Slutsky's lemma,
	\[
		\sqrt{n}(\hat{\Sigma}_{n,Z} - R)
		= \frac{1}{\sqrt{n}} \sum_{i=1}^n (Z_i Z_i^\top - R) + \oh_p(1),
		\qquad n \to \infty.
	\]
	The delta method and the identity $\varphi(R) = R$ yield
	\[
		\sqrt{n} (\hat{R}_n - R)
		= \dot{\varphi}_R \bigl( \sqrt{n}(\hat{\Sigma}_{n,Z} - R) \bigr)
		+ \oh_p(1), \qquad n \to \infty.
	\]
	The combination of the last two expansions gives~\eqref{eq:Rn:mat} in view of linearity of $\dot{\varphi}_R$ and the identity $\dot{\varphi}_R(R) = 0$, as $R$ has unit diagonal.
	
	In the Gaussian copula setting (GC), the expansion~\eqref{eq:Rn:jk} is Theorem~3.1 in \citet{klaassen1997efficient}.
	We have $Z_i = (Z_{i1},\ldots,Z_{id})$ with $Z_{ij} = \qnorm \circ F_j^{-1}(\xi_{ij})$ for $i \in \{1,\ldots,n\}$ and $j \in \{1,\ldots,d\}$. The common distribution of the random vectors $Z_i$ is $\normal_d(0, R)$ by the assumption that the copula of $\xi_i$ is Gaussian with correlation matrix $R$.
\end{proof}

\begin{rmk}
	The expansion~\eqref{eq:Rn:mat} remains valid for the empirical correlation matrix from an independent random sample $\xi_1,\ldots,\xi_n$ from a distribution with finite fourth moments and positive variances, upon defining $Z_i = D_\Sigma^{-1/2}(\xi_i - \mu)$ with $\mu$ and $\Sigma$ the population mean vector and covariance matrix, respectively.
	The random vectors~$Z_i$ have zero means and unit variances but are no longer Gaussian.
	From the expansion, the asymptotic distribution of the empirical correlation matrix can be found using the multivariate central limit theorem.
	The asymptotic distribution of $\sqrt{n}(\hat{R}_n - R)$ is a random matrix whose $d^2$ elements have a centered multivariate normal distribution the covariance matrix of which can be derived from~\eqref{eq:Rn:jk}. 
	See also \citet[Theorem~3.1.6]{kollo2006advanced}.
\end{rmk}

%% file: paper-disc.tex
In this paper, we investigated the possibility to rely on the properties of the 2-Wasserstein distance to define new dependence coefficients that are easy to interpret. We mostly developed the theory under a Gaussian lens, thus moving from the Wasserstein distance between distributions to the Bures--Wasserstein distance between covariance or correlation matrices.  Further, we have shown that the coefficients are particularly natural in this case and that they enjoy desirable properties. 
They can be estimated easily from an empirical covariance or correlation matrix.
The asymptotic distributions of the resulting plug-in estimators can be found by the delta method, with explicit expressions for the asymptotic variances, enabling inference.
Some questions remain open and are expected to lead to further research. 

The plug-in estimators turned out to have a positive bias, which we proposed to correct by eigenvalue shrinkage in the supplementary material. Some more developments towards bias correction would certainly be welcome, for instance in the context of the matrix of normal scores rank correlation coefficients for data drawn from a distribution with a Gaussian copula.

The Fréchet-differentiability of the maps that send a covariance matrix to its dependence coefficients paves the way for further developments in large-sample theory. In a high-dimensional setting, the correlation matrix could be estimated using regularisation techniques or exploiting modelling assumptions. In time series analysis, the focus would be on auto-covariance matrices. 

A technical challenge is to obtain the limit distribution of the plug-in estimators in case all cross-covariances are zero so that the dependence coefficients are zero. The rate of convergence may then be conjectured to be $\Oh_p(n^{-1})$ and the limit laws linear combinations of independent chi-squared random variables. 
Equally interesting is to quantify the impact of the non-linearity of the (Hadamard) derivatives in case of repeated eigenvalues. A further refinement would be to allow for positive semi-definite correlation matrices instead of positive definite ones. 

The differentiability questions we referred to are important for resampling. Indeed, the $n$-out-of-$n$ bootstrap is not consistent when the Fréchet derivative is not linear. A comprehensive and careful analysis of the bootstrap consistency in this case could also be potentially interesting per se. 

Finally, one could seek for nonparametric estimators of the distribution-based dependence coefficients $\tilde{\dep}_r$. This will require new probabilistic results to derive their limit laws---or at least guarantee the possibility to approximate their sampling distributions through a numeric scheme---as well as new algorithmic developments to determine the couplings in the maximally dependent case. Identifying the couplings furthest away from a given reference point in Wasserstein space is also an interesting theoretical challenge.

%% file: paper-simu.tex
In this Appendix, we investigate the plug-in estimators for the dependence coefficients by means of various simulation experiments. First, we evaluate the quality of the approximation of their finite-sample distributions by the asymptotically normal one (\ref{sec: GoFAsym}). We then numerically assess the impact of shrinking the eigenvalues of the empirical covariance matrix to reduce the inherent bias (\ref{sec: Shrink}) and finally we evaluate the actual coverage of confidence intervals based on the normal approximation (\ref{sec:simu:CI}).

\subsection{Gaussian goodness-of-fit for finite samples}
\label{sec: GoFAsym}
The Figure~\ref{fig: Simu1} presents P-P plots illustrating the asymptotic normality of the plug-in estimators in Section~\ref{sec:estim:asy}.

The results  are resented for for Gaussian data (GD) with correlation matrix estimated by $\hat{R}_{n}$ in~\eqref{eq:Rnhat}.
The standard normal distribution function is on the vertical axis while the actual sampling distribution function of $\sqrt{n} (\dep_{n,r} - \dep_r(R)) / \zeta_{n,r}$ based on 3000 independent replications is on the horizontal one. From left to right, the sample sizes are 50, 200, 1000 and 5000, respectively.

The three rows correspond to the three following settings.

\begin{enumerate}
\item A trivariate autoregressive 
matrix ($p=1$, $q=2$) as in \eqref{eq:AR-MA} with coefficient $\rho=0.25$. The true values of $\dep_1$ and $\dep_2$ are 0.026 and 0.025 respectively.
\item A trivariate autoregressive matrix ($p=1$, $q=2$) with coefficient $\rho=0.8$. The true values of $\dep_1$ and $\dep_2$ are 0.34 and 0.33 respectively.
\item A five-variate correlation matrix with $p=2$ and $q=3$ without any particular structure:
\[ 
	\begin{bmatrix}
	1.00& 	0.20&	0.15&	0.10&	0.25 \\
	0.20& 	1.00&	0.05&	0.30&	0.35 \\
	0.15&	0.05&	1.00&	0.40&	0.50 \\
	0.10&	0.30&	0.40&	1.00&	0.45 \\
	0.25&	0.35& 	0.50&	0.45&	1.00
	\end{bmatrix}.
\]
The true values of $\dep_1$ and $\dep_2$ are $0.051$ and $0.050$ respectively.
\end{enumerate}
Observing Figure~\ref{fig: Simu1} one can clearly see that in case $n = 50$, the quality of the normal approximation is much better for larger values of the coefficients. For the five-dimensional example, the lack-of-fit at $n = 50$ is rather pronounced, as one could expect given the number of matrix entries to estimate. In particular, the estimator has a large positive bias. In all three settings, the goodness-of-fit improves with the sample size, as expected. We evoke the high-dimensional case, that is, when the number of matrix entries is of the order of magnitude of $n$, in Section~\ref{sec:disc}.

\begin{figure}[ht!]
	\begin{center}
			\includegraphics[width=0.99\textwidth]{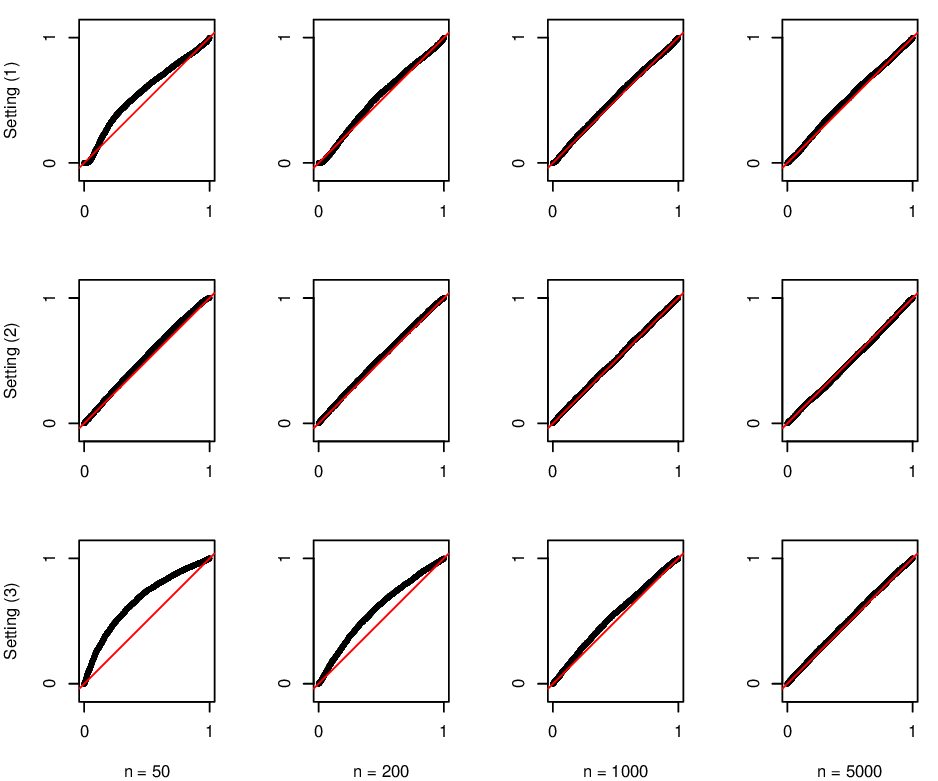}%
	\end{center}
	\caption[]{\label{fig: Simu1} \slshape\small  P-P plots for 3000 repetitions of the centred and (empirically) rescaled estimator of $\dep_1$ in the three settings from \ref{sec: GoFAsym} for increasing sample sizes (from left to right). The results are presented for an empirical correlation matrix in the case of Gaussian data.} 
\end{figure}
\begin{figure}[ht!]
	\begin{center}
			\includegraphics[width=0.99\textwidth]{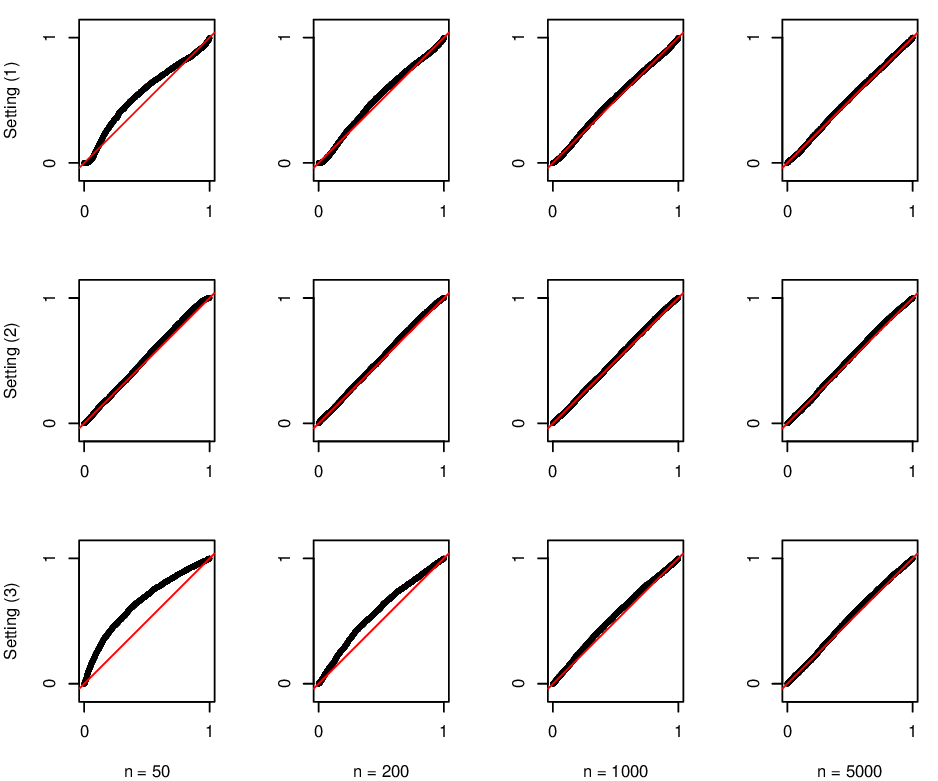}%
	\end{center}
	\caption[]{\label{fig: Simu2} \slshape\small  P-P plots for 3000 repetitions of the centred and (empirically) rescaled estimator of $\dep_2$ in the three settings from \ref{sec: GoFAsym} for increasing  sample sizes (from left to right). The results are presented for an empirical correlation matrix in the case of Gaussian data.}
\end{figure}
\subsection{Goodness-of-fit for rank-based estimation of the correlation matrix}
\label{sec: GoFAsym:Ranks}
We now repeat the simulations in the same settings as those of \ref{sec: GoFAsym} for the Gaussian copula case, that is when the estimated correlation matrix is $\check{R}_{n}$. The results for $\dep_1$ and $\dep_2$ are shown in Figures~\ref{fig: Simu3} and~\ref{fig: Simu4}, respectively.
\begin{figure}[ht!]
	\begin{center}
			\includegraphics[width=0.99\textwidth]{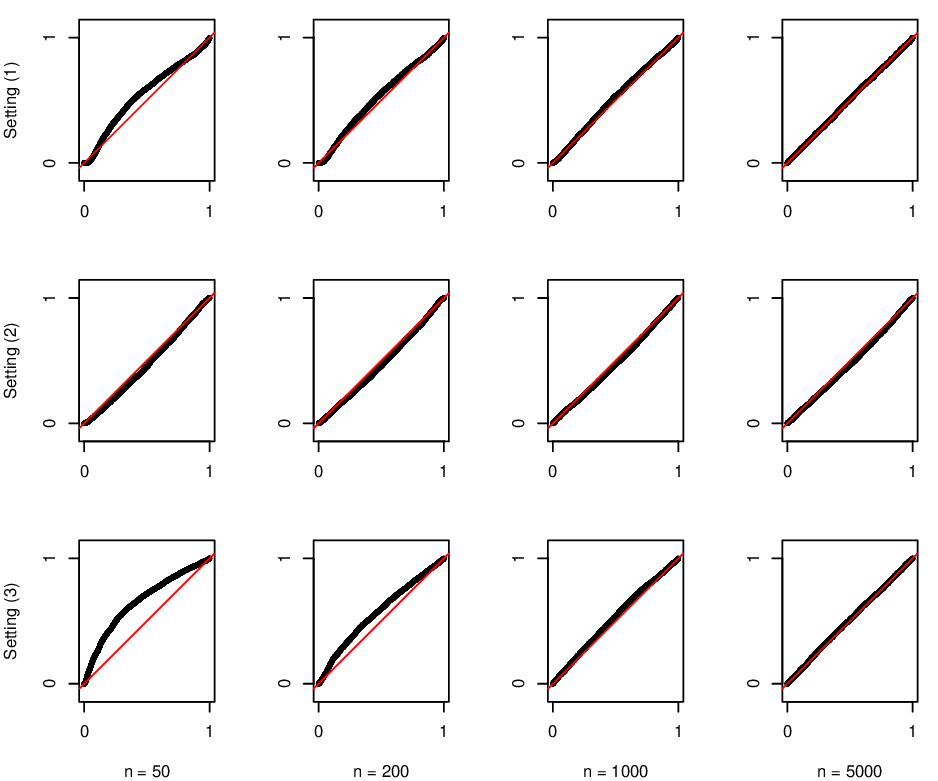}%
	\end{center}
	\caption[]{\label{fig: Simu3} \slshape\small  P-P plots for 3000 repetitions of the centred and (empirically) rescaled estimator of $\dep_2$ in the three settings from \ref{sec: GoFAsym} for increasing  sample sizes (from left to right). The results are presented for a Gaussian copula relying on $\check{R}_{n}$.}
\end{figure}
\begin{figure}[ht!]
	\begin{center}
			\includegraphics[width=0.99\textwidth]{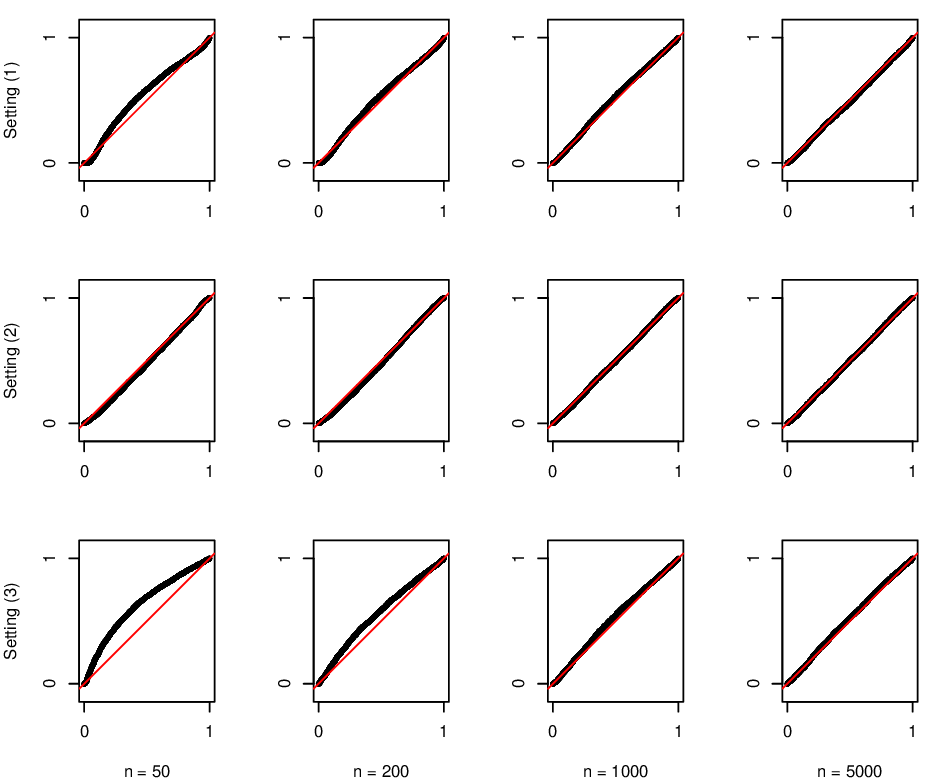}%
	\end{center}
	\caption[]{\label{fig: Simu4} \slshape\small  P-P plots for 3000 repetitions of the centred and (empirically) rescaled estimator of $\dep_2$ for the three settings from \ref{sec: GoFAsym} for increasing  sample sizes (from left to right). The results are presented for a Gaussian copula relying on $\check{R}_{n}$.}
\end{figure}

\subsection{Eigenvalue shrinkage}
\label{sec: Shrink}
The simulations in \ref{sec: GoFAsym} reveal the plug-in estimator to have a positive bias for small sample sizes. This is not surprising; it was already noted by C. Stein in the '60s and '70s that the eigenvalues of the empirical covariance matrix tend to be more spread out than their population counterparts. We refer to \citet{dey_srinivasan1985} and~\citet{donoho2018shrinkage} for references about the subject.

In the aforementioned works, new estimators of the covariance matrix were proposed.  The idea is to shrink the largest eigenvalues and increase the smaller ones to correct for the discrepancy arising. 
We follow \citet{dey_srinivasan1985}. Let $S$ be distributed according to the Wishart $W_d(\Sigma, n-1)$ distribution. The maximum likelihood estimator of the covariance matrix of the $\normal_d(\mu, \Sigma)$ distribution with unknown $\mu$ and $\Sigma$ based on an independent random sample of size $n$ has distribution $S/n$.

Let $\hat{\Sigma} = \hat{U} \hat{\Lambda} \hat{U}^\top$ where $\hat{U}$ is an orthogonal matrix and $\hat{\Lambda}$ is a diagonal matrix with elements $l_1 \ge \ldots \ge l_d$. 
\emph{Orthogonally invariant estimators} of $\Sigma$ are those of the form
\[
  \hat{\Sigma}_\ell = \hat{U} \ell( \hat{\Lambda} )  \hat{U}^\top
\]
where $\ell( \hat{\Lambda} )$ is a diagonal matrix with elements $\ell_1( \hat{\Lambda} ), \ldots, \ell_d(\hat{\Lambda})$. Many functions $\ell_j$ have been proposed that correspond to certain loss functions. The maximum likelihood estimator corresponds to $\ell_j^0( \hat{\Lambda} ) = n^{-1} l_j$. In \citet{dey_srinivasan1985}, the following choices are considered:
\begin{itemize}
\item $\ell_j^m( \hat{\Lambda} ) = d_j l_j$ (Theorem~3.1) with $d_j = 1/(n+d-2j)$ for $j=1,\ldots,d$, referred to as DS1.
\item $\ell_j^S( \hat{\Lambda} ) = d_j l_j - (l_j \log l_j) \tau(u) / (b_1 + u)$ (Theorem~3.2) where $u = \sum_{j=1}^d (\log l_j)^2$, $b_1 > 5.76 (d-2)^2/ (n+d-1)^2$ and $\tau(u)$ is a function satisfying, among others, $0 < \tau(u) < 2.4 (d-2) / (n+d-1)^2$. In their Section~4, they propose $b_1 = 5.8 (d-2)^2 / (n+d-1)$ and $\tau(u) = 1.2 (d-2) / (n+d-1)^2$. This method is referred to as  DS2.

\end{itemize}
The above shrinkage methods are based upon $\hat{\Sigma}$ sampled from $W_d(\Sigma, n)$, see \citet{dey_srinivasan1985}. Therefore, we replace $n$ by $n-1$.
These are but two choices out of a large number of shrinkage methods that depend on the loss function and the model. We refer to \citet{donoho2018shrinkage} for a survey.

In Table~\ref{tab:shrinkage}, we consider settings~(1) and~(3) from \ref{sec: GoFAsym} for sample size $n=200$. The number of replications is 3000 and the results are obtained for the empirical correlation matrix in the fully Gaussian case, that is, case~(GD) in Section~\ref{sec:estim:asy}. The entries in the table show the observed mean, median and standard deviation of the quantity $\sqrt{n} (\dep_r(\varphi(\hat{\Sigma}_\ell)) - \dep_r(R)) / \zeta_{n,r}$, where the estimator $\hat{\Sigma}_\ell$ of the covariance matrix uses one of the shrinkage functions defined above and where the estimated standard error $\zeta_{n,r}$ is based on plugging in the estimated correlation matrix $\varphi(\hat{\Sigma}_\ell)$, similar to what was done in Corollary~\ref{cor: LimitDep}.
From the results, we observe that shrinkage moves the median closer to zero in both settings while leaving the standard deviation close to one. 

There does not seem to be an important difference between DS1 and DS2. 

\begin{table}[]
	\begin{center}
	\begin{tabular}{@{}llrrrrrr@{}}
		\toprule
		\multicolumn{2}{c}{} & \multicolumn{3}{c}{$\dep_1$}                                                        & \multicolumn{3}{c}{$\dep_2$}                                                        \\ \cmidrule(l){3-5} \cmidrule(l){6-8} 
		Setting & Method                & \multicolumn{1}{c}{Mean} & \multicolumn{1}{c}{Median} & \multicolumn{1}{c}{SD} & \multicolumn{1}{c}{Mean} & \multicolumn{1}{c}{Median} & \multicolumn{1}{c}{SD} \\ \midrule
		(1) & MLE                 &      $-0.026$        &    0.118           &       1.292                 &         $-0.020$           &        0.146            &      1.282           \\
		& DS1                 &      $-0.125$        &    0.031           &       1.320                 &         $-0.116$            &        0.058            &      1.303            \\
		& DS2                 &      $-0.126$        &    0.031           &       1.320                 &         $-0.117$            &        0.058            &      1.303            \\
		\midrule
		(3) & MLE                 &      0.279       &    0.335           &       0.979                 &         0.220           &        0.261            &      0.984           \\
		& DS1                 &      0.118        &    0.174           &       0.981                 &         0.075            &        0.120            &      0.986            \\
		& DS2                 &      0.117        &    0.173           &       0.981                 &         0.074            &        0.119            &      0.986            \\
		\bottomrule \\
	\end{tabular}
	\end{center}
	\caption{\slshape\small Effect of eigenvalue shrinkage methods on the studentised estimator, $\sqrt{n} (\dep_r(\varphi(\hat{\Sigma}_\ell)) - \dep_r(R)) / \zeta_{n,r}$, at $n = 200$ in settings~(1) and~(3) from \ref{sec: GoFAsym}.}
	\label{tab:shrinkage}
\end{table}

\subsection{Coverage of confidence intervals}
\label{sec:simu:CI}

We investigate the actual coverage of the asymptotic $(1-\alpha) \times 100\%$ confidence intervals 
\[
	\bigl[\dep_r(\check{R}_n) \pm z_{1-\alpha/2} \times \zeta_{n,r}/\sqrt{n}\bigr] \cap [0, 1]
\] 
for various sample sizes, where $z_p$ is the quantile of a standard normal distribution at level~$p$. We consider settings~(1) and~(3) from \ref{sec: GoFAsym} in the Gaussian copula (GC) case, so $\check{R}_n$ and $\zeta_{n,r}$ are as in~\eqref{eq:Rncheck} and Corollary~\ref{cor: LimitDep}. 
The chosen coverage probability is 95\%. The results are presented in Tables~\ref{tab:covprob}. For each coefficient $\dep_1$ and $\dep_2$, we give the true value, the mean of the lower and upper bounds over 3000 independent replications, and, finally, the empirical coverage. We did not rely on shrinkage methods in this part. 

\begin{table}[]
	\begin{center}
	\begin{tabular}{@{}lcrrrrrrrr@{}}
		\toprule
		\multicolumn{2}{c}{} & \multicolumn{4}{c}{$\dep_1$}                                                        & \multicolumn{4}{c}{$\dep_2$}                                                        \\ \cmidrule(l){3-6} \cmidrule(l){7-10} 
		Setting & $n$               & \multicolumn{1}{c}{True} & \multicolumn{1}{c}{LB} & \multicolumn{1}{c}{UB} &  \multicolumn{1}{c}{Cov.}&  \multicolumn{1}{c}{True} & \multicolumn{1}{c}{LB} & \multicolumn{1}{c}{UB} &  \multicolumn{1}{c}{Cov.} \\ \midrule
		(1) & \phantom{00}50                 &      0.026        &    0.000           &       0.104                 &         93.8\%           &        0.025            &       0.000    &    0.098     &  92.8\%  \\
		& \phantom{0}200              &      0.026        &    0.001          &       0.057                 &         93.5\%           &        0.025            &        0.001   &    0.056     &  93.5\%   \\
		& 1000               &      0.026        &    0.014           &       0.039                 &         94.3\%            &        0.025            &        0.014    &    0.038      &  94.4\%   \\
		& 5000            &     0.026        &    0.021           &       0.032                 &         95.8\%            &        0.025            &         0.020    &    0.030      &  95.5\%  \\
		\midrule
		(3) & \phantom{00}50                 &      0.051        &    0.012           &       0.129               &         94.0\%           &        0.050            &       0.009    &    0.128     &  94.4\%  \\
		& \phantom{0}200             &      0.051        &    0.028           &       0.083                 &         94.8\%           &        0.050            &        0.027   &    0.083     &  94.9\%   \\
		& 1000             &      0.051        &    0.040           &       0.064                 &         94.6\%            &        0.050            &        0.039    &    0.064      &  94.8\%   \\
		& 5000           &     0.051        &    0.045           &       0.056                 &         95.4\%            &        0.050            &         0.045    &    0.056      &  94.3\%  \\
		\bottomrule\\
	\end{tabular}
	\end{center}
	\caption{\slshape\small Means of lower and upper bounds and actual coverage of rank-based asymptotic $95\%$ confidence intervals $[\dep_r(\check{R}_n)~\pm~z_{0.975}~\times~\zeta_{n,r}/\sqrt{n}]~\cap ~[0, 1]$ in settings~(1) and~(3) from \ref{sec: GoFAsym} over 3000 independent replications.}
	\label{tab:covprob}
\end{table}

\subsection{Shrinkage evaluation for EEG data}
\label{sec:EEG:shrink}
For the EEG case study in \ref{sec:EEG}, we conducted a preliminary assessment to evaluate whether the shrinkage methods in \ref{sec: Shrink} produce confidence intervals performing as they should.
The sample size and parameter values were taken to match those of the data.
The empirical coverage of the confidence intervals was estimated based on 2000 replications.
The results are presented in Figure~\ref{fig: PreTest}. In plots~(a) and~(c), the advantage of shrinking the eigenvalues is clearly visible for coefficient~$\dep_1$.

\begin{figure}
\begin{center}
\begin{tabular}{@{}c@{}c}
\includegraphics[width=0.5\textwidth, height=75mm]{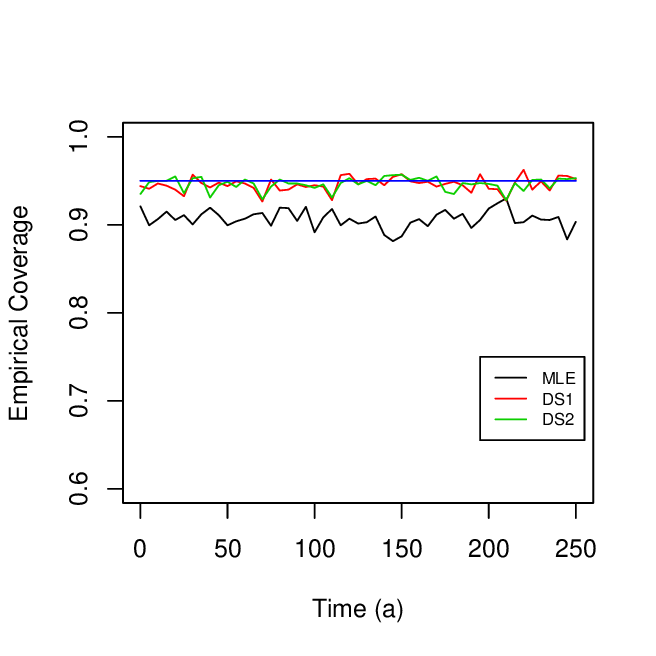}&
\includegraphics[width=0.5\textwidth, height=75mm]{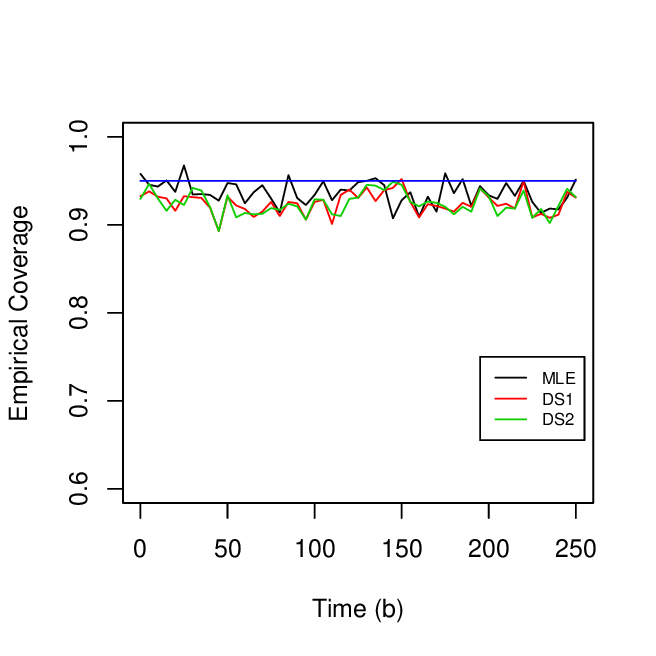}\\ 
\includegraphics[width=0.5\textwidth, height=75mm]{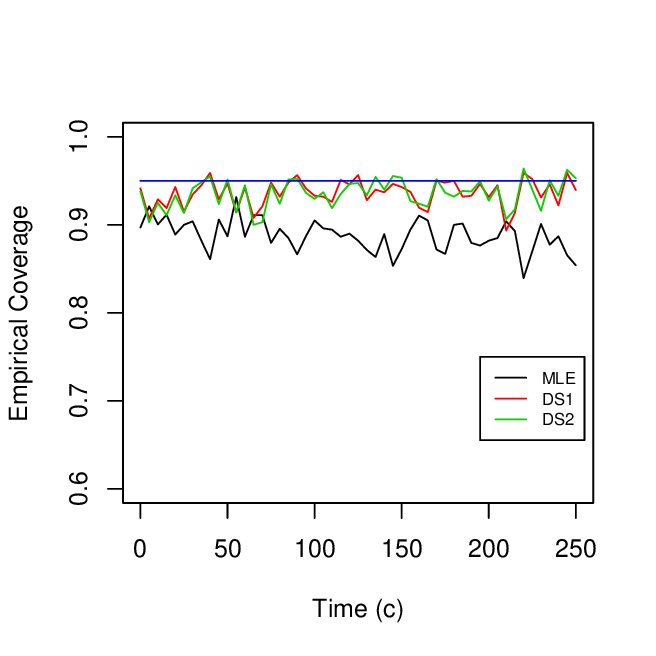}&
\includegraphics[width=0.5\textwidth, height=75mm]{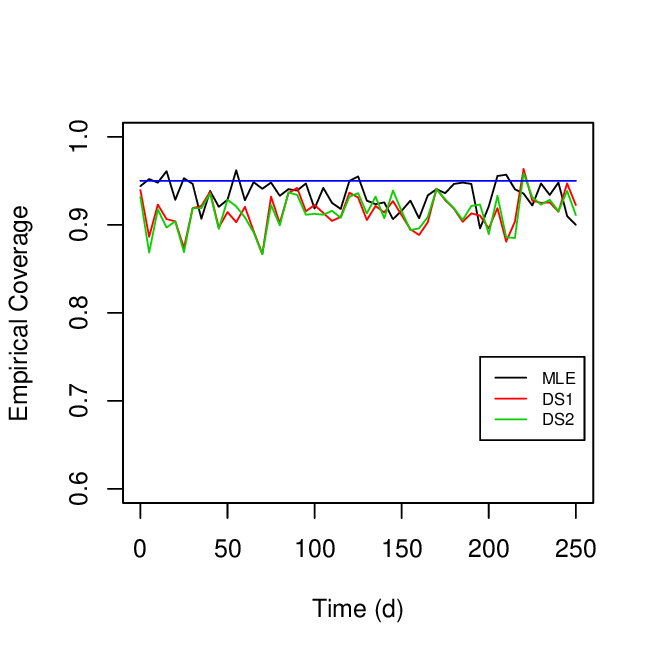}
\end{tabular}
\end{center}
\caption{\label{fig: PreTest} \slshape\small  Empirical coverage of 95\% confidence intervals estimated from 2000 replications in the Gaussian copula setting with sample size and parameters derived from the case study in \ref{sec:EEG}. MLE refers to no shrinkage while DS1 and DS2 refer to the two shrinkage methods in \ref{sec: Shrink}. (a) Alcoholic group with $\dep_1$. (b) Alcoholic group with $\dep_2$. (c) Control group with $\dep_1$. (d) Control group with $\dep_2$.}
\end{figure}

%% file: paper-EEG.tex
We now turn to an application on real data exhibiting a possible use of the new dependence coefficients.
We consider the electroencephalogram (EEG) dataset gathered by Henri Begleiter\footnote{At the Neurodynamics Laboratory at the State University of New York Health Center at Brooklyn.} and first analysed in \citet{zhang1995event}. Data are available for two types of patients: those suffering from alcoholism and a control group. The dataset consists of 120 trials for 122 subjects and is available on the UCI Machine Learning Archive \citep{Dua:2019}.

An EEG measures the electric activity of the brain and thus helps to understand its functioning. In the dataset we consider, the data are gathered through 64 electrodes placed on the patient's scalp.\footnote{The position of the electrodes follows the Standard Electrode Position Nomenclature put forward by the American Electroencephalographic Association in 1990.}
The electrical activity for each electrode is measured in $\mu V$ through time. Each patient is exposed to a visual stimulus during a one-second timespan during which 256 measurements are collected. The 120 trials are divided into three types of stimuli tested: a single visual stimulus, two stimuli where the second one matches the first one and two stimuli where the second one does not match the first one.  In each trial a different picture or different sets of pictures are used.

This dataset was recently analysed in \citet{solea2020copula} and \citet{anuragi2020empirical}. In this first paper, the dependence structure is modelled under a Gaussian copula assumption, which has become classical since the seminal work of \citet{liu2009nonparanormal}. Even though the Gaussian copula hypothesis may seem restrictive, it turned out quite successful and is well accepted in the field, as stressed in \citet{solea2020copula}. In the sequel, we also make the assumption that the copula is Gaussian and thus use the rank-based estimator $\dep_r( \check{R}_{n,r} )$ with $\check{R}_{n,r}$ the matrix of normal scores rank correlation coefficients in~\eqref{eq:Rncheck}.

The graphs in \citet[p.~11]{solea2020copula} present the results of different estimation procedures for the dependence graph. A visual inspection shows that the connectivity networks estimated by the different methods largely differ from one estimation procedure to another. These discrepancies motivate our analysis of the dependence between the pre-frontal (FP) and the anterio-frontal (AF) electrodes, as the methods seem to estimate different network structures for these particular blocks.  The AF region consists of the electrodes AF1, AF2, AF7, AF8 and AFZ while the FP region consists of the FP1, FP2 and FPZ electrodes. In our notation we are thus seeking to quantify dependence between a group of $p = 3$ variables and another one with $q = 5$ variables.

We chose to focus on trial No.~26. This choice is purely random and was made prior to the analysis. The only check that was made concerns the number of patients in the trial. Indeed, even though the experiment was carried out on 122 patients, certain results are missing. For the trial selected, the data for 99 patients were available. Among these 99 patients, 60 were alcoholic. Preliminary Monte-Carlo simulations evaluating the coverage probabilities of estimated confidence intervals---reported in \ref{sec:EEG:shrink}---suggested the use of the shrinkage estimator DS1 (\ref{sec: Shrink}) of the correlation matrix which is then standardised again via the square roots of the diagonal elements. This finding is purely empirical and theoretical justifications for this or other shrinkage methods in the context of the matrix of normal scores rank correlation coefficients are yet to be developed.

In the top row of Figure~\ref{fig:EEG},
we show estimates of various dependence coefficients for the two groups of patients. The coefficients are estimated at one out of five time instants to avoid overloading the graphs. To enable a proper comparison, the $\RV$ and $\overline{\RV}$ are computed on the same, shrinked matrix as the $\dep_r$ coefficients.  The interest of correcting the RV coefficient as in Remark~\ref{rem:RVadj} is clear. The various coefficients exhibit quite similar profiles over time. Interestingly, the curve of the square of the adjusted RV coefficient (not shown) would be close to $\dep_1$ and $\dep_2$.


\begin{figure}
	\begin{center}
		\begin{tabular}{@{}c@{}c}
			\includegraphics[width=0.46\textwidth, height=60mm]{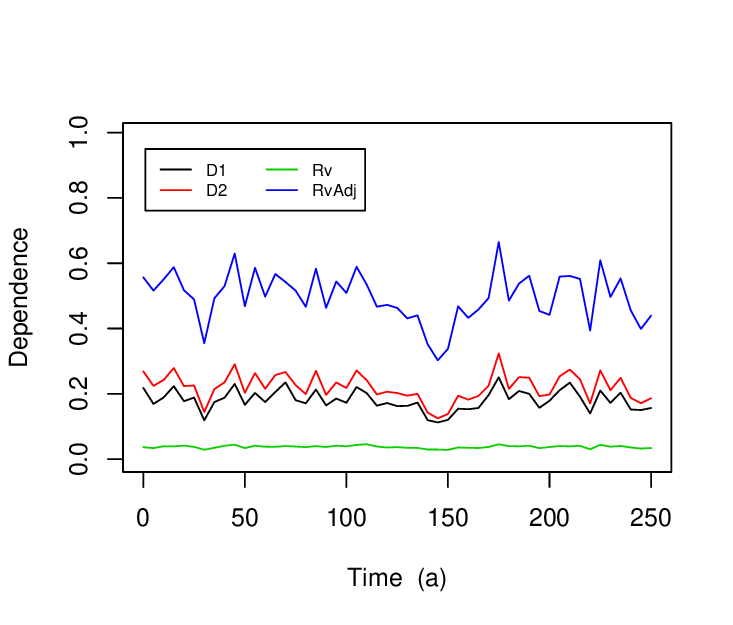}&%
			\includegraphics[width=0.46\textwidth, height=60mm]{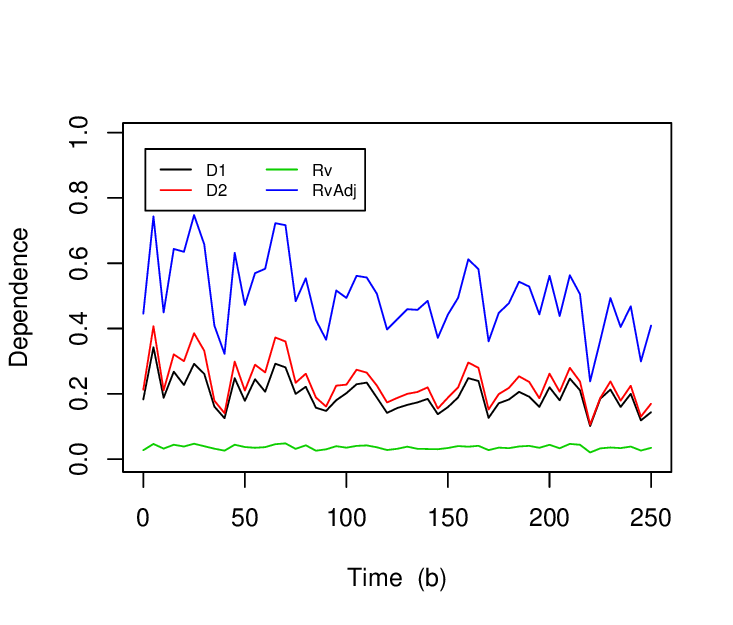}\\
		\end{tabular}
		\begin{tabular}{@{}c@{}c}
			\includegraphics[width=0.46\textwidth, height=60mm]{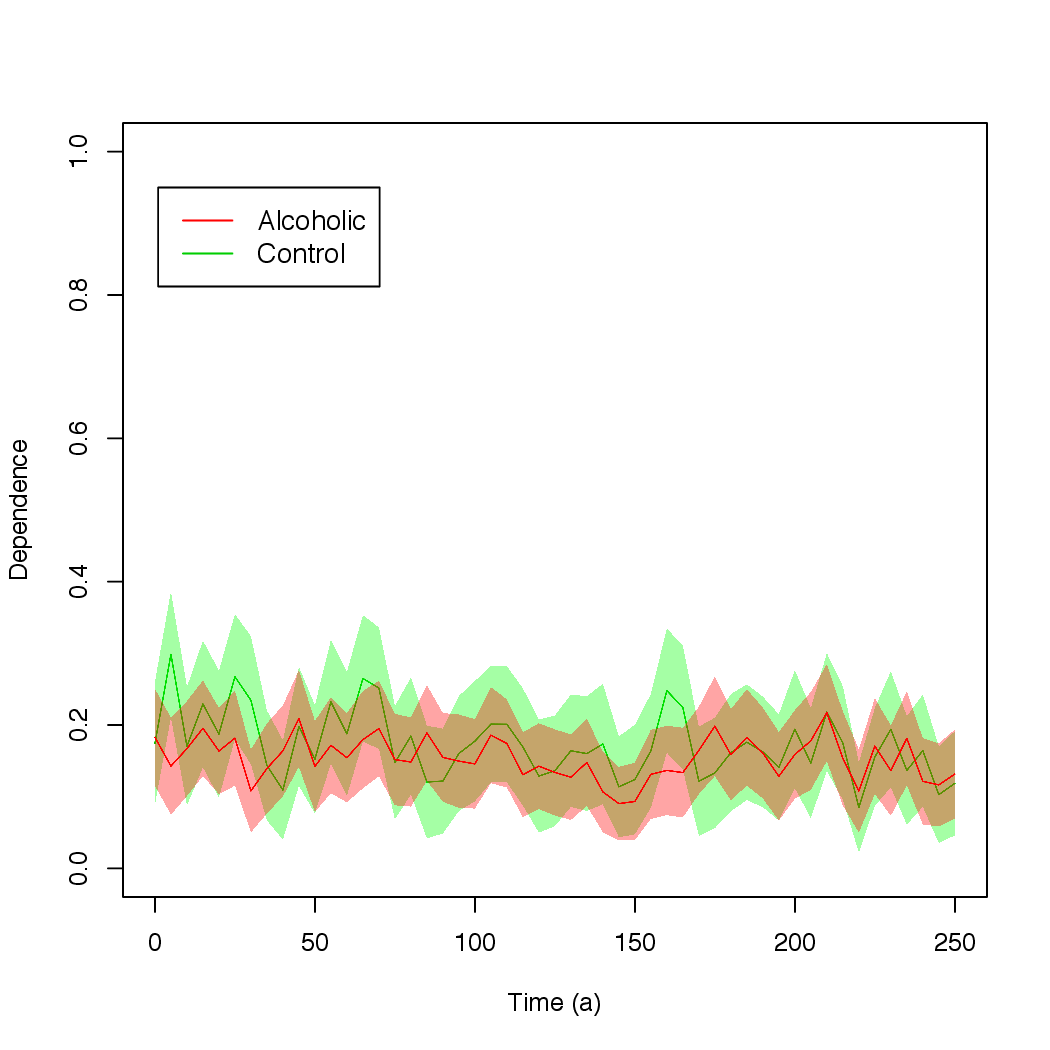}&%
			\includegraphics[width=0.46\textwidth, height=60mm]{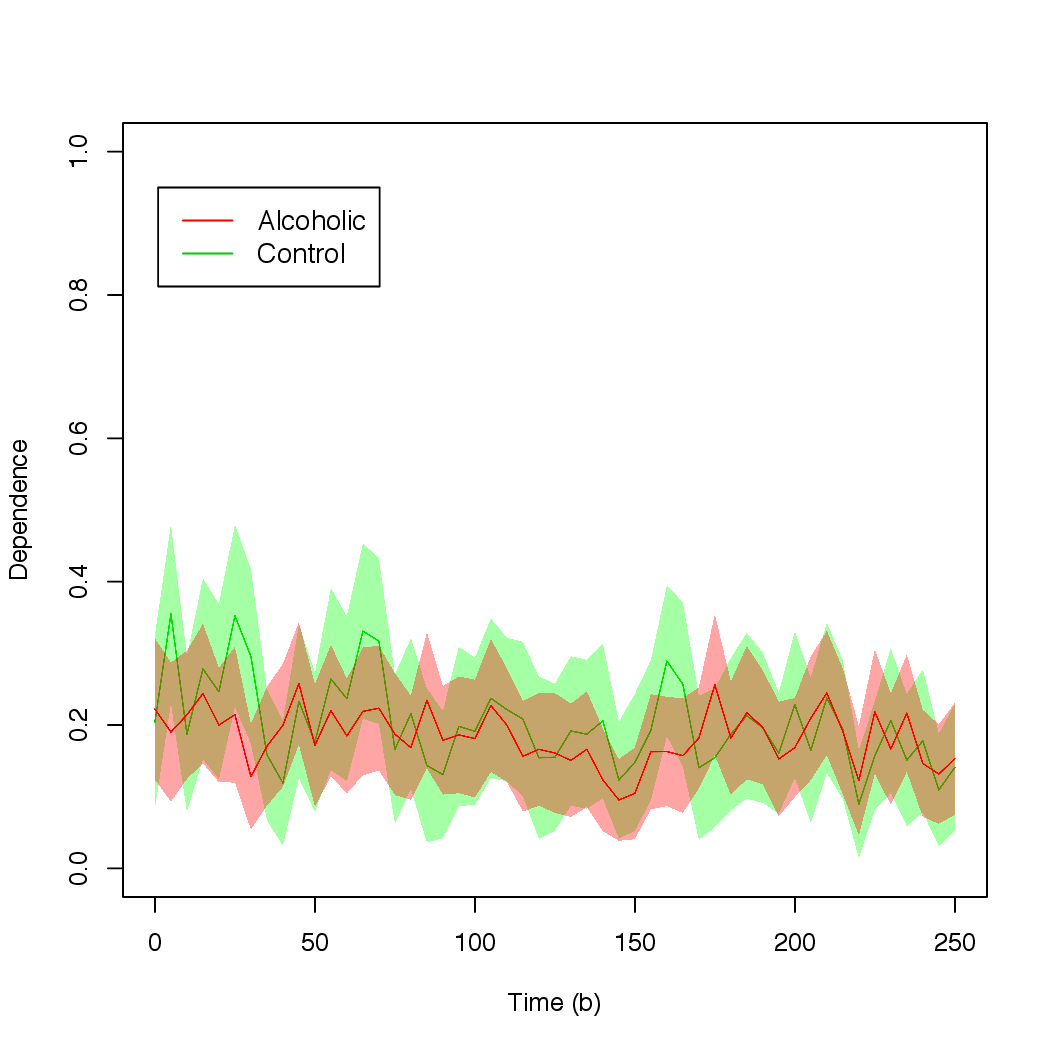}\\
		\end{tabular}
		\begin{tabular}{@{}c@{}c}
			\includegraphics[width=0.46\textwidth, height=60mm]{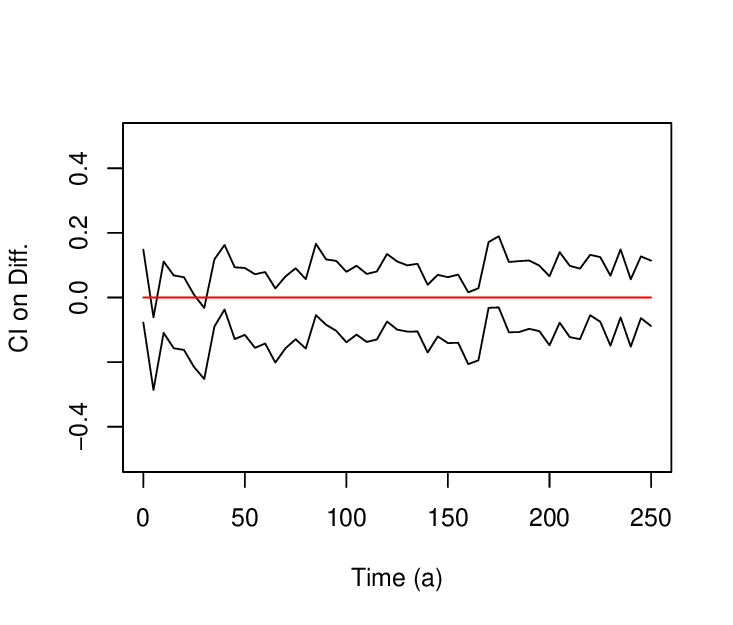}&%
			\includegraphics[width=0.46\textwidth, height=60mm]{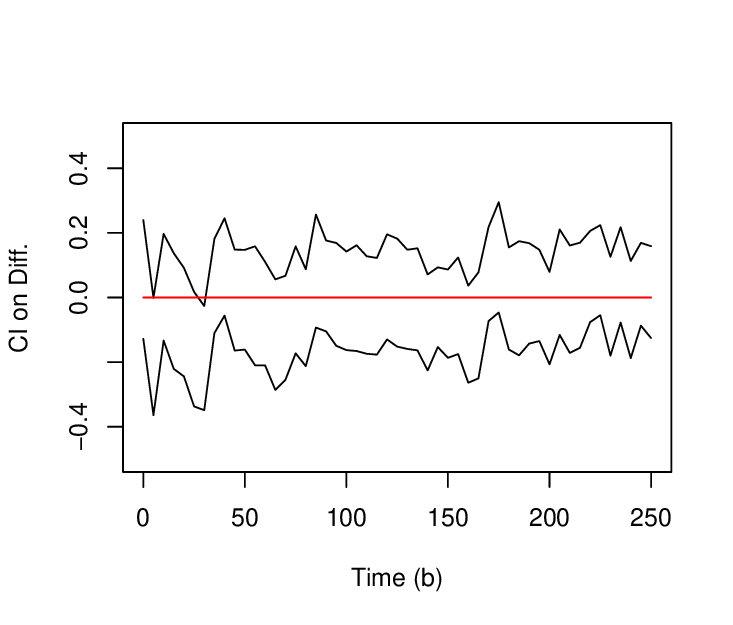}\\
		\end{tabular}
	\end{center}
	\caption[]{\slshape\small Top: Dependence across time for alcoholics (a) and control patients (b). Middle: Comparison of patients suffering from alcoholism versus control group using dependence coefficients $\dep_1$ (a) and $\dep_2$ (b). Estimates as solid lines and point-wise 95\% confidence bands as coloured shaded areas.
	Bottom: Asymptotic 95\% confidence intervals for the difference $\dep_r^{\Ctr} - \dep_r^{\Alc}$ between the two groups of patients for $\dep_1$ (a) and $\dep_2$ (b).}
	\label{fig:EEG}
\end{figure}

In the middle row of Figure~\ref{fig:EEG},
we compare the coefficients $\dep_1$ and $\dep_2$ for both types of patients and provide pointwise confidence bands.
The latter are formed out of a confidence interval at each time instant, are based on the estimated asymptotic variance and are chosen to have a 95\% coverage probability.

Assuming independence between alcoholics and control patients, an asymptotic two-sided $(1-\alpha)$ confidence interval for the difference $\dep_r^{\Ctr} - \dep_r^{\Alc}$ is
	\[
		\hat{\dep}_{r}^{\Ctr} - \hat{\dep}_{r}^{\Alc}
		\pm
		z_{1-\alpha/2} \sqrt{ \frac{(\hat{\zeta}_r^{\Ctr})^2}{n^{\Ctr}} + \frac{(\hat{\zeta}_r^{\Alc})^2}{n^{\Alc}} }
	\] 
	 with $n^{\Ctr} = 99-60 = 39$ and $n^{\Alc} = 60$ and with $z$ the standard normal quantile. We present the confidence intervals corresponding to the difference above in the bottom row of Figure~\ref{fig:EEG}.
From the data one cannot conclude that the two groups of patients have different dependence coefficients between the AF and FP regions. Still, it seems that the dependence between the two regions under study is higher for the control group than for the alcoholics. The variability of the data is too high to reject the null hypothesis of no difference, but complementary analyses with higher sample sizes might help settle the case. Also, a slight downward trend seems to be present for control patients; see Figure~\ref{fig:EEG}, top row, panel~(b). Time-varying modelling of dependence could thus also constitute a future research path.  


